\documentclass[11pt,a4paper]{amsart}
\usepackage{amsmath,amssymb,amsthm,amsfonts}
\usepackage{mathtools}
\usepackage{mathrsfs}
\usepackage{amscd}
\usepackage{tikz-cd}
\usepackage{listings}
\usepackage{paralist}
\usepackage{graphicx}
\usepackage{booktabs}
\usepackage{latexsym}
\usepackage[shortlabels]{enumitem}
\usepackage{thmtools}
\usepackage{url}
\usepackage[all,cmtip]{xy}
\usepackage[bookmarks=false]{hyperref} 

\newtheorem{theorem}{Theorem}[section]
\newtheorem{example}[theorem]{Example}

\newtheorem{remark}[theorem]{Remark}
\newtheorem{cor}[theorem]{Corollary}
\newenvironment{corollary}{\begin{cor} \em}{\end{cor}}
\newtheorem{tont}[theorem]{Definition}
\newenvironment{definition}{\begin{tont} \em}{\end{tont}}
\newtheorem{lemma}[theorem]{Lemma}

\newtheorem{proposition}[theorem]{Proposition}
 \DeclareMathOperator\codim{codim}

\begin{document}

\title [Local volumes, equisingularity, and generalized smoothability]{Local volumes, equisingularity and generalized smoothability}
\author{Antoni Rangachev}

\begin{abstract} We introduce the {\it restricted local volume} of a relatively very ample invertible sheaf as an invariant of equisingularity by determining its change across families. We apply this result to give numerical control of Whitney--Thom (differential) equisingularity for families of isolated complex analytic singularities. The characterization of the vanishing of the local volume gives rise to the class of {\it deficient conormal} (dc) singularities. We introduce a notion of {\it generalized smoothability} by considering the class of singularities that deform to dc singularities.  Using Whitney stratifications and the functoriality properties of conormal spaces we show that fibers of conormal spaces
are well-behaved under transverse maps. Then by Thom's transversality, the structure theorems of Hilbert--Burch and Buchsbaum--Eisenbud, we show that all smoothable singularities of dimension at least $2$, Cohen--Macaulay codimension $2$, Gorenstein codimension $3$, and more generally determinantal and Pfaffian singularities deform to dc singularities. 

\end{abstract}

\address{Department of Mathematics\\
  University of Chicago\\
 Chicago, IL 60637 \\
  Institute of Mathematics \\
and Informatics, Bulgarian Academy of Sciences\\
Akad. G. Bonchev, Sofia 1113, Bulgaria}
\maketitle
\tableofcontents

\section{Introduction}

A principal goal of equisingularity theory is to decide if two germs of sets or maps look alike in some sense. In general, this is a hard problem, but if the germs are part of a family, then it is somewhat easier to predict when the members of the family are the ``same'' \cite{Zariski}. The conditions that will guarantee this ``similiarity'' depend on the total space of the family. Nevertheless, we would like to control these conditions by fiberwise dependent numerical invariants. 

The main notion of equsingularity theory that will be addressed in this work is that of {\it Whitney--Thom equisingularity}, also known as {\it differential equisingularity}. We will use its algebro-geometric formulation and interpret the problem of finding numerical control for it as a problem of intersection theory. Variations of this problem appear in resolutions of singularities, equiresolutions, numerical control of flatness and existence of simultaneous canonical models. 

For the most part we will work in a fairly general setup: the schemes considered are of finite type over an arbitrary field. The applications of our results though will be in the complex analytic category via a classical translation.

Let’s fix some notation. Let $X$  and $Y$ be affine reduced schemes of finite type over a field $\Bbbk$. Assume $X$ is equidimensional and $Y$ is regular and integral of dimension one. Suppose $h \colon X \rightarrow Y$ is a morphism with equidimensional fibers of positive dimension. All points considered are assumed to be closed unless stated otherwise. Let $S$ be a subscheme of $X$ that is proper over $Y$  such that $S_y$ is nowhere dense in $X_y$ for each point $y \in Y$.  Let $C$  be an equidimensional reduced scheme, projective over $X$ such that the structure morphism $c \colon C \rightarrow X$  maps each irreducible component of $C$ to an irreducible component of $X$. Set $D:= c^{-1}S$ and $\dim C:= r+1$.

Fix a point $y_0$ in $Y$. Denote by $D_{\mathrm{vert}}$ the union of components of $D$ of top dimension $r$ that map to $y_0$ under $h \circ c$. Our goal is to control numerically the presence of $D_{\mathrm{vert}}$. 

In the applications to equisingularity, $S$ is a subscheme supported over the singular locus of a complex analytic variety $X$, $Y$ is a smooth subspace of $X$, and $X$ is viewed as the total space of a family obtained by a transverse retraction from $X$ to $Y$. The scheme $C$ is a conormal modification of $X$. Then the condition for $X$ to be equisingular along $Y$ in a neighborhood of  $y_0$ is expressed by asking $D_{\mathrm{vert}}$ to be empty. 

The first case where one can use basic intersection theory to control numerically the presence of $D_{\mathrm{vert}}$ is when $C:= \mathrm{Bl}_{S}X$. In this case $D$ is the exceptional divisor of the blowup. For each $y$ denote by $C(y)$ the blowup of $X_y$ by $S_y$ and by $D(y)$ the exceptional divisor. Set $l:=c_{1}\mathcal{O}_{C}(1)$ and $l_y:=c_{1}\mathcal{O}_{C(y)}(1)$. Let $U$ be a small enough affine neighborhood of $y_0$. Then we have the following {\it Excess--Degree Formula} (EDF):

$$ \int_{S_{y_0}} l_{y_0}^{r-1}[D(y_0)] - 
\int_{S_{y}}  l_{y}^{r-1}[D(y)] = \int_{S_{y_0}} l^{r}[D_{\mathrm{vert}}]$$
for $y \in U-\{y_0\}$. Suppose $\mathcal{O}_{C}(1)$ is ample on $D_{\mathrm{vert}}$. Then $D_{\mathrm{vert}}$ is empty if and only if the intersection number $\int_{S_{y}}  l_{y}^{r-1}[D(y)]$ is constant for $y \in Y$. A typical situation when $\mathcal{O}_{C}(1)$ is ample on $D_{\mathrm{vert}}$ is when $S$ is finite over $Y$, for example. 

What happens if $D$ is more generally a Weil divisor in $C$, and if $C$ is not birational to $X$? If $\dim X=2$ we show that the EDF still applies. In general, to obtain a formula similar to the EDF, one needs to introduce a volume-type invariant that generalizes the top self-intersection number of a Cartier divisor. 
The {\it volume of a line bundle} is an invariant studied extensively by birational geometers (cf.\ \cite[Chp.\ 2]{Laz} for definitions, basic properties and constructions). Its local counterpart the {\it local volume} was introduced by Fulger \cite{Fulger}.  The local volume's  algebraic analogue was studied by \cite{Validashti}.  

From now on assume $S$ is finite over $Y$ and for simplicity of the exposition assume that $\Bbbk$ is the residue field of each closed point $y$ and the points in $S_{y}$.
Let $\mathcal{L}$ be an invertible very ample sheaf on $C$ relative to $X$. Let $\mathcal{A}:= \oplus_{n \geq 0} \Gamma (C, \mathcal{L}^{\otimes n})$ be the ring of sections of $\mathcal{L}$. Denote by $\mathcal{A}_n$ the $n$th graded piece of $\mathcal{A}$. Assume $\dim X_y \geq 2$. Inspired by \cite{Fulger}, \cite{Validashti} and \cite{ELMNP}, for each point $y \in Y$ define the  {\it restricted local volume} of $\mathcal{L}$ at $S_y$ as

$$
\mathrm{vol}_{C_y}(\mathcal{L}):= \limsup_{n\to \infty} \frac{r!}{n^r} \dim_{\Bbbk}H_{S}^{1}(\mathcal{A}_n)\otimes_{\mathcal{O}_Y} k(y).
$$



The existence of the volumes as a limit has been a topic of extensive research (see \cite{Cut}  and Kaveh and Khovanskii \cite{KK}). The local volume might be an irrational number as shown by Cutkosky, Herzog and Srinivasan \cite{Irr.}.

One of the main results of this paper is the following relation which we refer simply as the {\it Local Volume Formula} (LVF): 

$$\mathrm{vol}_{C_{y_0}}(\mathcal{L})- \mathrm{vol}_{C_{y}}(\mathcal{L}) =  \int_{S_{y_0}} l^{r}[D_{\mathrm{vert}}].$$


Particular instances of the LVF and the EDF were known before as results of Hironaka, Schickhoff (see Rmk.\ 2.6 in \cite{Lip} and \cite{Hironaka} for a related result) and Teissier \cite{Teissier} for the Hilbert--Samuel multiplicity, Kleiman and Gaffney \cite{GK}, and Gaffney (\cite{GaffP} and \cite{Gaf3}) for the Buchsbaum--Rim multiplicity. For related work in the projective setting see \cite{Kol} and \cite[Sct.\ 4]{HMX18}. 

A direct consequence of the proof of the LVF is the following result: $D$ is flat over $Y$ if and only if $\dim_{\Bbbk}H_{S}^{1}(\mathcal{A}_n)\otimes_{\mathcal{O}_Y} k(y)$ is constant for $n \gg 0$. This is a local counterpart to the classical result of Hironaka that says that a family of projective schemes over an integral base is flat if and only if the Hilbert polynomials of the fibers remain the same (see \cite{Hironaka3} or \cite[Thm.\ 9.9, Chp.\ III]{Hartshorne}). An important local analogue of Hironaka's flatness result appears in \cite{Resolution} in regard to what Hironaka calls { \it normal flatness}, which is the flatness of the normal cone of a smooth subvariety in a smooth ambient variety. Hironaka controls normal flatness using the Hilbert--Samuel function. His result can be derived from the proof of the LVF. 

How do we control numerically the presence of $D_{\mathrm{vert}}$ when $Y$ is integral of arbitrary dimension? In the case of the EDF assuming $Y$ is regular, we reduce to the case $\dim Y=1$. In the setup of the LVF we can do this provided that the restricted local volume is constant over Zariski open dense subset of $Y$. From now on assume $\Bbbk = \mathbb{C}$. View $X \rightarrow Y$ as a part of a larger family $\mathcal{X} \rightarrow W$, where $W$ is irreducible and $Y \subset W$. Assume that $S$ and $C$ are defined in the same way for $\mathcal{X}$. We say that $\mathrm{vol}_{C_w}(\mathcal{L})$ is {\it stable}  if it is constant for $w$ in a Zariski open dense subset of $W$. In this case we say that $W$ is a {\it good base space}.

Following \cite{GR} we show that we can control the presence of vertical components of $D$ for $X \rightarrow Y$ by computing the restricted local volumes from generic one-parameter families connecting $X_{y_0}$ and $X_y$
to fibers $\mathcal{X}_w$ for which the local volume stabilizes. This gives an extension of the LVF to good base spaces of arbitrary dimension. 

The extension of the LVF to the complex analytic setting is obtained in a standard fashion using compact Stein neighborhoods. The main application of the LVF in this work is obtaining numerical control for {\it Whitney--Thom (differential) equisingularity}. Let's review briefly its definition and some of its applications. By a classical result of Whitney every complex analytic variety $V$  can be partitioned into a locally finite family of submanifolds, called {\it strata}, so that any pair of incident strata gives rise to a family of singularities obtained as the fibers of a retraction to the lower dimensional stratum and such that the total space of the family satisfies certain geometric compatibility conditions. 

More precisely, let $X \subset V$ be a complex analytic variety and $Y$ a smooth subvariety of $X$ of dimension $k$ such that $X-Y$ is smooth and $(X-Y,Y)$ is a pair of strata. Embed $X$ in $\mathbb{C}^{n+k}$ in such a way that so that $Y$ contains $0$ and $(Y,0)$ is linear subspace of dimension $k$. We say $H$ is a tangent hyperplane at $x \in X-Y$ if $H$ is a hyperplane in $\mathbb{C}^{n+k}$ that contains the tangent space $T_{x}X$. Let $(x_i)$ be a sequence of points from $X-Y$ and $(y_i)$ be a sequence of points from $Y$ both converging to $0$. Suppose that the sequence of secants $(\overline{x_{i}y_{i}})$ has limit $l$ and the sequence of tangent hyperplanes $\{T_{x_i}X\}$ has limit $T$. We say that $X$ is {\it Whitney equisingular} along $Y$ at $0$, or that the pair of strata $(X-Y,Y)$ satisfies Whitney conditions at $0$, if $l \subset T$. 

The Whitney stratifications play an important role in the classification of differentiable maps (\cite{Mather1}, \cite{Mather2}, and \cite{Gaf0}), in $D$-module theory and the solution of the Riemann--Hilbert problem \cite{LM}, and in the Goresky--McPherson theory of intersection homology \cite{GM}, to mention few.

Let $f$ be a function on $X$ of constant rank off $Y$. The relative form of Whitney conditions, called the $W_f$ condition, is defined in the same way as the Whitney condition for the pair $(X-Y,Y)$ by replacing the tangent hyperplanes that contain $T_{x_i}X$ by tangent hyperplanes that contain the tangent space
$T_{x_i}f^{-1}fx_{i}$ to the level surface $f^{-1}fx_{i}$, where each $x_i$ is a smooth point of $f^{-1}fx_{i}$. 

Denote by $X_y$ the fiber of a transverse projection to $Y$ and set $f_y:=f|X_y$. Then the Thom-Mather second isotopy lemma yields the following result: if $W_f$ holds, then there exists an ambient homeomorphism $q$ from $X_0 \times Y$ onto $X$ such that $fq = f_0 \times 1_Y$. Hence, for all $y \in Y$ close enough to $0$ the pairs $X_y,f_y$ are topologically the same. 

Our goal is to find numerical invariants depending on the fibers $X_y,f_y$ and use the LVF to show that their constancy across $Y$ is equivalent to $W_f$. As $W_f$ involves limits of tangent hyperplanes and secant lines, it's natural to expect that the invariants should be defined in terms of conormal spaces. Define the {\it conormal space} $C(X,f)$ (see Sct.\ \ref{Whitney, Jacobian}) relative to $f$ as the closure in $X \times \check{\mathbb{P}}^{n+k-1}$ of the set of pairs $(x,H)$ such that $x$ is a smooth point of the level set $f^{-1}fx$ and $H$ is a tangent hyperplane to $f^{-1}fx$. Let $c_{X,f} \colon C(X,f) \rightarrow X$ be structure morphism. Consider the blowup 
of $C(X,f)$ with center $c_{X,f}^{-1}(Y)$. Teissier shows that $W_f$ holds at $0$ if and only if the exceptional divisor of the blowup does not have vertical components of top dimension. To control the presence of vertical components with invariants depending on the fibers $X_y,f_y$ one needs to replace $C(X,f)$ with its relative version. Define the {\it relative conormal space} $C_{\mathrm{rel}}(X,f)$ of $X$ relative to $h: X \rightarrow Y$ in the same way but this time requiring that $H$ contains a parallel to $Y$. This replacement is achieved by Teissier's Principle of Specialization of Integral Dependence.

Define the conormal space $C_{\mathrm{rel}}(X)$ of $X$ relative to $h: X \rightarrow Y$ to be the closure in $X \times \check{\mathbb{P}}^{n+k-1}$ of the set of pairs $(x,H)$ where $x$ is a smooth point of $X$ and $H$ is a tangent hyperplane at $x$ containing a parallel to $Y$.  Denote by $J_{\mathrm{rel}}(X)$ the {\it Jacobian module}, and by $J_{\mathrm{rel}}(X,f)$ the {\it Jacobian module} of $X$ and $f$ with respect to the fiber coordinates (see Sct.\ \ref{Whitney, Jacobian}). Both modules are contained in a free module $\mathcal{F}$. Then $C_{\mathrm{rel}}(X,f) = \mathrm{Projan}(\mathcal{R}(J_{\mathrm{rel}}(X,f)))$ where $\mathcal{R}(J_{\mathrm{rel}}(X,f))$ is the Rees algebra of  $J_{\mathrm{rel}}(X,f)$. 
Denote by $C$ the blowup of $C_{\mathrm{rel}}(X,f)$ with center the inverse image of $Y$. Then  $C:=\mathrm{Projan}(\mathcal{R}(m_{Y}J_{\mathrm{rel}}(X,f)))$ where $m_Y$ is the ideal of $Y$ in $\mathcal{O}_{X,0}$.



Assume $X_y$ and $f_y$ have isolated singularities at $y$. Let $m_y$ be the ideal of $y$ in $\mathcal{O}_{X_{y},y}$. For $C$ denote by $\varepsilon m(y)$ the restricted local volume of $\mathcal{O}_{C}(1)$ at $y$. For $C_{\mathrm{rel}}(X)$ denote by $\varepsilon (y)$ the restricted local volume of $\mathcal{O}_{C_{\mathrm{rel}}(X)}(1)$ at $y$.
These notations align with the closely related notion of {\it $\varepsilon$-multiplicity of modules} (see \cite{Validashti}): $\varepsilon m(y)$ and $\varepsilon (y)$ are the restricted versions of the $\varepsilon$-multiplicity of $m_{y}J(X_y,f_y)$ and $J(X_y,f_y)$ where $m_y$ is the ideal of $y$ in $\mathcal{O}_{X_y}$ and $J(X_y,f_y)$ is the  Jacobian module of $X_y,f_y$.



Assume $X \rightarrow Y$ can be read from  the family $\mathcal{X} \rightarrow W$, where $W$ is a component of the miniversal base space of $X_0$ and $\mathcal{X}$ is the total deformation space of $X_0$ over $W$. Assume $\varepsilon (w)$ is stable for generic $w \in W$ and generic unfolding $\tilde{f}$ of $f$. By generic specialization, the stability condition means that $\varepsilon (w)$, which in this case is the $\varepsilon$-multiplicity $\varepsilon(J(\mathcal{X}_w,\tilde{f}_w))$ of the Jacobian module $J(\mathcal{X}_w,\tilde{f}_w)$, is constant for generic $w$. As an application of the LVF we show that $(X-Y,Y)$ satisfies $W_f$ if and only if $\varepsilon m (y)$
is contanst along $Y$. Thus owing to Grauert's theorem \cite{Grauert}, in the case of volume stability, to each isolated singularity $X_0$, we can associate finitely many invariants, each associated with an irreducible component of the base space of miniversal deformations of $X_0$, which allow to recognize all equisingular deformations of $X_0$.



We prove a similar statement for Thom's $A_f$ condition, which is a relative stratification condition for the study of functions and mappings on stratified sets. It plays an important role in Thom’s second isotopy theorem, and provides a transversality condition in the development of the Milnor fibration.

Our result for $W_f$ was first proved for isolated hypersurface singularities by Teissier \cite{Teissier2} using the Hilbert-Samuel multiplicity. Gaffney (\cite{Gaf1} and \cite{Gaf2} based on ideas conceived in \cite{Gaf0}) and Gaffney and Kleiman \cite{GK} treated the case of isolated complete-intersection singularities using the Buchsbaum--Rim (BR) multiplicity. More recently, Gaffney and Rangachev \cite{GR} adressed the case of families of isolated determinantal singularities using Gaffney's Multiplicity-Polar Theorem for the relative BR multiplicity. In all this cases the base space $W$ of miniversal deformations of $X_0$ is smooth, $\mathcal{X}_w$ is smooth for generic $w$, and $\varepsilon (w)=0$ which as we show in Rmk.\ \ref{stability of f} is equivalent to the vanishing of the $\varepsilon$-multiplicity $\varepsilon (J(\mathcal{X}_w))$ where $J(\mathcal{X}_w)$ of the Jacobian module of $\mathcal{X}_w$. 

Thus it is natural to consider the class of isolated singularities $X_0$ for which $\varepsilon (w)=0$ for generic $w$, which is the simplest instance of volume stability. Assume $\mathcal{X}_w \subset \mathbb{C}^n$. 
Let $c(w) \colon C(\mathcal{X}_w) \rightarrow \mathcal{X}_w$ be the structure morphism for the conormal space of $\mathcal{X}_w$ in $\mathbb{C}^n$ and let $S_w$ be the singular locus of $\mathcal{X}_w$. We show that 
$$\varepsilon (w)=0 \  \text{iff} \  \codim (c(w)^{-1}S_w, C(\mathcal{X}_w)) \geq 2 \ \ \text{iff} \  X_w \  \text{has no polar curve}.$$

In other words the volume vanishes if and only if the fibers of  $C(\mathcal{X}_w)$ over the singular points of $\mathcal{X}_w$ are of dimension less than expected. We call such $\mathcal{X}_w$ {\it deficient conormal (dc) singularity}. We show that the dc property is independent of the embedding of $\mathcal{X}_w$ in affine space and stable under deformations. Furthermore, if $\mathcal{X}_w$ is smooth of dimension at least $2$, then $\mathcal{X}_w$ is a dc singularity. Thus the class of singularities that admit deformations to dc singularities is a natural generalization of the class of smoothable singularities of dimension at least $2$. We show that determinantal and Pfaffian singularities belong to this class. This allows us to propose a notion of {\it generalized smoothability} by considering all singularities that deform to dc singularities. 

The paper is organized as follows. In Sct.\ \ref{Excess-Degree} we prove the EDF. Based on Ramanujam's interpretation of the Hilbert--Samuel multiplicity we recover and generalize  formulas due to Teissier for the Hilbert--Samuel multiplicity and Gaffney's Multiplicity-Polar Theorem. Our approach relies on basic intersection theory as developed by Fulton in \cite{Ful}, Bertini type theorems,  and the geometric theory of the multiplicity of pairs of standard graded algebras developed by Kleiman and Thorup in \cite{KT-Al}. We show how to derive the EDF in the complex analytic setting and discuss some of its applications to the deformation theory of singularities. 

In Sct.\ \ref{computing} we compute the local volume using a Noether normalization-type result of the author \cite[Prp.\ 2.6 ]{Ran19a} that shows that every reduced standard graded algebra $\mathcal{A}$ admits a homogeneous embedding in a standard graded algebra $\mathcal{B}$ that behaves like a polynomial ring. We show that in the computation for the local volume we can replace $H_{S}^{1}(\mathcal{A}_n)$ by $H_{S}^{0}(\mathcal{B}_n/\mathcal{A}_n)$ which is more manageable. The section culminates with a result characterizing the vanishing of the local volume. The main technical ingredient here are results of the author (\cite{Ran19a} and \cite{Rangachev2}) about the structure of $\mathrm{Ass}_{\mathcal{A}}(\mathcal{B}_n/\overline{\mathcal{A}_n})$. We pay special attention 
to the case when $\mathcal{A}$ is the Rees algebra of a module, in which case the local volume was introduced by \cite{Validashti} under the name of epsilon multiplicity. Its relevance to equisingularity was discovered by Kleiman, Ulrich and Validashti \cite{KUV}. Their work served as an inspiration for the present work. 

In Sct.\ \ref{sec. main result} we prove the LVF. We show that the restricted local volume specializes
with passage to the generic fibers using a result of the author about the finiteness of $\mathrm{Ass}_{\mathcal{A}}(\mathcal{B}_n/\overline{\mathcal{A}_n})$ (see \cite[Thm.\ 1.1 (ii)]{Ran19a}). Also, we show that the restricted local volume exists as a limit. 

In Sct.\ \ref{stability} we show how to extend the LVF to the case when $\dim Y >1$ under the assumption of volume stability  using a covering argument due to Gaffney and the author \cite{GR} and by computing the local volumes through one-parameter generic deformations. We prove a general version of Teissier's Principle of Specialization of Integral Dependence (PSID). 

In Sct.\ \ref{Whitney, Jacobian} we review the algebraic and analytic formulations of Whitney--Thom equisingularity (the $W_f$ and Thom's $A_f$ conditions) and we state the necessary results about integral closure of modules and conormal geometry needed for their numerical characterization. In Sct.\ \ref{numerical control} we show how to characterize numerically $W_f$ and $A_f$ using the LVF and the PSID.

In Sct.\ \ref{deficient section} we introduce the notion of deficient conormal (dc) singularities. 
We give examples coming from affine cones over projective varieties having duals with positive defect. We show that the dc property is intrinsic, using a result of Teissier about polar varieities, and that it is stable under infinitesimal deformations using the LVF. We show that Cohen--Macaulay codimension $2$, Gorensten codimension $3$, and more generally determinantal and Pfaffian singularities 
(with two exceptions), admit deformations to dc singularities. The proof is based on several results. 

First we show that the fibers of conormal spaces pullback set-theoretically under morphisms between affine spaces satisfying certain transversality conditions. In particular, in our setup a dc singularity pulls back to a dc singularity. We prove all this using stratification theory, the Lagrangian and functoriality nature of conormal spaces and a recent result of Gaffney and the author \cite[Thm.\ 3.1]{GRB}. 

Then using the structure theorems of Hilbert-Burch and Buchsbaum-Eisenbud, we show that the
classes of singularities under consideration can be obtained as pullbacks of generic determinantal and Pfaffian singularities by holomorphic maps between affine spaces.  By a result of Laksov and Buchweitz, unfolding of these maps gives rise to deformation of the singularities under consideration. By a complex analytic version of Thom's transversality due to Trivedi, the unfoldings can be done in such a way so that the resulting maps are generically transverse to the strata of a Whitney stratification of the generic singularities. Finally, we are left with showing that the generic determinantal and Pfaffian singularities are dc. This is achieved by computations of Gaffney and the author, and Lakshmibai and Singh.




Inspired by a result of Koll\'ar and Kov\'{a}cs \cite{KK18} we show that affine cones over normally embedded abelian varieties of dimension at least $2$ cannot be deformed to dc singularities. We finish the section by showing how to compute the restricted local volume associated with the conormal space of an isolated nonsmoothable Cohen--Macaulay codimension $2$ singularity.

{\bf  Acknowledgements.} It's a pleasure to thank Terence Gaffney, Steven Kleiman and Bernard Teissier for countless helpful and stimulating discussions and for their support and encouragement. Special thanks go to David Massey and Madhav Nori for many fruitful and stimulating conversations. I would like to thank  Lawrence Ein, Mihai Fulger, S\'{a}ndor Kov\'{a}cs, Rahul Singh, Bernd Ulrich, Javid Validashti and Mathias Zach for helpful conversations. I was partially supported by the University of Chicago FACCTS grant ``Conormal and Arc Spaces in the Deformation Theory of Singularities'' and NEU summer and travel grants. I would like to acknowledge the hospitality of IMPA and IMJ-PRG where part of this work was completed.

\section{The Excess--Degree Formula and applications}\label{Excess-Degree}

In this section we prove the {\it Excess--Degree Formula} (EDF) and generalize and recover multiplicity-polar results by Gaffney for the Buchsbaum--Rim (BR) multiplicity and by Teissier for the Hilbert--Samuel multiplicity. 

The following problem is at the heart of equisingularity theory. Let $h \colon X \rightarrow Y$ be a morphism of schemes of finite type over a field $\Bbbk$ with equidimensional fibers of positive dimension $r$. Assume $X$ is equidimensional and $Y$ is regular. Let $S$ be a closed subscheme of $X$ proper over $Y$ such that for each point $y \in Y$ the fiber $S_y$ is nowhere dense in $X_y$. Denote by $C$ the blowup of $X$ with center $S$, by $D$ the exceptional divisor, and by $c$ the blowup map $c \colon C \rightarrow X$. For each $y \in Y$ denote by $C(y)$ the blowup of $X_y$ with center $S_y$ and by $D(y)$ the exceptional divisor. Let $y_0$ be a closed point of $Y$.  We can ask: when do we have an equality of fundamental cycles
\begin{equation}\label{fund. cycles}
[C_{y_0}] =[C(y_0)]?
\end{equation}
It's not hard to see that (\ref{fund. cycles}) holds if and only if the fiber of the exceptional divisor $D_{y_0}$ is equidimensional. More precisely, we would like to find a  numerical invariant depending solely on $D(y)$ whose constancy across $Y$ guarantees (\ref{fund. cycles}) for each closed point $y_0 \in Y$. 

Set $l:= c_1 \mathcal{O}_{C}(1)$ and $l_y:= c_1 \mathcal{O}_{C(y)}(1)$ for each $y \in Y$. Denote by $D_{\mathrm{vert}}$ the union of the components of $D$ that surject onto $y_0$. For a closed point $y_0 \in Y$ and affine neighborhood $U$ of $y_0$ denote by $U'$ the punctured neighborhood $U-\{y_0\}$. A partial answer to our question is given by the following result.

\begin{theorem}[Excess--Degree Formula]\label{EDF} The following holds.
\begin{itemize}
\item [\rm{(i)}] Assume $Y$ is integral and regular of dimension one and $y_0$ is a closed point in $Y$. Let $U$ be a small enough affine neighborhood of $y_0$ such that $C_{y_0}$ is a principal Cartier divisor in $C$ after base change $U \rightarrow Y$ and $C$ is flat over $U'$. Then we have the following Excess--Degree Formula (EDF):
$$ \int_{S_{y_0}} l_{y_0}^{r-1}[D(y_0)] - 
\int_{S_{y}}  l_{y}^{r-1}[D(y)] = \int_{S_{y_0}} l^{r}[D_{\mathrm{vert}}]$$
for $y \in U'$. 
\item [\rm{(ii)}] Assume $Y$ is regular and $S$ is finite over $Y$. Then (\ref{fund. cycles}) holds if and only if $\int_{S_{y}}  l_{y}^{r-1}[D(y)]$ is constant for each closed $y$ is a sufficiently small affine neighborhood $U$ of $y_0$. 
\end{itemize}
\end{theorem}

\begin{proof}
Preserve the setup from the introduction. For each $k$ denote by $Z_k(C)$ the group of $k$-cycles on $C$ and by $A_{k}(C)$ the group of $k$-cycles modulo rational equivalence.  For each $y \in Y$ consider the refined  Gysin homomorphism $i_{y}^{!} \colon A_{k}(C) \rightarrow A_{k-1}(C_y)$ defined from the fiber square
\begin{displaymath}
\begin{CD}
C_y @>>> C\\
@VVV  @VVV\\
 \{y\}  @>>> Y
\end{CD}
\end{displaymath}
(see \cite[Sct.\ 6.2]{Ful}). For a $k$-cycle class $Z$ on $C$ set $Z_y: = i_{y}^{!}(Z).$ Finally, denote by $D_{\mathrm{hor}}$ the union of the irreducible components of $D$ that surject onto $Y$.

First, note that  $C(y_0)$ and $C_{y_0}$ are isomorphic over points $x \in X_{y_0}$ with $x \notin S_{y_0}$.  Also, the irreducible components of  $C(y_0)$ surject onto those of $X_{y_0}$. But $S_{y_0}$ is nowhere dense in $X_{y_0}$ by hypothesis. Hence the irreducible components of $C_{y_0}$ are either vertical components of $D$ or components that surject onto the irreducible components of $X_{y_0}$. Thus we have

\begin{equation}\label{key eq.}
[C(y_0)]  -   [C_{y_0}]   =  -[D_{\mathrm{vert}}] \ \text{in} \  Z_r(C). 
\end{equation}

Second, $ D\cdot[C_{y_0}]=C_{y_0}\cdot[D]$ because $C_{y_0}$ is Cartier.  Write $[D]=[D_{\mathrm{hor}}]+[D_{\mathrm{vert}}]$. Then $C_{y_0}\cdot[D] = C_{y_0}\cdot[D_{\mathrm{hor}}]$ in $Z_{r-1}(C)$ because $C_{y_0}$ is principal in $C$, so $C_{y_0}\cdot[D_{\mathrm{vert}}]=0$ (see Rmk.\ 2.3 in \cite{Ful}). Thus, 
$C_{y_0}\cdot [D] =[D_{\mathrm{hor}}]_{y_0}$. Hence $ D\cdot[C_{y_0}]=[D_{\mathrm{hor}}]_{y_0}.$ 

Additionally, observe that $D\cdot[C(y_0)]=[D(y_0)]$. Hence by intersecting each term in (\ref{key eq.}) with $D$ we get 

$$[D(y_0)]  - [D_{\mathrm{hor}}]_{y_0} = -D\cdot[D_{\mathrm{vert}}] \ \text{in} \  A_{r-1}(C)$$
or equivalently 
\begin{equation}\label{divisors}
[D(y_0)]  - [D_{\mathrm{hor}}]_{y_0} = l[D_{\mathrm{vert}}].
\end{equation}
as $D$ is dual to $\mathcal{O}_{C}(1)$. Because $\mathcal{O}_{C}(1)$ restricts to $O_{C(y_0)}(1)$ on $C(y_0)$ then $l^{r-1} [D(y_0)]= l_{y_0}^{r-1} [D(y_0)].$ Apply $l^{r-1}$ to both sides of (\ref{divisors})
to get

\begin{equation}\label{zero cycle}
l_{y_0}^{r-1}[D(y_0)]  - l^{r-1}[D_{\mathrm{hor}}]_{y_0} = l^{r}[D_{\mathrm{vert}}] \ \text{in} \  A_{0}(C).
\end{equation}

Next apply \cite[Prp.10.2]{Ful}
to the $1$-cycle $l^{r-1}[D_{\mathrm{hor}}]$ in $D$ and the map $D \rightarrow Y$ which is proper because it is composition of two proper maps: $D \rightarrow S$ and $S \rightarrow Y$. We have

\begin{equation}\label{conservation}
\int_{S_{y_0}} (l^{r-1}[D_{\mathrm{hor}}])_{y_0}=
\int_{S_y} (l^{r-1}[D_{\mathrm{hor}}])_{y}.
\end{equation}

Note that $[D]_{y} = [D(y)]$ because by assumption $C$ is flat over $U'$ so $C_y=C(y)$ as schemes. Then by \cite[Prop.\ 10.1 (d)]{Ful} about specialization of Chern classes, we have

\begin{equation}\label{specialization}
l^{r-1}[D_{\mathrm{hor}}]_{y}= l_{y}^{r-1}[D(y)].
\end{equation}

Combining (\ref{zero cycle}), (\ref{specialization}) and (\ref{conservation}) we get


$$\int_{S_{y_0}} l_{y_0}^{r-1}[D(y_0)] - \int_{S_y}l_{y}^{r-1}[D(y)] = \int_{S_{y_0}} l^{r}[D_{\mathrm{vert}}].$$
The proof of $\rm{(i)}$ is complete.

Consider $\rm{(ii)}$. Suppose that for each affine neighborhood $U$ of $y_0$ there exists $y \in U$ such that $$ \int_{S_{y_0}} l_{y_0}^{r-1}[D(y_0)] \neq
\int_{S_{y}}  l_{y}^{r-1}[D(y)].$$ 
Suppose $\dim D_{y_0} <r$. Because $D \rightarrow Y$ is proper, by upper semi-continuity $\dim D_y <r$ for each closed $y$ in a sufficiently small neighborhood $U$ of $y_0$.
Then $[D_y]=[D(y)]$ for $y \in U$. Because $i_{y} \colon y \hookrightarrow Y$ is a regular embedding, the operation of pullback along $i_y$ commutes with that of pushforth along the proper map $D \rightarrow Y$. Thus by \cite[Prp.10.2]{Ful} $\int_{S_{y}}  l_{y}^{r-1}[D(y)]$ is constant along $U$ which is a contradiction. Therefore, $\dim D_{y_0} \geq r$ which implies $[C_{y_0}] \neq [C(y_0)]$.

Next, suppose $[C_{y_0}] \neq [C(y_0)]$. Then $\dim D_{y_0} \geq r$. Replace $Y$ with an affine open neighborhood of $y_0$. Then $h_{|S}^{-1}Y$ is affine because $S \rightarrow Y$ is finite. Replace $X$ with an affine open subset containing $h_{|S}^{-1}Y$. As $X$ is affine we can fix an embedding $C \hookrightarrow X \times \mathbb{P}^m$. Denote by $A(C):=\mathcal{O}_X[v_0, \ldots, v_m]/\mathcal{I}_C$ the homogeneous coordinate ring of $C$ in $X \times \mathbb{P}^m$. Select $r$ general (to be specified below) hypersurfaces $V_1, \ldots, V_r$ in $X \times \mathbb{P}^m$ such that each hypersurface is defined by a general $\Bbbk$-linear combination of the generators of a sufficiently high power of the irrelevant ideal $(v_0, \ldots, v_m)$. 

Let's specify the genericity conditions on the $V_i$s. We rely on the fact that irrelevant ideal $(v_0, \ldots, v_m)$ is maximal in $A(C)$. By prime avoidance we can select the defining equations for the $V_i$s in $A(C)$ in such a way that the ideal of $V:=V_1 \cap \ldots \cap V_r$ is of height $r$ modulo each minimal prime. Then by the dimension formula \cite[Ch.\ 5, \S 14]{Matsumura} $V \cap C$ is equidimensional of dimension equal to $\dim Y$. Furthermore, by prime avoidance we can ensure that $\dim (V \cap D) = \dim Y-1$. Thus, the generic point of each irreducible component of $V \cap C$ is not contained in $D$. Because $C$ is birational to $X$, we get that $c(V \cap C)$ is equidimensional and $\dim c(V \cap C)=\dim Y$. Because $\dim D(y_0) = r-1$ by prime avoidance we can select the $V_i$s so that $V \cap D(y_0) = \emptyset$. Thus $c(V \cap C(y_0))$ is either empty or consists of finitely many points which do not lie in $S_{y_0}$. By shrinking $X$ if necessary we get $V \cap C(y_0) = \emptyset$. But $[C_{y_0}]-[C(y_0)] \subset D_{y_0}$. Because $\dim D_{y_0} \geq r$ we have $V \cap C_{y_0} \neq \emptyset$. Thus, $V \cap C_{y_0} \subset  D_{y_0}$. Hence $c(V \cap C)_{y_0} \subset S_{y_0}$. Therefore, $(h \circ c) (V \cap C)$ is dense in $Y$. 

We have the following local version of Bertini's theorem. Denote by $\mathfrak{m}_{y_0}$ the maximal ideal of $\mathcal{O}_{Y,y_0}$. Let $P_1, \ldots, P_n$ be prime ideals in $\mathcal{O}_{Y,y_0}$ different from $\mathfrak{m}_{y_0}$. Let $h_1 \in \mathfrak{m}_{y_0}-\mathfrak{m}_{y_0}^2$. Note that the ideal $(h_1)+\mathfrak{m}_{y_0}^2$ avoids each $P_i$. Then by E.\ Davis' version of prime avoidance \cite[Ex.\ 16.8]{Matsumura87}, there exists $h_2 \in \mathfrak{m}_{y_0}^2$ such that $h:=h_1+h_2$ avoids each $P_i$. Furthermore, because $h_1 \not \in \mathfrak{m}_{y_0}^2$ the ring $\mathcal{O}_{Y,y_0}/h\mathcal{O}_{Y,y_0}$ is regular. We will apply this observation to reduce the proof of part $\rm{(ii)}$ to the case $\dim Y=1$.

Say that a component of $D$ is {\it vertical} if it maps to a proper closed subset of $Y$. Set $\dim Y:=k$. Using repeatedly the local version of Bertini's theorem we select a nested sequence $Y:=Y_0 \supset Y_1 \ldots \supset Y_{k-1}$ where $Y_i$ is regular hypersurface in $Y_{i-1}$ through $y_0$ for $i=1, \ldots, k-1$ subject to the following genericity conditions:

\begin{itemize}
    \item [\rm{(1)}]$Y_1$ is chosen so that it intersects properly the images under $h \circ c$ of all vertical components of $D$.
    
    \item [\rm{(2)}] For each $i=1, \ldots, k-1$ set $C_{Y_i}:= C \times_{Y} Y_i$. Denote by $C(Y_i)$ the blowup of $X \times_{Y} Y_i$ with center  $S \times_{Y} Y_i$ and by $D(Y_i)$ the corresponding exceptional divisor. We choose $Y_i$ so that it intersects properly the closures of the images under $h \circ c$ of the vertical components of $D(Y_{i-1})$.
    \item[\rm{(3)}] $Y_{k-1}$ is not in the closure of $Y \setminus (h \circ c) (V \cap C)$.
    
\end{itemize}
Then by applying repeatedly Bertini's theorem for extreme morphisms \cite[Thm.\ 4.3]{Rangachev2} for the semi-local scheme $(X \times_{Y} Y_i,S \times_{Y} Y_i)$ we get $C_{Y_i}=D_{y_0} \cup C(Y_i)$. In particular, $C_{Y_{k-1}}= D_{y_0} \cup C(Y_{k-1})$ where $Y_{k-1}$ is a regular curve. 


We can assume that $(C \cap V) \times_{Y} Y_{k-1}$  surjects onto $Y_{k-1}$ after possibly replacing the latter with a smaller affine neighborhood containing $y_0$. But $(C \cap V) \times_{Y} Y_{k-1} = C_{Y_{k-1}} \cap V$. Thus $V \cap C(Y_{k-1})$ is a curve that surjects onto $Y_{k-1}$. In particular, $V \cap C(Y_{k-1})_{y_0} \neq \emptyset$. But $V \cap C(y_0) = \emptyset$ and $V$ is generic. Thus $D(Y_{k-1})$ has a vertical component $D_{\mathrm{vert}}$. But $D_{\mathrm{vert}}$ is supported over a subscheme of $S_{y_0} \times \mathbb{P}^m$, so  $\mathcal{O}_{C(Y_{k-1})}(1)$ is ample on $D_{\mathrm{vert}}$. Set $l:= c_{1}\mathcal{O}_{C(Y_{k-1})}(1)$. Then $l^r[D_{\mathrm{vert}}]>0$. By part $\rm{(i)}$ we obtain $$\int_{S_{y_0}} l_{y_0}^{r-1}[D(y_0)] \neq
\int_{S_{y}}  l_{y}^{r-1}[D(y)]$$
for generic $y \in Y_{k-1}$. 
\end{proof}

Below we prove strengthening of Thm.\ \ref{MPT} \rm{(i)}. We show that the assumption that $S_{y_0}$ is nowhere dense in $X_{y_0}$ in part $\rm{(i)}$ can be substantially relaxed which is what's needed for applications to equisingularity theory.  

\begin{corollary}\label{str} Let $X_{y_0}'$ be the union of those components of $X_{y_0}$ that are not irreducible components of $S$. Set $S_{y_0}'= S_{y_0} \times_{X_{y_0}} X_{y_0}'$. Denote by $C'(y_0)$ the blowup of $X_{y_0}'$ with center $S_{y_0}'$ and denote by $D'(y_0)$ its exceptional divisor. Then the EDF remains valid after replacing $D(y_0)$ with $D'(y_0)$. 
\end{corollary}
\begin{proof} Note that $C_{y_0}$ and $C'(y_0)$ are isomorphic over points $x \in X_{y_0}$ with $x \not \in S_{y_0}$. 
Therefore the cycle $[C(y_0)]-[C_{y_0}]$ is supported over the irreducible components of $S$ that are also irreducible components of $X_{y_0}$. Hence, once again $[C(y_0)]-[C_{y_0}]=-[D_{\mathrm{vert}}].$ Additionally, observe that 
 $D \cdot [C'(y_0)]=[D'(y_0)]$ because 
 $D \times_{C} C'(y_0) = D'(y_0)$. The rest of the proof goes through unchanged. 
\end{proof}
In equisingularity theory we apply Cor.\ \ref{str} with $C$ the blowup of the relative conormal space of a family
$(\mathfrak{X},x_0) \rightarrow (Y,y_0)$ with center the singular locus of $\mathfrak{X}$. In this setting $X_{y_0}'$ is the conormal space of $\mathfrak{X}_{y_0}$. So 
the intersection numbers appearing on the left-hand side of the EDF depend only on the fibers.

For related results to Thm.\ \ref{EDF} see \cite[Chp.\ 5]{Kol}. As a first application of Theorem \ref{MPT} \rm{(i)} we recover a result by Teissier for  the Hilbert-- Samuel multiplicity. Preserve the setup of Thm.\  \ref{EDF}. Assume $S$ is finite over $Y$. Following Fulton (see \cite{Ram}, and \S 4.3 and Ex.\ 4.3.4 in \cite{Ful}) define the Hilbert--Samuel multiplicity $$e(S_y,X_y):= \int_{S_y} l_{y}^{r-1}[D(y)].$$

\begin{theorem}[Teissier, Prp.\ 3.1 in \cite{Teissier2} and Rmk.\ 5.1.1 in  \cite{Teissier}]\label{MPTT}
We have 
$$e(S_{y_0},X_{y_0})-e(S_{y},X_y) = \mathrm{deg}(D_{\mathrm{vert}})$$
where $\mathrm{deg}(D_{\mathrm{vert}}):=\int l^r[D_{\mathrm{vert}}].$
\end{theorem}
\begin{proof} Follows immediately from Thm.\ \ref{EDF} \rm{(i)}.
\end{proof}

\begin{remark}\label{CA}
\rm{In complex analytic singularity theory one works with a complex analytic germ $(X,0) \subset (\mathbb{C}^{n+k},0)$ and a smooth subspace $(Y,0) \subset (X,0)$ which can be taken to be $(\mathbb{C}^{k},0)$. In this situation $h \colon X \rightarrow Y$ is obtained from a transverse projection to $Y$ (see the beginning of Sct.\ \ref{Whitney, Jacobian}). To obtain the EDF in this setting one needs to work with the right analogue of an affine subsets of $X$ and $Y$. This is done via distinguished compact Stein sets: one takes an open neighborhood $U$ of $0$ in $\mathbb{C}^{n+k}$ and considers $Q \cap X$ where $Q$ is a compact stone in $U$ (see Chp.\ II, Sct.\ 1 and Chp.\ III in Moonen's appendix in \cite{HIO88}). We do the same for $(Y,0)$. Next one considers the corresponding analytic spectrums and the projective analytic spectrum on the blowup. The basic intersection theory we used in the proof of the EDF extends to this new setting (cf.\ \cite[App.\ A]{Massey}).}
\end{remark}

\begin{example}
\rm{Thm.\ \ref{MPTT} provides a useful tool in the deformation theory of singularities. Consider an analytic function $f\colon (\mathbb{C}^n,0) \rightarrow (\mathbb{C},0)$ with an isolated critical point at the origin. Let ${\bf x}:=(x_1,\ldots,x_n)$ be a linear choice of coordinates. Denote by $J(f):=(\partial f / \partial x_1, \ldots, \partial f / \partial x_n)$ the {\it Jacobian ideal} of $f$ in $\mathcal{O}_{\mathbb{C}^n,0}$. Because $J(f)$ is generated by $n$ elements and is primary to the maximal ideal of the regular local ring $\mathcal{O}_{\mathbb{C}^n,0}$, the Hilbert--Samuel multiplicity $e(J(f),\mathcal{O}_{\mathbb{C}^n,0})$ equals
the Milnor number $\mu (f): = \dim_{\mathbb{C}} \mathbb{C}\{x_1, \ldots, x_n\}/J(f)$. Consider $f_t\colon (\mathbb{C}^n \times \mathbb{C}, (0,0)) \rightarrow (\mathbb{C},0)$ where $f_0=f$ and $f_t$ for $t \neq 0$ is a generic perturbation of $f$ such that its critical points are nondegenerate, or Morse. Let $J_{{\bf x}}(f_t)$ be the relative Jacobian ideal, i.e.\ the ideal generated by the partials of $f_t$ with respect to the ${\bf x}$ coordinates. 
Let $S$ be the subscheme of $\mathbb{C}^n \times \mathbb{C}$ defined by $J_{{\bf x}}(f_t)$. There are no vertical components of the exceptional divisor $D$ of $\mathrm{Bl}_{S} (\mathbb{C}^n \times \mathbb{C})$ because $D_0$ is set-theoretically in $\{0\} \times \mathbb{P}^{n-1}$. Applying Thm.\ \ref{MPTT}  we get that $\mu (f)= \mu(f_t)$ for generic $t$. But $f_t$ has only Morse critical points and their number is $\mu(f_t)$. Thus $\mu (f)$ equals the number of critical points of a generic perturbation of $f$.

As another example, let $(X,0) \rightarrow (Y,0)$ be a smoothing of an isolated hypersurface $(X_0,0)$ singularity, defined as the zero locus in $(\mathbb{C}^N,0)$ of a complex analytic function $F$, and $S$ be the subscheme defined by the partials of $F$ with respect to the fiber coordinates. Applying the EDF to this setting we get that the first left-hand side term is equal to the Hilbert--Samuel multiplicity of the Jacobian ideal in $\mathcal{O}_{X_{0},0}$ associated with $X_0$, whereas the second term is zero because $X_y$ is smooth. By  conservation of number argument, the right-hand side of the EDF is equal to the intersection multiplicity of the {\it relative polar curve} associated with the total space of the family with a generic fiber. In turn, by \cite[Cor.\ 1.5, pg.\ 320]{Teissier2} this intersection multiplicity is the sum of the Milnor number of $(X_{0},0)$ and the Milnor number of a generic hyperplane slice of $(X_0,0)$}. 
\end{example}

In what follows we extend Thm.\ \ref{EDF} to the more general case when we work with a projective scheme over $X$. Let $X \rightarrow Y$ be a morphism of schemes of finite type over a field $\Bbbk$ with equidimensional fibers. Assume $X$ is equidimensional and reduced and $Y$ is integral and regular of dimension one. Consider two graded sheaves of $\mathcal{O}_X$-algebras $\mathcal{A}:=\bigoplus \mathcal{A}_n$ and $\mathcal{A}':=\bigoplus \mathcal{A}_n'$ with $\mathcal{A}_0=\mathcal{A}_0'=\mathcal{O}_X$ such that $\mathcal{A}$ and $\mathcal{A}'$ are locally generated by $\mathcal{A}_1$ and $\mathcal{A}_1'$ as $\mathcal{O}_X$-algebras. Assume $\mathcal{A}_1$ and $\mathcal{A}_1'$ are coherent $\mathcal{O}_X$-modules and  $\mathcal{A} \subset \mathcal{A}'$ is a homogeneous and birational inclusion. Set $P:=\mathrm{Proj}(\mathcal{A})$ and $Q:=\mathrm{Proj}(\mathcal{A}')$. Assume that the irreducible components of $P$ and $Q$ surject onto those of $X$. Denote by $Z:=\mathbb{V}(\mathcal{A}_1)$ the variety of the ideal sheaf in $\mathcal{A}'$ generated by $\mathcal{A}_{1}$. Assume $Z$ is proper over $Y$. Let $C$ be the blowup of $Q$ with center $Z$ and let $D$ be its exceptional divisor. Denote by $c_p$ and $c_q$ the projections of $C$ to $P$ and $Q$ respectively. Let $p \colon P \rightarrow X$ and $q \colon Q \rightarrow X$ be structure morphisms. Set $\mathcal{L}:=\mathcal{O}_{P}(1)$ and $l:=c_{1}(\mathcal{L})$, and $\mathcal{L}':=\mathcal{O}_{Q}(1)$ and $l':=c_{1}(\mathcal{L}')$. 

Further, we assume that $\mathcal{A}$ and $\mathcal{A}'$ are contained in a graded $\mathcal{O}_X$-algebra sheaf $\mathcal{B}:=\oplus \mathcal{B}_n$ satisfying locally the properties described in Prp.\ \ref{nice embedding}. If $\mathcal{A}'$ is reduced, such an embedding exists by Prp.\ \ref{nice embedding}. 
For each $y \in Y$ denote by $\mathcal{B}_{y}$ the induced sheaf on $X_y$ and by $\mathcal{A}(y)$ and $\mathcal{A}'(y)$ the images of $\mathcal{A}$ and $\mathcal{A}'$ in $\mathcal{B}_y$. Set $P(y):= \mathrm{Proj}(\mathcal{A}(y))$ and $Q(y):=\mathrm{Proj}(\mathcal{A}'(y))$. Consider the following diagram 

\begin{displaymath}
\begin{CD}
C(y) @>c_q>> Q(y)\\
@VVc_pV  @VVqV\\
P(y)  @>p>> X_y
\end{CD}
\end{displaymath}
where $C(y)$ is the blowup of $Q(y)$ with center $Z \times_{Q} Q(y)$ and $D(y)$ is the corresponding exceptional divisor. Set $\mathcal{L}_y:= \mathcal{O}_{P(y)}(1)$ and $\mathcal{L}_y':= \mathcal{O}_{Q(y)}(1)$ and $l_y:= c_{1}(\mathcal{L}_y')$ and 
 $l_y':= c_{1}(\mathcal{L}_y')$. Assume $P(y)$ and $Q(y)$ are birational and equidimensional of dimension $r$. Define the generalized Buchsbaum--Rim (BR) multiplicity (see \cite[Sct.\ 5]{KT-Al})
\begin{equation}\label{BR}
e(\mathcal{A}(y),\mathcal{A}'(y)):= \sum_{i=0}^{r-1} \int_{Z_y} (c_{p}^{*}l_{y})^{r-i-1}(c_{q}^{*}l_{y}')^{i}[D(y)]. 
\end{equation}
where $c_{p}^{*}l_{y}= c_{1}(c_{p}^{*}\mathcal{L}_y)$ and $c_{q}^{*}l_{y}'= c_{1}(c_{q}^{*}\mathcal{L}_y')$. Finally, denote by $D_{\mathrm{vert}}^{P}$ and $D_{\mathrm{vert}}^{Q}$ the projections of $D_{\mathrm{vert}}$ to $P$ and $Q$. 
\begin{theorem}\label{GMPT}
There exists an affine neighborhood $U$ of $y_0$ in $Y$
such that 
\begin{equation}\label{GMPTF}
e(\mathcal{A}(y_0),\mathcal{A}'(y_0))-e(\mathcal{A}(y),\mathcal{A}'(y))=l^r[D_{\mathrm{vert}}^{P}]-l'^r[D_{\mathrm{vert}}^{Q}]. 
\end{equation}  
for each closed point $y \in U-\{y_0\}$.
\end{theorem}
\begin{proof}
By Prp.\ \ref{nice embedding} the irreducible components of $Q(y_0)$ surject onto those of $X_{y_0}$. Because $P(y_0)$ and $Q(y_0)$ are birational $q(Z_{y_0})$ and $X_{y_0}$ do not share irreducible components. By Cor.\ \ref{str} applied to the family $Q \rightarrow Y$ with $S:=Z$ and by (\ref{divisors}) we get 
\begin{equation}\label{mixed divisors}
(c_{p}^{*}l_{y_0})^{r-i-1}(c_{q}^{*}l_{y_0}')^{i}[D(y_0)]  -(c_{p}^{*}l_{y})^{r-i-1}(c_{q}^{*}l_{y}')^{i} [D_{\mathrm{hor}}]_{y_0} = -(c_{p}^{*}l_{y})^{r-i-1}(c_{q}^{*}l_{y}')^{i}D[D_{\mathrm{vert}}].
\end{equation}
By (\ref{conservation}) and (\ref{specialization}) we get
\begin{equation}\label{conserv}
\int_{Z_{y_0}}(c_{p}^{*}l_{y_0})^{r-i-1}(c_{q}^{*}l_{y_0}')^{i} [D_{\mathrm{hor}}]_{y_0}=
\int_{Z_y} (c_{p}^{*}l_{y})^{r-i-1}(c_{q}^{*}l_{y}')^{i}[D(y)].
\end{equation}
By \cite[Prp.\ 2.2 ]{KT-Al} we have 
\begin{equation}\label{KTDiv}
\mathcal{O}_{C}(D) = c_{p}^{*}\mathcal{L} \otimes c_{q}^{*}\mathcal{L}'^{-1}.
\end{equation}
Then (\ref{KTDiv}) yields
\begin{equation}\label{cancellation}
\sum_{i=1}^{r-1}-(c_{p}^{*}l_{y})^{r-i-1}(c_q^{*}(l_{y}')^{i}D[D_{\mathrm{vert}}]=(c_p^{*}l)^r[D_{\mathrm{vert}}]-(c_q^{*}l')^r[D_{\mathrm{vert}}].
\end{equation}
Summing over $i$ in (\ref{mixed divisors}) and plugging (\ref{cancellation}) and (\ref{conserv}) in (\ref{mixed divisors}) we get 
$$e(\mathcal{A}(y_0),\mathcal{A}'(y_0))-e(\mathcal{A}(y),\mathcal{A}'(y))=(b_p^{*}l)^r[D_{\mathrm{vert}}]-(b_c^{*}l')^r[D_{\mathrm{vert}}].$$
Applying the projection formula  (see  \cite[Prp.\ 2.5 (c)]{Ful}) to each term of the right-hand side of the last equality 
we get (\ref{GMPTF}). 
\end{proof}

Suppose $X$ and $Y$ are local with closed points $x_0$ and $y_0$ respectively and suppose the fibers of $X \rightarrow Y$ are equidimensional of positive dimension $d$. Let $\mathcal{M} \subset \mathcal{N} \subset \mathcal{F}$ be coherent $\mathcal{O}_{X}$-modules such that $\mathcal{F}$ is free and $\mathcal{M}$ and $\mathcal{N}$ are free of constant rank $e$ at the generic point of each irreducible component of $X$. Let $\mathcal{R}(\mathcal{M})$ and $\mathcal{R}(\mathcal{N})$ be the Rees algebras of $\mathcal{M}$ and $\mathcal{N}$ respectively. These are defined as the subalgebras of $\mathrm{Sym}(\mathcal{F})$ generated in degree one by the generators of each of the two modules. 

Now work in the setup of Thm.\ \ref{GMPT}. Set $C:= \mathrm{Proj}(\mathcal{R}(\mathcal{M}))$ and $P:= \mathrm{Proj}(\mathcal{R}(\mathcal{N}))$.   Assume $q(Z)$ is finite over $Y$. One can check that this is equivalent to assuming that  $\mathrm{Supp}_{X}(\overline{\mathcal{N}}/\overline{\mathcal{M}})$ is finite over $Y$, where $\overline{\mathcal{N}}$ and $\overline{\mathcal{M}}$ are the integral closure of the two modules in $\mathcal{F}$ (see Sct.\ \ref{Whitney, Jacobian} for the definition of integral closure of a module). For each $y \in Y$ denote by $\mathcal{F}_{y}$ the restriction of $\mathcal{F}$ to $X_y$ and by $\mathcal{M}(y)$ and $\mathcal{N}(y)$ the images of $\mathcal{M}$ and $\mathcal{N}$ in $\mathcal{F}_{y}$. Define the Buchsbaum--Rim multiplicity $e(\mathcal{M}(y),\mathcal{N}(y))$ as in (\ref{BR}) with $\mathcal{A}(y)=\mathcal{R}(\mathcal{M}(y))$ and $\mathcal{A}'(y)=\mathcal{R}(\mathcal{N}(y))$. Here we turn results of Kleiman and Thorup (see \cite[Sct.\ 5]{KT-Al}) into a definition. In the original treatment of Buchsbaum and Rim \cite{Buch}, the multiplicity $e(\mathcal{M}(y),\mathcal{N}(y))$ is defined as the normalized leading coefficient of the length $\lambda (\mathcal{N}^{i}(y)/\mathcal{M}^{i}(y))$ which is a polynomial of degree $r:=d+e-1$ for $i$ large enough, where $\mathcal{N}^{i}(y)$ and $\mathcal{M}^{i}(y)$ are the $i$th graded components of the respective Rees algebras.

Suppose $\Bbbk$ is algebraically closed field of characteristic zero. Assume $X$ is generically reduced and $Y$ is smooth of arbitrary dimension. Let $NF(\mathcal{M})$ be the nonfree locus of $\mathcal{M}$ in $X$. Consider the composition of maps

$$p^{-1} (NF(\mathcal{M})) \hookrightarrow  X \times \mathbb{P}^{g(M)-1} \xrightarrow{pr_2} \mathbb{P}^{g(M)-1}$$
where $g(M)$ is the number of a generating set for $\mathcal{M}$ as an $\mathcal{O}_X$-module. As $\mathcal{M}$ is of generic rank $e$, by the Kleiman Transversality Theorem \cite{K}, the intersection of $p^{-1} (NF(\mathcal{M}))$ with a general plane $H_{r}$ from  $\mathbb{P}^{g(M)-1}$ of codimension $r$, is of dimension at most $\dim Y - 1$. Denote by $\Gamma_d(\mathcal{M})$ the projection of $\mathrm{Proj}(\mathcal{R}(\mathcal{M})) \cap H_{r}$ to $X$. This is what Gaffney \cite{GaffP} calls the {\it polar curve} of $\mathcal{M}$. For a generic $y  \in Y$ the fiber  of $\Gamma_d(\mathcal{M})$ over $y$ consists of the same number of points, each of them appearing with multiplicity one because $X$ is generically reduced, and because at each one of them $\mathcal{M}$ is free. Denote this number by $\mathrm{mult}_{Y}\Gamma_d(\mathcal{M})$. Similarly, define $\mathrm{mult}_{Y}\Gamma_d(\mathcal{N}).$ Thm.\ \ref{GMPT} yields Gaffney's Multiplicity-Polar Theorem (see \cite{GaffMPT}).

\begin{corollary}(Gaffney).\label{MPT} Suppose $\Bbbk$ is algebraically closed of characteristic zero. Suppose $X$ is generically reduced, $Y$ is smooth, and $\mathrm{Supp}_{X}(\overline{\mathcal{N}}/\overline{\mathcal{M}})$ is finite over $Y$. Then for generic $y \in Y$
$$
e(\mathcal{M}(y_0),\mathcal{N}(y_0))-
e(\mathcal{M}(y),\mathcal{N}(y))=\mathrm{mult}_{Y}\Gamma_d(\mathcal{M})-\mathrm{mult}_{Y}\Gamma_d(\mathcal{N}).$$
\end{corollary}
\begin{proof}
Let $P_{\mathrm{vert}}$ and $Q_{\mathrm{vert}}$ be the union of irreducible components of $p^{-1}X_{y_0}$ and $q^{-1}X_{y_0}$ that are supported over $x_0$. Then $D_{\mathrm{vert}}^{P} = P_{\mathrm{vert}}$ and $D_{\mathrm{vert}}^{Q}=Q_{\mathrm{vert}}$. By a conservation of number, for generic smooth curve $Y'$ passing through $y_0$, the degrees on the right-hand side of (\ref{GMPTF})
are equal to $\mathrm{mult}_{Y}\Gamma_d(\mathcal{M})$ and $\mathrm{mult}_{Y}\Gamma_d(\mathcal{N})$, respectively. Then the statement follows from (\ref{GMPTF}).
\end{proof}

If $\mathcal{N}$ is free, then $\dim q^{-1}x_0 = e-1<r$, so $\mathrm{mult}_{Y}\Gamma_d(\mathcal{N})=0$. Additionally, if $e=1$ one recovers Thm.\ \ref{MPTT}. In these situations the change of multiplicities detects precisely the presence of $P_{\mathrm{vert}}$. In general, the multiplicities may fail to detect $P_{\mathrm{vert}}$ because of the contribution of $\mathrm{mult}_{Y}\Gamma_d(\mathcal{N})$. To get rid of $\mathrm{mult}_{Y}\Gamma_d(\mathcal{N})$, instead of the Buchsbaum--Rim multiplicity we use the local volume which absorbs $\mathrm{mult}_{Y}\Gamma_d(\mathcal{N})$. This is done by the LVF in Thm.\ \ref{LVF}. Another way to do this is indicated in Rmk.\ \ref{Fujita Rmk}. There an asymptotic version of (\ref{GMPTF}) is indicated in which the contribution coming from the various $D_{\mathrm{vert}}^Q$ ``vanishes asymptotically". 


\section{Computing and vanishing of local volumes}\label{computing} 
Let $(R,\mathfrak{m})$ be a reduced Noetherian local equidimensional ring
of dimension at least $2$. Let $\mathcal{A}:= \oplus_{i=0}^{\infty} \mathcal{A}_i$ be a reduced equidimensional standard graded $R$-algebra of dimension $r$.  Denote by $\lambda_{R}(-)$ the length function. In this section we will be interested in computing
\begin{equation}\label{epsilon}
\varepsilon(\mathcal{A}):= \limsup_{n \to \infty} \frac{r!}{n^r} \lambda_{R}( H_{\mathfrak{m}}^{1}(\mathcal{A}_n)).
\end{equation}
In our applications to geometry, $R$ is be assumed to be essentially of finite type over a field and $\mathcal{A}$ is the homogeneous coordinate ring of $\mathrm{Proj}(\mathcal{A})$ associated with an invertible very ample sheaf $\mathcal{L}$ on $\mathrm{Proj}(\mathcal{A})$ relative to $\mathrm{Spec}(R)$.  Then $(\ref{epsilon})$ is what Fulger\cite{Fulger} calls the {\it local volume} of $\mathcal{L}$.


When $\mathcal{A}$ is the Rees algebra of a torsion-free finitely generated $R$-module $\mathcal{M}$ (see \cite{Eisenbud} for a definition), then (\ref{epsilon}) is called the {\it epsilon multiplicity} of $\mathcal{M}$. It is denoted by $\varepsilon(\mathcal{M})$ (see \cite{Validashti} and the remark below). 

Often it's preferable to work with an $H_{\mathfrak{m}}^{0}$ instead of $H_{\mathfrak{m}}^{1}$. For example, we can do this if we assume that the minimal primes of $\mathcal{A}$ contract to minimal primes of $R$. Set $\mathfrak{X}:=\mathrm{Spec}(R)$. Let $x_0$ be the closed point of $\mathfrak{X}$. Then $\mathcal{A}$ viewed as $\mathcal{O}_{\mathfrak{X},x_0}$-module is torsion free. Set $U:=\mathfrak{X}-x_0$. For each $n$ consider the standard exact sequence 
$$H_{x_0}^{0}(\mathfrak{X},\widetilde{\mathcal{A}_n}) \rightarrow H^{0}(\mathfrak{X},\widetilde{\mathcal{A}_n}) \rightarrow H^{0}(U,\widetilde{\mathcal{A}_n)} \rightarrow H_{x_0}^{1}(\mathfrak{X},\widetilde{\mathcal{A}_n}) \rightarrow H^{1}(\mathfrak{X},\widetilde{\mathcal{A}_n})$$
where $\widetilde{A_n}$ is the associated sheaf of $\mathcal{A}_n$. The two extreme terms vanish because $\mathcal{A}$ is torsion free and $\mathfrak{X}$ is affine. Thus we can compute $H_{x_0}^{1}(\mathcal{A}_n)$ as $H^{0}(U,\widetilde{\mathcal{A}_n}) / \widetilde{\mathcal{A}_n}$ (see Ex.\ 3.3 (b) Chp.\ III in  \cite{Hartshorne}).

The existence of (\ref{epsilon}) as a limit 
has been established in some cases by Das \cite{Das} and Cutkosky \cite{Cut} based on ideas of Kaveh and Khovanskii \cite{KK}, and Okounkov \cite{Ok}, and by Fulger \cite{Fulger}. 

Assume $\mathcal{B}$ is a graded $R$-algebra, not necessarily finitely generated over $R$, such that $\mathcal{A} \subset \mathcal{B}$ is a homogeneous inclusion and
\begin{equation}
\lim_{n \to \infty} \frac{r!}{n^r}\lambda_{R}( H_{\mathfrak{m}}^{i}(\mathcal{B}_n))=0
\end{equation}
for $i=1$ and $2$. Then an exact sequence of local cohomology yields 
\begin{equation}\label{zero and one}
\limsup_{n \to \infty} \frac{r!}{n^r}\lambda_{R}(H_{\mathfrak{m}}^{1}(\mathcal{A}_n)) = \limsup_{n \to \infty} \frac{r!}{n^r}\lambda_{R}(H_{\mathfrak{m}}^{0}(\mathcal{B}_n/\mathcal{A}_n)).
\end{equation}
The reason we will make use of (\ref{zero and one}) is that $H_{\mathfrak{m}}^{0}(\mathcal{B}_n/\mathcal{A}_n)$ is more manageable 
than $H_{\mathfrak{m}}^{1}(\mathcal{A}_n)$ when studying the behavior of $\varepsilon(\mathcal{A})$ in families. 
Below we list several instances when the representation (\ref{zero and one}) is possible. Denote by  $c_{\mathcal{A}}$ and by $c_{\mathcal{B}}$ the structure morphisms from $\mathrm{Proj}(\mathcal{A})$ and  $\mathrm{Proj}(\mathcal{B})$ to $\mathrm{Spec}(R)$, respectively. 

\begin{proposition}\label{intr. of limits} Assume that $(R,\mathfrak{m})$ is a local Noetherian equidimensional ring with $\dim R \geq 2$. The following holds:
\begin{enumerate}
\item[\rm{(i)}] Assume $\mathrm{depth}(R) \geq 2$. Suppose the minimal primes of $\mathcal{A}$ contract to minimal primes of $R$. Then (\ref{zero and one}) holds with $\mathcal{B}_n:= \mathcal{A}_n^{**}$, where $\mathcal{A}_n^{**}$ is the double dual of the $R$-module $\mathcal{A}_n$.

\item[\rm{(ii)}] Assume that $R$ is Nagata and $\mathcal{B}$ is reduced, equidimensional of dimension $r$. Suppose that $\codim c_{\mathcal{B}}^{-1}(\mathfrak{m}) \geq 2$ in $\mathrm{Proj}(\mathcal{B})$ and the minimal primes of $\mathcal{B}$ contract to minimal primes of $R$. Then  (\ref{zero and one}) holds. 

\item[\rm{(iii)}] Assume $R$ is essentially of finite type over a field. Let $\mathcal{M}$ be an $R$-module free of rank $e$ at locally at each minimal prime of $R$. Then $\mathcal{M}$ admits an embedding into an $R$-free module $\mathcal{F}$ of rank $e$ such that (\ref{zero and one}) holds with $\mathcal{A}:=\mathcal{R}(\mathcal{M})$ and $\mathcal{B}:= \mathrm{Sym}(\mathcal{F})$. Moreover, the right-hand side of (\ref{zero and one}) is in fact a limit. 
\end{enumerate}
\end{proposition}
\begin{proof}
Consider $\rm{(i)}$. Because $\mathcal{A}$ is reduced and its minimal primes contract to minimal primes of $R$, then $\mathcal{B}$ and $\mathcal{A}$ are torsion free. Because $X$ is reduced, the natural map $\mathcal{A}_n \rightarrow \mathcal{A}_n^{**}$ is injection by a straightforward generalization of  \cite[\href{http://stacks.math.columbia.edu/tag/0AV0}{Tag 0AV0}]{Stacks}. But $\mathcal{A}_n^{**}$ is reflexive, so any $2$-regular sequence from $\mathfrak{m}$ lifts to a $2$-regular sequence of $\mathcal{A}_n^{**}$ (see  \cite[\href{http://stacks.math.columbia.edu/tag/0AV5}{Tag 0AV5}]{Stacks}). Thus, $H_{x_0}^{i}(\mathcal{B}_n)=0$ for $i=0,1$ by \cite[Thm.\ 3.8]{Gr}. 

Consider $\rm{(ii)}$. Observe that the total ring of fractions $Q(\mathcal{B})$ of 
$\mathcal{B}$ is a graded ring. Denote by $\overline{\mathcal{B}}$
the integral closure of $\mathcal{B}$ in $Q(\mathcal{B})$. Then $\overline{\mathcal{B}}$ is graded by Prp.\ 2.3.5 in \cite{Huneke}. Because $\mathcal{B}$ is reduced and is of finite type over $R$, which is Nagata, then  $\overline{\mathcal{B}}$ is module-finite over $\mathcal{B}$ by \cite[\href{http://stacks.math.columbia.edu/tag/03GH}{Tag 03GH}]{Stacks}. Thus $H_{\mathfrak{m}}^{0}(\overline{\mathcal{B}}/\mathcal{B})$ is a finite $\mathcal{B}/\mathfrak{m}^{l}\mathcal{B}$-module for some positive $l$. This implies that $\lambda_{R}(H_{\mathfrak{m}}^{0}(\overline{\mathcal{B}_n}/\mathcal{B}_n))$ is a polynomial of degree at most $\dim c_{\mathcal{B}}^{-1}(\mathfrak{m})$. Because $\mathcal{B}$ is equidimensional of dimension $r$, and the minimal primes of $\mathcal{B}$ contract to minimal primes of $R$, then $\dim c_{\mathcal{B}}^{-1}(\mathfrak{m}) \leq r-1$. Hence 
\begin{equation}\label{int. zero}
\limsup_{n \to \infty} \frac{r!}{n^r}\lambda_{R}(H_{\mathfrak{m}}^{0}(\overline{\mathcal{B}_n}/\mathcal{B}_n)) =0.
\end{equation}
Consider the sequence 
\begin{equation}\label{standard ses}
H_{\mathfrak{m}}^{0}(\overline{\mathcal{B}_n}/\mathcal{B}_n) \rightarrow H_{\mathfrak{m}}^{1}(\mathcal{B}_n) \rightarrow H_{\mathfrak{m}}^{1}(\overline{\mathcal{B}_n}).
\end{equation}
Because $\overline{\mathcal{B}}$ and $\mathcal{B}$ are generically equal, each minimal prime of the former ring contracts to a minimal prime of the latter. Select $x_1 \in \mathfrak{m}$ that avoids the minimal primes of $R$. Hence $x_1$ avoids the minimal primes of $\overline{\mathcal{B}}$. Because $\overline{\mathcal{B}}$ is normal, it follows that
the associated primes of the ideal $(x_1)$ generated by $x_1$ in  $\overline{\mathcal{B}}$ are of height one. But  $\codim c_{\mathcal{B}}^{-1}(\mathfrak{m}) \geq 2$ in $\mathrm{Proj}(\mathcal{B})$, so  $\codim c_{\overline{\mathcal{B}}}^{-1}(\mathfrak{m}) \geq 2$ in $\mathrm{Proj}(\overline{\mathcal{B}})$. Therefore, none of the associated primes of $(x_1)$ in  $\overline{\mathcal{B}}$ contracts to $\mathfrak{m}$. Hence by prime avoidance we can select $x_2 \in \mathfrak{m}$ such that $x_2$ is a nonzero divisor of $\overline{\mathcal{B}}/x_1\overline{\mathcal{B}}$. In this way we have constructed a $2$-regular sequence of $\overline{\mathcal{B}_n}$ for each $n$. Thus $H_{\mathfrak{m}}^{i}(\overline{\mathcal{B}_n})=0$ for $i=0,1$ by \cite[Thm.\ 3.8]{Gr}. Finally, (\ref{standard ses}) and (\ref{int. zero}) yield (\ref{zero and one}). 

Consider $\rm{(iii)}$. Note that the Rees algebra of $\mathcal{R}(\mathcal{M})$ is reduced because $X$ is reduced. Thus $\mathcal{R}(\mathcal{M}) \hookrightarrow \mathcal{R}(\mathcal{M}) \otimes Q(R)$ where $Q(R)$ is the total ring of fractions of $R$. Let $\mathfrak{p}_1, \ldots, \mathfrak{p}_q$ be the minimal primes of $R$. By hypothesis $\mathcal{M}_{\mathfrak{p}_i} =e$ for each for each minimal prime $\mathfrak{p}_i$. Then there exists $e$ generic combinations of the generators of $\mathcal{M}$ that generate  $\mathcal{M}_{\eta_i}$ for each $i$. Consider the module $\mathcal{F}'$ generated by these elements in $\mathcal{M} \otimes Q(R)$. Scale the generators of $\mathcal{F}'$ with elements from $Q(R)$ in such a way that all of the remaining generators of $\mathcal{M}$ can be expressed as a linear combination of the scaled generators $\mathcal{F}'$ with coefficients in $R$. Denote by $\mathcal{F}$ the $R$-module generated by the scaled generators of $\mathcal{F}'$. Then $\mathcal{F}$ is free of rank $e$.
At this point one can apply part \rm{(ii)} with $\mathcal{A}:=\mathcal{R}(\mathcal{M})$ and $\mathcal{B}:=\mathrm{Sym}(\mathcal{F})$ to derive (\ref{zero and one}), or proceed directly as follows.
Let $\mathcal{F}^n$ be the $n$th symmetric power of $\mathcal{F}$ and $\mathcal{M}^n$ be the $n$th homogeneous component of the subalgebra of $\mathrm{Sym}(\mathcal{F})$ generated by $\mathcal{M}$. Consider the exact sequence 
\begin{equation}\label{ses modules}
H_{\mathfrak{m}}^{0}(\mathcal{F}^n) \rightarrow H_{\mathfrak{m}}^{0}(\mathcal{F}^n/\mathcal{M}^n) \rightarrow H_{\mathfrak{m}}^{1}(\mathcal{M}^n) \rightarrow H_{\mathfrak{m}}^{1}(\mathcal{F}^n).
\end{equation}
Because $X$ is reduced of positive dimension and $\mathcal{F}^n$ is free, then $H_{\mathfrak{m}}^{0}(\mathcal{F}^n)=0$. 
Because $\mathcal{F}^n$ is free of rank $\binom{e+n-1}{n}$, then  $H_{\mathfrak{m}}^{1}(\mathcal{F}^n)$ is equal to the direct sum of  $\binom{e+n-1}{n}$ copies of $H_{\mathfrak{m}}^{1}(R)$. Because $R$ is esentially of finite type over a field, then $R$ is a homomorphic image of a regular ring. Because $\dim R \ge 2$, Grothendieck's finiteness theorem (see Expos\'e VIII, Corollaire 2.3 in \cite{Gr2}) implies that $H_{\mathfrak{m}}^{1}(R)$ is finitely generated. Hence $\lambda_{R}( H_{\mathfrak{m}}^{1}(\mathcal{F}^n)) = O(n^{e-1})$. Recall that by \cite[Prp.\ 5.1 $\rm{(2)}$]{Rangachev2} we have $\dim \mathcal{R}(\mathcal{M})= \dim X + e-1$. Because $R$ is of positive dimension, then $r>e-1$, so $\lim_{n \to \infty} \frac{r!}{n^r} \lambda_{R}( H_{\mathfrak{m}}^{1}(\mathcal{F}^n))=0$ which proves (\ref{zero and one}) for $\mathcal{A}:=\mathcal{R}(\mathcal{M})$ and $\mathcal{B}:= \mathrm{Sym}(\mathcal{F})$. 

Note that $R$ is analytically unramified because $R$ is reduced and of finite type over a field. So by Theorem 3.2 in \cite{Cut}
$\limsup_{n \to \infty} \frac{r!}{n^r}\lambda_{R}(H_{\mathfrak{m}}^{0}(\mathcal{F}^n/\mathcal{M}^n))$
exists as a limit. \end{proof}

The next proposition guarantees the existence of a $\mathcal{B}$ satisfying the hypothesis of Prp.\ \ref{intr. of limits} \rm{(ii)}. The proof is based on Noether normalization. 

\begin{proposition}[Prp.\ 2.6 in \cite{Ran19a}]\label{nice embedding}
 Suppose $R$ is a reduced equidimensional universally catenary Noetherian ring of positive dimension or an infinite field. Assume $\mathcal{A} = \oplus_{i=0}^{\infty} \mathcal{A}_{i}$ is a reduced equidimensional standard graded algebra over $R$. Assume that the minimal primes of $\mathcal{A}$ contract to minimal primes of $R$. Then there exists a standard graded $R$-algebra $\mathcal{B} = \oplus_{i=0}^{\infty} \mathcal{B}_{i}$ such that 
\begin{enumerate}
    \item[\rm{(1)}]$\mathcal{B}$ is a birational extension of $\mathcal{A}$, and the inclusion $\mathcal{A} 
 \subset \mathcal{B}$ is  homogeneous;
   
    \item[\rm{(2)}] For each prime $\mathfrak{p}$ in $R$ the minimal primes of $\mathfrak{p}\mathcal{B}$ are of height at least $\mathrm{ht}(\mathfrak{p}/\mathfrak{p}_{\mathrm{min}})$ where $\mathfrak{p}_{\mathrm{min}}$ is a minimal prime of $R$ contained in $\mathfrak{p}$.
\end{enumerate}
\end{proposition}

\begin{corollary}\label{epsilon as limit}
Suppose $(R,\mathfrak{m})$ is excellent and equidimensional with $\dim R \geq 2$. Suppose $\mathcal{A}$ is reduced. Then $\varepsilon(\mathcal{A})$ exists as a finite limit.
\end{corollary}
\begin{proof}
Because $R$ is excellent, then $R$ is Nagata and universally catenary. By Prp.\ \ref{nice embedding} applied with $\mathfrak{p}=\mathfrak{m}$ there exists a standard graded $R$-algebra $\mathcal{B}$ such that $\mathcal{A} \subset \mathcal{B}$ is a homogeneous inclusion, and $\mathcal{B}$ satisfies the hypothesis of Prp.\ \ref{intr. of limits} \rm{(ii)} because $R$ is local and $\dim R \geq 2$. Then by Prp.\ \ref{intr. of limits} \rm{(ii)} (\ref{zero and one}) holds. By \cite[Thm.\ 5]{Das}
$\varepsilon(\mathcal{A})$ exists as a finite limit.
\end{proof}

The following proposition is key to proving
our vanishing result for $\varepsilon(\mathcal{A})$. It's essentially the content of the two main results of \cite{Ran19a}. Let $\mathcal{A} \subset \mathcal{B}$ be commutative rings with identity. Denote by $\overline{\mathcal{A}}$ the integral closure of $\mathcal{A}$ in $\mathcal{B}$. 

\begin{proposition}\label{associated primes} Let $(R,\mathfrak{m})$ be a local Noetherian equidimensional ring. Let $\mathcal{A} \subset \mathcal{B}$ be finitely generated $R$-algebras. 
\begin{enumerate}
    \item [\rm{(i)}] Suppose $R$ is universally catenary and the minimal primes of $\mathcal{B}$ contract to minimal primes of $\mathcal{A}$. If $\mathfrak{m} \in \mathrm{Ass}_{R}(\mathcal{B}/\overline{\mathcal{A}})$, then 
    $\codim c_{\mathcal{A}}^{-1}(\mathfrak{m}) \leq 1$.
    \item [\rm{(ii)}] Suppose $\codim c_{\mathcal{B}}^{-1}(\mathfrak{m}) \geq  2$. If $\codim c_{\mathcal{A}}^{-1}(\mathfrak{m}) \leq 1$, then $\mathfrak{m} \in \mathrm{Ass}_{R}(\mathcal{B}/\overline{\mathcal{A}}).$
\end{enumerate}
\end{proposition}
\begin{proof} By the second part of \cite[Prp.\ 2.1]{Ran19a} or \cite[\href{http://stacks.math.columbia.edu/tag/05DZ}{Tag 05DZ}]{Stacks}
it's enough to consider primes in $\mathrm{Ass}_{\mathcal{A}}(\mathcal{B}/\overline{\mathcal{A}})$ contracting to $\mathfrak{m}$. Then part \rm{(i)} follows from \cite[Thm.\ 1.1 \rm{(iii)}]{Ran19a}.  Part \rm{(ii)} follows from \cite[Thm.\ 1.2]{Ran19a}.
\end{proof}

The following theorem characterizes the vanishing of $\varepsilon(\mathcal{A})$. The result is the main inspiration for introducing the class of deficient conormal singularities in  Sct.\ \ref{deficient section}. 
\begin{theorem}\label{vanishing} Let $(R,\mathfrak{m})$ be a reduced Noetherian local equidimensional ring with $\dim R \geq 2$. Assume $R$ is Nagata and universally catenary. Assume $\mathcal{A}$ is a reduced standard graded $R$-algebra such that its minimal primes contract to minimal primes of $R$. Then 
 \begin{equation}\label{vanishing of volumes}
 \varepsilon(\mathcal{A}) = 0
 \end{equation}
 if and only if $\codim c_{\mathcal{A}}^{-1}(\mathfrak{m}) \geq 2$. 
\end{theorem}
\begin{proof}
As in the proof of Cor.\ \ref{epsilon as limit} there exists a $\mathcal{B}$ with the properties prescribed in Prp.\ \ref{nice embedding} such that 
(\ref{zero and one}) holds. 
Assume $\codim c_{\mathcal{A}}^{-1}(\mathfrak{m}) \geq 2$. Denote by  $\overline{\mathcal{A}_n}$ the integral closure of $\mathcal{A}_n$ in $\mathcal{B}_n$. Note that $\overline{\mathcal{A}}=\oplus_{n=0}^{\infty} \overline{\mathcal{A}_n}$ with $\overline{\mathcal{A}_0}=R$ by \cite[Prp.\ 2.3.5]{Huneke}.

By Prp.\ \ref{associated primes} \rm{(i)} $\mathfrak{m} \not \in \mathrm{Ass}_{R}(\mathcal{B}_n/\overline{\mathcal{A}_n})$. Thus $H_{\mathfrak{m}}^{0}(\mathcal{B}_n/\mathcal{A}_n) \subset \overline{\mathcal{A}_n}/\mathcal{A}_n$ for each $n$. But $R$ is Nagata. Then so is $\mathcal{B}$. Furthermore, $\mathcal{B}$ is reduced, because $\mathcal{A}$ is. So, $\overline{\mathcal{A}}$ is module-finite over $\mathcal{A}$. Therefore, by repeating the argument from the proof of Prp.\ \ref{intr. of limits} $\rm{(ii)}$ we obtain (\ref{vanishing of volumes}). 

Conversely, assume that (\ref{vanishing of volumes}) holds. Suppose $b$ is an element from $\mathcal{B}_k$ for some $k$ such that there exists positive $l$ with $\mathfrak{m}^{l}b_k \in \overline{\mathcal{A}_k}$ and $b_k \not \in \overline{\mathcal{A}_k}$. Set $G':= \bigoplus_{i=1}^{\infty} \mathcal{A}_{ik}$ and let $G$ be the algebra generated by $G$ and $b_k$. By shift in degrees assume $G'$ and $G$ are standard graded. Then by \cite[Cor.\ 5.10 ]{KT-Al} 
$$\lambda_{R}(G_n/G'_n):= e(G',G)n^r/r! + \cdots$$
where the Buchsbaum--Rim multiplicity $e(G',G)$ is a positive integer.  Thus 
$$\limsup_{n \to \infty} \frac{r!}{(nk)^r}\lambda_{R}(H_{\mathfrak{m}}^{0}(\mathcal{B}_{nk}/\mathcal{A}_{nk})) \geq \lim_{n \to \infty}\frac{r!}{n^r}\lambda_{R}(G_n/G'_n) = e(G',G)>0$$
which contradicts our assumption. Hence $\mathfrak{m} \not \in \bigcup_{n=1}^{\infty}\mathrm{Ass}_R(\mathcal{B}_n/\overline{\mathcal{A}_n})$ which by reasons of grading is equivalent to $\mathfrak{m} \not \in \bigcup_{n=1}^{\infty}\mathrm{Ass}_R(\mathcal{B}/\overline{\mathcal{A}})$. So by Prp.\ \ref{associated primes} \rm{(ii)} $\codim c_{\mathcal{A}}^{-1}(\mathfrak{m}) \geq 2$ in $\mathrm{Proj}(\mathcal{A})$. 
\end{proof}
The forward direction of Thm.\ \ref{vanishing} was proved for the  epsilon multiplicity $\varepsilon (\mathcal{A}|\mathcal{B})$ with the additional hypothesis that $\mathcal{B}$ is $R$-flat by Ulrich and Validashti (see \cite[Thm.\ 4.2]{Validashti}).
We record two observations about  primary decomposition and local cohomology that will be used in the proof of the Local Volume Formula in the next section. Their statements and proofs can be found in \cite[Chp.\ 18]{Altman} for example. 

\begin{proposition}\label{saturation} Let $X$ be an affine Noetherian  scheme and let $S$ be a subscheme of $X$ with an ideal sheaf $\mathcal{I}_S$ in $\mathcal{O}_X$. Assume $\mathcal{M}$ and $\mathcal{F}$  are coherent $\mathcal{O}_{X}$-modules with $\mathcal{M}\subset \mathcal{F}$.
Let $\mathcal{M} = \cap \mathcal{M}_{i}$ be a primary  decomposition of $\mathcal{M}$ in $\mathcal{F}$ with
$x_{i} = \mathrm{Ass}_{X}(\mathcal{F}/\mathcal{M}_{i})$. The following holds

$$H_{S}^{0}(\mathcal{F}/\mathcal{M}) = (\cap_{\{j|x_j \not \subset S\}}\mathcal{M}_{j})/\mathcal{M}.$$
\end{proposition}

\begin{proposition}\label{maximality of saturation} Let $X$ be an affine Noetherian scheme, and let $x$ be a point in $X$. Consider the nested chain of coherent $\mathcal{O}_X$-modules $\mathcal{M} \subset \mathcal{N}\subset \mathcal{F}$. Assume that $\mathcal{M}_{z} = N_{z}$ for every point $z \notin \overline{\{x\}}$. Then
$$\mathcal{N}/\mathcal{M} \subset H_{x}^{0}(\mathcal{F}/\mathcal{M}).$$
Furthermore, $H_{x}^{0}(\mathcal{F}/\mathcal{M})$ is the largest submodule of $\mathcal{F}/\mathcal{M}$ equal to $\mathcal{M}$ locally off $\overline{\{x\}}$.
\end{proposition}

Sometimes it will be convenient to write the local cohomology in terms of saturations: 
$H_{S}^{0}(\mathcal{F}/\mathcal{M})=\cup_{n=0}^{\infty}(\mathcal{M} :_{\mathcal{F}} \mathcal{I}_{S}^{n})/\mathcal{M}$ and
$H_{x}^{0}(\mathcal{F}/\mathcal{M})=\cup_{n=0}^{\infty}(\mathcal{M}:_{\mathcal{F}} \mathcal{I}_{x}^{n})/\mathcal{M}.$

\section{The Local Volume Formula}\label{sec. main result}
In this section we prove the Local Volume Formula (LVF). We show that the restricted local volume exists as a limit.

Let $X$  and $Y$ be affine reduced schemes of finite type over a field $\Bbbk$. Assume $X$ is equidimensional and $Y$ is regular and integral of dimension one. Suppose $h \colon X \rightarrow Y$ is a morphism with equidimensional fibers of positive dimension. Let $S$ be a subscheme of $X$ that is quasi-finite over $Y$. Assume that $\Bbbk$ is the residue field of each closed point $y$ and the points in $S_{y}$. Let $C$  be an equidimensional reduced scheme projective over $X$ such that the structure morphism $c \colon C \rightarrow X$  maps each irreducible component of $C$ to an irreducible component of $X$. Set $D:= c^{-1}S$ and $\dim C = r+1$. Fix a closed point $y_0$ in $Y$ such that $S_{y_0}$ is nonempty. 
For simplicity we will assume that $y_0$ and the points in $S_{y_0}$ are rational. For an affine neighborhood $U$ of $y_0$ set $U':=U-\{y_0\}$. Denote by $D_{\mathrm{vert}}$ the union of components of $D$ that maps to $y_0$ under $h \circ c$.

Let $\mathcal{L}$ be an invertible very ample sheaf on $C$ relative to $X$. Let $\mathcal{A}:= \oplus_{n \geq 0} \Gamma (C, \mathcal{L}^{\otimes n})$ be the ring of sections of $\mathcal{L}$. Denote by $\mathcal{A}_n$ the $n$th graded piece of $\mathcal{A}$. Suppose $\dim X_y \geq 2$. Recall that restricted local volume of $\mathcal{L}$ at $S_y$ is defined as

$$
\mathrm{vol}_{C_y}(\mathcal{L}):= \limsup_{n\to \infty} \frac{r!}{n^r} \dim_{\Bbbk}H_{S}^{1}(\mathcal{A}_n)\otimes_{\mathcal{O}_Y} k(y).
$$

In the definition of the restricted local volume we can assume that $\mathcal{A}$ is the coordinate ring of $C$ as the $n$th graded component of the coordinate ring and $\Gamma (C, \mathcal{L}^{\otimes n})$ coincide for $n \gg 0$ (cf.\ Chp.\ II Ex.\ 5.9 in \cite{Hartshorne}). 

Let $\mathcal{B}$ be a standard graded algebra containing $\mathcal{A}$ satisfying the properties listed in Prp.\ \ref{nice embedding}. Denote by $\overline{\mathcal{A}}$ the integral closure of $\mathcal{A}$ in $\mathcal{B}$. Set $\mathcal{B}(y):=\mathcal{B} \otimes_{\mathcal{O}_X} \mathcal{O}_{X_y}$ and $\mathcal{B}_n(y):=\mathcal{B}_n \otimes_{\mathcal{O}_X} \mathcal{O}_{X_y}$. Denote the images of $\mathcal{A}$ in $\mathcal{B}(y)$ and $\mathcal{A}_n$ in $\mathcal{B}_n (y)$ by  $\mathcal{A}(y)$ and $\mathcal{A}_n(y)$, respectively. We say that the restricted volume {\it specializes with passage to the fiber} $X_y$ if 
$$
\mathrm{vol}_{C_y}(\mathcal{L}):= \limsup_{n\to \infty} \frac{r!}{n^r} \dim_{\Bbbk}H_{S_y}^{1}(\mathcal{A}_n(y)).
$$

\begin{theorem}[Local Volume Formula]\label{LVF} The following holds. 
\begin{enumerate}
\item[\rm{(i)}] Suppose $\dim X=2$ and $S:=\mathrm{Supp}_{\mathcal{O}_X}(\mathcal{B}/\overline{\mathcal{A}})$. Then there exists a an affine neighborhood $U$ of $y_0$ in $Y$ such that
$$e(\mathcal{A}(y_0),\mathcal{B}(y_0))-e(\mathcal{A}(y),\mathcal{B}(y))= \int_{S_{y_0}} l^{r}[D_{\mathrm{vert}}]$$
for each $y \in U'$. 
\item[\rm{(ii)}] Suppose $\dim X \geq 3$. Then there exists an affine neighborhood $U$ of $y_0$ in $Y$ such that $$\mathrm{vol}_{C_{y_0}}(\mathcal{L})- \mathrm{vol}_{C_{y}}(\mathcal{L}) =  \int_{S_{y_0}} l^{r}[D_{\mathrm{vert}}]$$
for each $y \in U'$ and the restricted local volume specializes with passage to $X_y$. Furthermore, $\mathrm{vol}_{C_{y}}(\mathcal{L})$ is finite and exists as a limit for each $y \in U$.
\end{enumerate}
\end{theorem}

\vspace{.2cm}
\begin{center}
{\it Proof}
\end{center}
\vspace{.2cm}



Consider $\rm{(i)}$. Denote by $c_{\mathcal{B}} \colon \mathrm{Proj}(\mathcal{B}) \rightarrow X$ the structure morphism. Because $\mathcal{B}$ is a birational extension of $\mathcal{A}$, it follows that the irreducible components of $\mathrm{Proj}(\mathcal{B})$ surject onto those of $X$.  By Prp.\ \ref{nice embedding} $\rm{(ii)}$ and the fact that $X$ and $X_y$ are equidimensional, and $Y$ is regular, it follows that the irreducible  components of $\mathrm{Proj}(\mathcal{B}(y_0))$ surject onto those of $X_{y_0}$. So no irreducible component of $\mathrm{Proj}(\mathcal{B}(y_0))$ is supported over $S_{y_0}$. As $\mathcal{B}(y_0)$ is integral over $\mathcal{A}(y_0)$ locally off $S_{y_0}$ it follows that $c^{-1}X_{y_0} = \mathrm{Proj}(\mathcal{A}(y_0)) \cup D_{\mathrm{vert}}$ (cf.\ Prp.\ \ref{residual ideal}). Applying Thm.\ \ref{GMPT} we get the desired result.

Consider  $\rm{(ii)}$. We break the proof of the LVF into several parts. Because $\dim X \geq 3$ using the analysis in Prp.\ \ref{intr. of limits}, we can replace in the LVF $H_{S}^{1}(\mathcal{A}_n)$  by $H_{S}^{0}(\mathcal{B}_n/\mathcal{A}_n).$ In Prp.\ \ref{initial version} for each $n$ we relate the change of the dimension of $H_{S}^{0}(\mathcal{B}_n/\mathcal{A}_n)\otimes_{\mathcal{O}_Y} k(y_0)$ 
as $y$ ``moves'' from $y_0$ to a generic $y \in U'$,  to the dimension of $H_{S_{y_0}}^{0}(\mathcal{B}_n/\mathcal{A}_n)\otimes_{\mathcal{O}_Y} k(y_0)$. Here the main algebraic tool we use is the well-known fact that a torsion-free module over a principal ideal domain is free. Then we relate the limit of normalized vector space dimensions of $H_{S_{y_0}}^{0}(\mathcal{B}_n/\mathcal{A}_n)\otimes_{\mathcal{O}_Y} k(y_0)$ to the degree of the cycle $[c^{-1}(S_{y_0})]_r$ using two reductions. The first relates the dimensions of $H_{S_{y_0}}^{0}(\mathcal{B}_n/\mathcal{A}_n)\otimes_{\mathcal{O}_Y} k(y_0)$ to the dimension of the $n$th graded piece of the ideal defining the residual scheme in $c^{-1}X_{y_0}$ to the union of components that surject onto $S_{y_0}$. The second one relates the limit of normalized vector space dimensions of the graded pieces of the ideal of the residual scheme to $\int_{S_{y_0}} l^{r}[D_{\mathrm{vert}}]$. Finally, we prove that the formation of the ``generic'' limit term in the LVF specializes with passage to generic fibers using in \cite[Prp.\ 2.1] {Ran19a}, which along with the LVF implies that $\mathrm{vol}_{C_{y}}(\mathcal{L})$ is finite and exists as a limit for each $y \in U$.

Let $\mathcal{P}_n$ and $\mathcal{N}_n$ be submodules of $\mathcal{B}_n$ such that
\begin{displaymath}
H_{S}^{0}(\mathcal{B}_n/\mathcal{A}_n)=\mathcal{N}_{n}/\mathcal{A}_n \ \text{and} \ H_{S_{y_0}}^{0}(\mathcal{B}_n/\mathcal{A}_n) = \mathcal{P}_{n}/\mathcal{A}_n.
\end{displaymath}
Let $V$ be an affine neighborhood in $X$ containing $S_{y_0}$. Denote by $A[V]$ and $A[U]$ the homogeneous coordinate rings of $V$ and $U$.  Because $S \rightarrow Y$ is quasi-finite and $S_{y_0}$ is nonempty, then by shrinking $U$ if necessary, we can assume that that $S \rightarrow Y$ is finite.
Note that the support of each of the quotients $\mathcal{N}_n/\mathcal{P}_n$, $\mathcal{N}_n/\mathcal{A}_n$ and $\mathcal{P}_n/\mathcal{A}_n$ is in $S$. Because $S$ is finite over $Y$, the direct image of each of these quotient by $h$ is a coherent $\mathcal{O}_Y$-module. 

\vspace{.2cm}
\begin{center}
{\it Flatness}
\end{center}
 
\begin{proposition}\label{initial version} Assume $U$ is small enough so that for each $n$, the associated points of  $\mathcal{B}_n/\mathcal{A}_n$ viewed as $A[V]$-module map to $y_0$ or the generic point of $U$. Then
\begin{equation}\label{initial eq.}
\dim_{\Bbbk} (\mathcal{N}_n/\mathcal{A}_n) \otimes_{\mathcal{O}_Y} k(y_0)- \dim_{\Bbbk}(\mathcal{N}_n/\mathcal{A}_n)\otimes_{\mathcal{O}_Y} k(y) = \dim_{\Bbbk}(\mathcal{P}_n/\mathcal{A}_n)\otimes_{\mathcal{O}_Y} k(y_0).
\end{equation}
\end{proposition}
\begin{proof}
By \cite[Prp.\ 2.1]{Ran19a} there are finitely many points in $\bigcup_{i=0}^{\infty}\mathrm{Ass}_{X}(\mathcal{B}_i/\mathcal{A}_i)$. Each of them maps to a closed point in $Y$, or the generic point of $Y$. Thus we can select $U$ so that the only associated points of $\mathcal{B}_n/\mathcal{A}_n$ are those that map to $y_0$ or to the generic point of $U$. Therefore, the only associated points of $\mathcal{N}_n/\mathcal{A}_n$ viewed now as $A[U]$-module are $y_0$ or the generic point of $U$. 

Consider the nested chain of modules
$$\mathcal{A}_{n} \subset \mathcal{P}_{n} \subset \mathcal{N}_{n}$$
Form the exact sequence
\begin{equation}\label{snake}
\begin{CD}
0    @>>> \mathcal{P}_{n}/\mathcal{A}_{n} @>>> \mathcal{N}_{n}/\mathcal{A}_{n} @>>> \mathcal{N}_{n}/\mathcal{P}_{n} @>>> 0\\
@.   @VV\mu V   @VV\mu V   @VV\mu V   @. \\
0 @>>> \mathcal{P}_{n}/\mathcal{A}_{n} @>>> \mathcal{N}_{n}/\mathcal{A}_{n} @>>> \mathcal{N}_{n}/\mathcal{P}_{n} @>>> 0
\end{CD}
\end{equation}
where the vertical maps $\mu$ are multiplication by the ideal $\mathfrak{m}_{y_0}$ of $y_0$. Since $\mu$ is injective on $\mathcal{N}_n/\mathcal{P}_n$, then the Snake Lemma yields
\begin{equation}\label{special ex. seq}
\begin{CD}
0    @>>> \mathcal{P}_{n}/\mathcal{A}_{n}\otimes_{\mathcal{O}_Y} k(y_0) @>>> \mathcal{N}_{n}/\mathcal{A}_{n}\otimes_{\mathcal{O}_Y} k(y_0) @>>> \mathcal{N}_{n}/\mathcal{P}_{n}\otimes_{\mathcal{O}_Y} k(y_0)@>>> 0.
\end{CD}
\end{equation}
Therefore,
\begin{equation}\label{snake special}
\dim_{\Bbbk} \mathcal{N}_{n}/\mathcal{A}_{n}\otimes_{\mathcal{O}_Y} k(y_0) = \dim_{\Bbbk} \mathcal{P}_{n}/\mathcal{A}_n\otimes_{\mathcal{O}_Y} k(y_0) + \dim_{\Bbbk}\mathcal{N}_{n}/\mathcal{P}_{n}\otimes_{\mathcal{O}_Y} k(y_0).
\end{equation}

Now let $y \in Y$ be a point from $U'$. Then
\begin{equation}
\begin{CD}
\mathcal{P}_{n}/\mathcal{A}_n\otimes_{\mathcal{O}_Y} k(y) @>>> \mathcal{N}_{n}/\mathcal{A}_n\otimes_{\mathcal{O}_Y} k(y) @>>> \mathcal{N}_{n}/\mathcal{P}_{n}\otimes_{\mathcal{O}_Y} k(y)@>>> 0.
\end{CD}
\end{equation}
However, $\mathcal{P}_{n}/\mathcal{A}_n\otimes_{\mathcal{O}_Y} k(y)=0$ as the support of $\mathcal{P}_{n}/\mathcal{A}_n$ is $S_{y_0}$. Thus, $$\mathcal{N}_{n}/\mathcal{A}_n\otimes_{\mathcal{O}_Y} k(y) \cong \mathcal{N}_{n}/\mathcal{P}_{n}\otimes_{\mathcal{O}_Y} k(y).$$ But the only associated point of the $A[U]$-module $\mathcal{N}_{n}/\mathcal{P}_{n}$ is the generic point of $U$, hence $\mathcal{N}_{n}/\mathcal{P}_{n}$ is torsion-free $\mathcal{O}_{Y,y}$ module for each closed point $y \in U$. 
Because 
$\mathcal{O}_{Y,y}$ is a DVR, and $Y$ is reduced, then  $\mathcal{N}_{n}/\mathcal{P}_{n}$ is locally free. Hence
$\dim_{\Bbbk} \mathcal{N}_{n}/\mathcal{P}_n\otimes_{\mathcal{O}_Y} k(y) = \dim_{\Bbbk}\mathcal{N}_{n}/\mathcal{P}_{n}\otimes_{\mathcal{O}_Y} k(y_0)$ (cf.\ Ex.\ 5.8 in Chp.\ II in \cite{Hartshorne}). So, 
\begin{equation}\label{snake generic}
\dim_{\Bbbk}\mathcal{N}_{n}/\mathcal{A}_n\otimes_{\mathcal{O}_Y} k(y)=\dim_{\Bbbk}\mathcal{N}_{n}/\mathcal{P}_n\otimes_{\mathcal{O}_Y} k(y_0)
\end{equation}
Subtracting (\ref{snake generic}) from (\ref{snake special}) we get the desired result.
\end{proof}

Let $x \in S_{y_0}$. Define $\mathcal{P}_{n}(x)$ so that $H_{x}^{0}(\mathcal{B}_n/\mathcal{A}_n)=\mathcal{P}_{n}(x)/\mathcal{A}_n$. Clearly, $$\dim_{\Bbbk}(\mathcal{P}_n/\mathcal{A}_n)\otimes_{\mathcal{O}_Y} k(y_0)= \sum_{x \in S_{y_0}}\dim_{\Bbbk}(\mathcal{P}_{n}(x)/\mathcal{A}_n)\otimes_{\mathcal{O}_Y} k(y_0).$$
Let $D_{\mathrm{vert}}(x)$ be the union of components of $D_{\mathrm{vert}}$ that map to $x$. We will show that $$\lim_{n\to \infty} \frac{r!}{n^r}\dim_{\Bbbk}(\mathcal{P}_{n}(x)/\mathcal{A}_n)\otimes_{\mathcal{O}_Y} k(y_0)=\int_{x} l^{r}[D_{\mathrm{vert}}(x)].$$
Thus we can assume that $S_{y_0}$ consists of a single point $x_0$. 

\vspace{.1cm}
\begin{center}
{\it Key Isomorphism}
\end{center}
\vspace{.1cm}

The next proposition provides a key isomorphism that allows to connect $\mathcal{P}_n/\mathcal{A}_n$ to the ideal of the residual scheme to the union of components of dimension $r$ in $c^{-1}(x_0)$.
Let $t \in \mathfrak{m}_{y_0}-\mathfrak{m}_{y_0}^2$ be a uniformizing parameter of $\mathcal{O}_{Y,y_0}$. Identify $t$ with its image $h^{\#}t$ in $\mathfrak{m}_{x_0}$. Replace $\mathcal{P}_n$ and $\mathcal{A}_n$ by their localizations at $x_0$. This does not affect the length of $\mathcal{P}_n/\mathcal{A}_n \otimes_{\mathcal{O}_Y} k(y_0)$. 
\begin{proposition}\label{key isomorphism} We have
\begin{equation}
\mathcal{P}_{n}/\mathcal{A}_{n}\otimes_{\mathcal{O}_Y} k(y_0) \simeq (t\mathcal{P}_n \cap \mathcal{A}_{n})/t\mathcal{A}_n
\end{equation}
as $\Bbbk$-vector spaces.
\end{proposition}
\begin{proof}  To begin, observe that both quotients $\mathcal{P}_{n}/\mathcal{A}_{n}/t(\mathcal{P}_{n}/\mathcal{A}_{n})$ and $(t\mathcal{P}_n \cap \mathcal{A}_{n})/t\mathcal{A}_n$ are supported over $x_0$. Because $k(y_0)=\Bbbk$ is contained in $\mathcal{O}_{X,x_0}$ and equals $k(x_0)$, then the two quotients are $\Bbbk$-vector spaces of finite dimension. Because $X$ and the fibers of $X \rightarrow Y$ are equidimensional, there is no component of $X$ supported over $X_{y_0}$. Also, $\mathrm{Proj}(\mathcal{B})$ is reduced and its components surject onto those of $X$. Thus $t$ is a nonzero divisor of $\mathcal{B}$. 

The isomorphism is built from the obvious correspondence: for $b \in \mathcal{P}_n$, not divisible by $t$ in $\mathcal{B}_n$, associate $t^{s}b \in t\mathcal{P}_n \cap \mathcal{A}_n$ where $s$ is the smallest integer such that $t^sb \in \mathcal{A}_n$ but $t^sb \not \in t\mathcal{A}_n$. Below we make this  precise.

For each $b$ in $\mathcal{P}_{n}$ define $\mathrm{ord}_{t}(b)$ to be the smallest integer such that $t^{\mathrm{ord}_{t}(b)}b \in \mathcal{A}_n$. Set
$s_n := \max \{\mathrm{ord}_{t}(b)| b \in \mathcal{P}_{n}\}$ and let $\mathcal{P}_{n}^{(i)}$ be the image of $(\mathcal{A}_n :_{\mathcal{P}^{n}} t^{i})$ in $\mathcal{P}_{n}/\mathcal{A}_{n}\otimes_{\mathcal{O}_Y} k(y_0)$. Let $L_{n}^{(i)}$ be submodule of $(t\mathcal{P}_n \cap \mathcal{A}_{n})/t\mathcal{A}_n$ generated by the images of all elements $a \in \mathcal{A}_n$ such that $a/t^i \in \mathcal{P}_n$. Consider the following filtrations

$$\mathcal{P}_{n}/\mathcal{A}_{n}\otimes_{\mathcal{O}_Y} k(y_0)=:\mathcal{P}_{n}^{(s_n)} \supset \mathcal{P}_{n}^{(s_{n}-1)} \supset \cdots \supset \mathcal{P}_{n}^{(1)} \supset 0$$
and
$$0 \subset L_{n}^{(s_n)} \subset \cdots \subset L_{n}^{(1)}:=(t\mathcal{P}_n \cap \mathcal{A}_{n})/t\mathcal{A}_n.$$
To prove the isomorphism in question, it's enough to show 
$$\mathcal{P}_{n}^{(i)}/\mathcal{P}_{n}^{(i-1)} \simeq L_{n}^{(i)}/L_{n}^{(i+1)}$$
as $\Bbbk$-vector spaces for each $i=1, \ldots s_n$, where we assume that $\mathcal{P}_{n}^{(0)}=L_{n}^{(s_n+1)}=0$. Let $\hat{b_{i}}$ be an element in $\mathcal{P}_{n}^{(i)}/\mathcal{P}_{n}^{(i-1)}$ and $b_{i}$ be an element $(\mathcal{A}_n :_{\mathcal{P}^{n}} t^{i})$ that maps to it. Denote the image of $t^{i}b_{i}$ in $L_{n}^{(i)}/L_{n}^{(i+1)}$ by $\widetilde{t^{i}b_{i}}$. Define the map $\psi_{i}(\hat{b_{i}}) = \widetilde{t^{i}b_{i}}$. To see that the map is well-defined suppose there exists $b_{i}'\in (\mathcal{A}_n :_{\mathcal{P}^{n}} t^{i})$ such that $\hat{b_{i}'}=\hat{b_{i}}$. Then $b_i-b_{i}' \in (\mathcal{A}_n :_{\mathcal{P}^{n}} t^{i-1})$ and so $t^{i}(b_i-b_{i}') \in t\mathcal{A}_n$ which implies that $\widetilde{t^ib_i}=\widetilde{t^ib_{i}'}$.
The map $\psi_i$ is a $\Bbbk$-linear because the composition $(\mathcal{A}_n :_{\mathcal{P}^{n}} t^{i}) \rightarrow t\mathcal{P}_n \cap \mathcal{A}_n \rightarrow (t\mathcal{P}_n \cap \mathcal{A}_{n})/t\mathcal{A}_n$ is a $\Bbbk$-homomorphism. Also, $\psi_{i}$ is obviously surjective. Suppose  $\psi_{i}(\hat{b_i})=0$. Then $t^ib_{i} \in t\mathcal{A}_n$. As $t$ is nonzero divisor in $\mathcal{B}_n$ it follows that $t^{i-1}b_{i} \in \mathcal{A}_n$, or $b_{i} \in (\mathcal{A}_n :_{\mathcal{P}^{n}} t^{i-1})$ which implies that $\hat{b_i}=0$. Thus $\psi_i$ is an isomorphism. 
\end{proof}

\begin{center}
{\it The Residual Scheme}
\end{center}
\vspace{.3cm}

The next proposition identifies the ideal of the residual scheme to the union of components of $c^{-1}(X_{y_0})$ that surject onto $x_0$. In what follows we will replace $\mathcal{B}$ with its integral closure in its total ring of fractions and redefine $\mathcal{P}_n$ and $\mathcal{N}_n$ accordingly. By the analysis in Prp.\ \ref{intr. of limits} $\rm{(ii)}$ applied to the family setting, the terms participating in the LVF will be unaffected. 

Let $V$ be the union of the components of $c^{-1}(X_{y_0})$ whose projection to $X_{y_0}$ is $x_0$. Let $W:=(c^{-1}(X_{y_0})-V)^{-}$ be the residual scheme of $V$ in $c^{-1}(X_{y_0})$. As the problem is local at $x_0$, assume $X_{y_0}$ is local with closed point $x_0$. Denote by $A[C_{y_0}]$ the homogeneous coordinate ring of $c^{-1}(X_{y_0})$. Set $\mathcal{P} := \mathcal{O}_{X,x_0} \oplus \mathcal{P}_1 \oplus \mathcal{P}_2 \oplus \cdots$

\begin{proposition}\label{residual ideal} The ideal $\mathcal{I}_W$ of $W$ in $A[C_{y_0}]$ is  $(t\mathcal{P} \cap \mathcal{A})/t\mathcal{A}.$
\end{proposition}
\begin{proof} Since $A[C_{y_0}] =\mathcal{A}/t\mathcal{A}$, then by Prp.\ \ref{saturation}
$\mathcal{I}_W = \bigcup_{i=0}^{\infty}(t\mathcal{A} :_{\mathcal{A}} \mathfrak{m}_{x_0}^{i})/t\mathcal{A}$
where $\mathfrak{m}_{x_0}$ is the ideal of $x_0$ in $\mathcal{O}_{X,x_0}$. We claim that
$$\bigcup_{i=0}^{\infty}(t\mathcal{A} :_{\mathcal{A}} \mathfrak{m}_{x_0}^{i}) = t\mathcal{P}\cap \mathcal{A}.$$
Indeed, let $tp \in t\mathcal{P} \cap \mathcal{A}$ where $p \in \mathcal{P}$. Let $j \gg 0$ such that
$\mathfrak{m}_{x_0}^{j}p \in \mathcal{A}$. Then $\mathfrak{m}_{x_0}^{j}tp \in t\mathcal{A}$ and hence $tp \in \bigcup_{i=0}^{\infty}(t\mathcal{A} :_{\mathcal{A}} \mathfrak{m}_{x_0}^{i})$.

As in the proof of Prp.\ \ref{intr. of limits} $\rm{(ii)}$ we can find $\mathfrak{z}$ such that $t,\mathfrak{z}$  
is a $2$-regular sequence for $\mathcal{B}$. Let $q$ be an element from $\bigcup_{i=0}^{\infty}(t\mathcal{A} :_{\mathcal{A}} \mathfrak{m}_{x_0}^{i})$. For $j \gg 0$ we have $\mathfrak{m}_{x_0}^{j}q \in t\mathcal{A}$ so we can write $\mathfrak{z}^{j}q=ta$ for $a \in \mathcal{A}$. Consider the last identity as an identity in $\mathcal{B}$. Because $\mathfrak{z}$ is a nonzero divisor modulo $t$, then $q/t \in \mathcal{B}$. We claim that $q/t \in \mathcal{P}$. Indeed, because $\mathfrak{m}_{x_0}^{j}q \in t\mathcal{A}$ we have $\mathfrak{m}_{x_0}^{j}q/t \in \mathcal{A}$ which yields $q/t \in \mathcal{P}$. Hence $q \in t\mathcal{P} \cap \mathcal{A}$.
\end{proof}

Let $Z:= \mathrm{Spec}(R)$ be an affine Noetherian scheme of finite type over a field $K$. Let $C$ be a subscheme of $Z \times \mathbb{P}_{K}^{u}$, where $u$ is a positive integer. Denote the homogeneous coordinate ring of $C$ by $A[C]$. It is a graded ring with respect to the coordinates of $\mathbb{P}_{K}^{u}$. Denote its $n$th graded piece by $A[C]_n$. Let $W$ be a closed subscheme of $C$  and denote its ideal in $A[C]$ by $\mathcal{I}_W$. Set $(\mathcal{I}_W)_n = A[C]_n \cap \mathcal{I}_W$. 

Let $\mathrm{pr}_1$ be the projection from $Z\times \mathbb{P}_{K}^{u}$ onto the first factor. Suppose $z$ is a rational point of $Z$. Set $V: = \mathrm{pr}_{1}^{-1}(z) \cap C$ and $r: = \dim V$. Let $l:=c_{1}\mathcal{O}_{C}(1)$ and let $[V]_r$ be the dimension $r$ part of the fundamental cycle of $V$. Finally, define $\deg [V]_r:=\int l^{r}V$. 

\begin{proposition}\label{residual-degree}
Assume $C= V \cup W$ where $W$ is a closed subscheme of $Z \times \mathbb{P}_{K}^{u}$ with $\dim (V\cap W)<r$. Then $(\mathcal{I}_W)_n$ is a finite-dimensional $K$-vector space for each $n$, and
$$\lim_{n \to \infty} \frac{r!}{n^{r}} \dim_{K} (\mathcal{I}_W)_{n}=\deg [V]_r.$$
\end{proposition}


\begin{proof}
Because $\mathcal{I}_V \cap \mathcal{I}_W =0$ in $A[C]$, then as $R$-modules
\begin{equation}\label{Noether first}
(\mathcal{I}_W)_n \simeq ((\mathcal{I}_W)_n+(\mathcal{I}_V)_n)/(\mathcal{I}_V)_n
\end{equation}
for each $n$. Set $A[V]=A[C]/\mathcal{I}_V$. Since the residue field of $x$ is $K$, it follows that $A[V]_n$ is finite-dimensional $K$-vector space. The inclusion $((\mathcal{I}_W)_n+(\mathcal{I}_V)_n)/(\mathcal{I}_V)_n \subset A[V]_n$
shows that $(\mathcal{I}_W)_n+(\mathcal{I}_V)_n)/(\mathcal{I}_V)_n$ is finite dimensional $K$-vector space as well. But $R$ is a $K$-algebra. Thus, by (\ref{Noether first}) the $R$-module $(\mathcal{I}_W)_n$ is a finite-dimensional $K$-vector space and

\begin{equation}\label{dim. Noether first}
\dim_{K}(\mathcal{I}_W)_n=\dim_{K}((\mathcal{I}_W)_n+(\mathcal{I}_V)_n)/(\mathcal{I}_{V})_n.
\end{equation}
Next, consider the exact sequence
$0 \rightarrow (\mathcal{I}_W +\mathcal{I}_V)/\mathcal{I}_{V} \rightarrow A[V] \rightarrow A[V\cap W] \rightarrow 0$
where $A[V\cap W]= A[C]/\mathcal{I}_{V \cap W}$. 
As $\dim (V \cap W) <r$, we get $\lim_{n \to \infty} \frac{r!}{n^{r}} \dim_K A[V\cap W]_n =0.$
Thus
\begin{equation}\label{intermidiate isom.}
\lim_{n \to \infty} \frac{r!}{n^{r}} \dim_K ((\mathcal{I}_W)_n+(\mathcal{I}_V)_n)/(\mathcal{I}_V)_n=\lim_{n \to \infty} \frac{r!}{n^{r}} \dim_K A[V]_n.
\end{equation}
But
\begin{equation}\label{def. degree}
\lim_{n \to \infty} \frac{r!}{n^{r}} \dim_K A[V]_n = \deg[V]_r.
\end{equation}
Combining (\ref{dim. Noether first}), (\ref{intermidiate isom.}) and (\ref{def. degree}) we get the desired result.
\end{proof}
\begin{proposition}\label{lim-degree relation}
We have
$$\lim_{n \to \infty} \frac{r!}{n^{r}} \dim_{\Bbbk}\mathcal{P}_n/\mathcal{A}_n \otimes_{\mathcal{O}_{Y}} k(y_0) = \int_{S_{y_0}} l^{r}[D_{\mathrm{vert}}].$$
\end{proposition}
\begin{proof} By Prp. \ref{key isomorphism}
$\dim_{\Bbbk}\mathcal{P}_n/\mathcal{A}_n \otimes_{\mathcal{O}_{Y}} k(y_0) = \dim_{\Bbbk} (t\mathcal{P}_n \cap \mathcal{A}_n)/t\mathcal{A}_n.$ By Prp.\ \ref{residual ideal} we have $(\mathcal{I}_W)_n =(t\mathcal{P}_n \cap \mathcal{A}_n)/t\mathcal{A}_n$. Therefore,
$\dim_{\Bbbk} \mathcal{P}_n/\mathcal{A}_n \otimes_{\mathcal{O}_{Y}} k(y_0) = \dim_{\Bbbk}(\mathcal{I}_W)_n$.
Finally, Prp.\ \ref{residual-degree} applied with $Z=X_{y_0}$ and $z=x_0$, $K=\Bbbk$ and $C=c^{-1}(X_{y_0})$ give the desired equality.
\end{proof}

\vspace{.1cm}
\begin{center}
{\it Specialization to the generic fiber}
\end{center}
\vspace{.1cm}

\begin{proposition}\label{generic limit} Let $U$ be a sufficiently small neighborhood of $y_0$. Then 
$$H_{S}^{0}(\mathcal{B}_n/\mathcal{A}_n)\otimes_{\mathcal{O}_Y}
k(y)=H_{S_y}^{0}((\mathcal{B}_n/\mathcal{A}_n)\otimes_{\mathcal{O}_Y} k(y))$$
for  $y \in U'$ and each $n$. 
\end{proposition}
\begin{proof}
As usual identify $H_{S}^{0}(\mathcal{B}_n/\mathcal{A}_n)$ with $\mathcal{N}_n/\mathcal{A}_{n}$. We want to show that $$\mathcal{N}_n/\mathcal{A}_n \otimes_{\mathcal{O}_Y} k(y) = H_{S_y}^{0}((\mathcal{B}_n/\mathcal{A}_n)\otimes_{\mathcal{O}_Y} k(y))$$ for each $y \in U$. First, because the support of $\mathcal{N}_n/\mathcal{A}_n \otimes_{\mathcal{O}_Y} k(y)$ is in $S_y$, then by Prp.\ \ref{maximality of saturation} it follows that $\mathcal{N}_n/\mathcal{A}_n \otimes_{\mathcal{O}_Y} k(y) \hookrightarrow H_{S_y}^{0}((\mathcal{B}_n/\mathcal{A}_n)\otimes_{\mathcal{O}_Y} k(y))$. Suppose $\mathcal{N}_n'$ is such that
$$\mathcal{N}_n'/\mathcal{A}_n \otimes_{\mathcal{O}_Y} k(y)=H_{S_y}^{0}((\mathcal{B}_n/\mathcal{A}_n)\otimes_{\mathcal{O}_Y} k(y)).$$ 

By  \cite[Prp.\ 2.1]{Ran19a} $\bigcup _{n=1}^{\infty}\mathrm{Ass}_{X}(\mathcal{B}_n/\mathcal{A}_n)$ is finite (we have better qualitative result if instead we appeal to \cite[Thm.\ 1.4]{Ran19b} which describes the associated primes $\mathrm{Ass}_{X}(\mathcal{B}_n/\overline{\mathcal{A}_n})$ as the generic points of the irreducible components where the fiber dimension of $c$ jumps (cf.\  \cite[Thm.\ 7.8]{Rangachev2}). Shrink $V$ and $U$ so that for each associated point $z \in \bigcup _{n=1}^{\infty}\mathrm{Ass}_{X}(\mathcal{B}_n/\mathcal{A}_n)$, the image $h(\overline{\{z\}})$ surjects onto $U$ or onto $y_0$. Then further shrink $U$ if necessary, so that each $\overline{\{z\}}$ of dimension one that is not a component of $S$ intersects $S$ at $x_0$ only. Set $S': =\mathrm{Supp}_{X}(\mathcal{N}_{n}'/\mathcal{A}_n)$. The support of $\mathcal{N}_n'/\mathcal{A}_n \otimes_{\mathcal{O}_Y} k(y)$ is $S_{y}'$. Also, the support of $H_{S_y}^{0}((\mathcal{B}_n/\mathcal{A}_n)\otimes_{\mathcal{O}_Y} k(y))$ is $S_y$. This forces $S_{y}'=S_y$. Therefore, the union of components of $S'$ that do not surject onto $x_0$ is of dimension one and hence equal to $S$. By Prp.\  \ref{maximality of saturation} $\mathcal{N}_n$ is the  maximal submodule of $\mathcal{B}_n$ that contains $\mathcal{A}_n$ with $\mathrm{Supp}(\mathcal{N}_n/\mathcal{A}_n)=S$. Thus $(\mathcal{N}_{n}'/\mathcal{N}_n)\otimes_{\mathcal{O}_Y} k(y)=0$.
\end{proof}

\vspace{.1cm}
\begin{center}
{\it Completing the proof of the LVF}
\end{center}
\vspace{.1cm}

Now we are in position to finish the proof of Thm.\ \ref{LVF}.
By Prp.\ \ref{generic limit} and by our hypothesis 
$$\limsup_{n \to \infty} \frac{r!}{n^{r}} \dim_{\Bbbk}H_{S}^{0}(\mathcal{B}_n /\mathcal{A}_n)\otimes_{\mathcal{O}_Y} k(y)$$
is finite for $y \in U'$. Therefore, by
Prp.\ \ref{initial version} and Prp.\ \ref{lim-degree relation}
$$\limsup_{n \to \infty} \frac{r!}{n^{r}} \dim_{\Bbbk}H_{S}^{0}(\mathcal{B}_n /\mathcal{A}_n)\otimes_{\mathcal{O}_Y} k(y_0)-\limsup_{n \to \infty} \frac{r!}{n^{r}} \dim_{\Bbbk}H_{S}^{0}(\mathcal{B}_n /\mathcal{A}_n)\otimes_{\mathcal{O}_Y} k(y)=\int_{S_{y_0}} l^{r}[D_{\mathrm{vert}}].$$
This completes the proof of the LVF. 
By Prp.\ \ref{generic limit} and Cor.\ \ref{epsilon as limit} $\mathrm{vol}_{C_{y}}(\mathcal{L})$ exists as a finite limit for $y \in U'$. In particular, $\mathrm{vol}_{C_{y}}(\mathcal{L})$ is the sum of the local volumes at each point of $S_y$. By the LVF $\mathrm{vol}_{C_{y_0}}(\mathcal{L})$ exists as a finite limit, too.
\qedsymbol


We conclude this section with two immediate applications of Thm.\ \ref{LVF}.
\begin{corollary}\label{upper semi-continuity} The restricted local volume $\mathrm{vol}_{C_{y}}(\mathcal{L})$
is constant for $y \in U'$.
\end{corollary}
\begin{proof} Indeed, both $\mathrm{vol}_{C_{y_{0}}}(\mathcal{L})$ and $\int_{S_{y_0}} l^{r}[D_{\mathrm{vert}}]$ remain constant as $y$ moves around in $U'$.
\end{proof}
The following is an immediate corollary from the proof of the LVF. 
\begin{corollary}
In the setup of Thm.\ \ref{LVF} the scheme $c^{-1}S$ is flat over $Y$ if and only if 
$\dim_{\Bbbk}H_{S}^{0}(\mathcal{B}_n /\mathcal{A}_n)\otimes_{\mathcal{O}_Y} k(y)$  is constant over $Y$ for $n \gg 0$. 
\end{corollary}
\begin{proof}
We have $\dim_{\Bbbk}H_{S}^{0}(\mathcal{B}_n /\mathcal{A}_n)\otimes_{\mathcal{O}_Y} k(y)$ is constant over $Y$ for $n \gg 0$ iff $(\mathcal{I}_W)_n=0$ for $n \gg 0$ iff $V$ is empty iff $c^{-1}S$ is flat (see \cite[Prp.\ 9.7]{Hartshorne}). 
\end{proof}


\begin{remark}\label{CA LVF} \rm{The proof of the LVF translates almost verbatim to the complex analytic setting of Rmk.\ \ref{CA}. As before one should work with sufficiently small distinguished compact Stein subsets of $X$ and $Y$ containing $S_{y_0}$ and $y_0$. The theory of local cohomology and in particular of gap-sheaves (saturations) of the corresponding analytic sheaves is developed in \cite{ST71}. The rest of the proof remains the same.}

\end{remark}
\begin{remark}\label{Fujita Rmk}
\rm{In \cite{Ran19d} we give a Fujita-type version of the LVF by replacing local cohomology with intersection numbers: using Thm.\ \ref{GMPT} it's shown that in the LVF $\mathrm{vol}_{C_{y}}(\mathcal{L})$ can be replaced by the BR multiplicity $e(\mathcal{A}_{n}(y),\mathcal{N}_{n}(y))/n^r$ for each $y$. More details can be found in the previous arXiv version of the current paper.}
\end{remark}

\section{Volume Stability and the Principle of Specialization of Integral Dependence}\label{stability}
In this section we show how to extend the LVF to the case $\dim Y >1$ under the assumption of volume stability. As an application we prove a general version of Teissier's Principle of Specialization of Integral Dependence. 

Let $X \rightarrow Y$ be a family of reduced complex analytic spaces such that $X$ is equidimensional, $Y$ is reduced and irreducible, and the fibers $X_y$ are equidimensional of dimension at least $2$. Fix a point $y_0 \in Y$. Assume $Y$ is contained in a reduced and irreducible analytic space $W$ and 
\begin{equation}\label{def. eq.}
 \mathcal{X} \longrightarrow W
\end{equation}
is an equidimensional reduced family with equidimensional fibers of positive dimension such that $\mathcal{X}_{y_0}=X_{y_0}$. Assume $S(W)$ is a subspace of $\mathcal{X}$ quasi-finite over $W$ with $S(W)_{y_0}$ nonempty. Let $c \colon C(W) \rightarrow \mathcal{X}$  be a reduced,  equidimensional and projective space over $\mathcal{X}$ such that its irreducible components surject onto irreducible components of $\mathcal{X}$ and such that for generic $w$ the fiber $C(W)_w$ is equidimensional of dimension $r$ with irreducible components surjecting onto those of $\mathcal{X}_w$. Let $\mathcal{L}$ be an invertible very ample sheaf on $C(W)$ relative to $X$. Set $\mathcal{A}:= \oplus_{n \geq 0} \Gamma (C(W), \mathcal{L}^{\otimes n})$. For each closed point $w$ define the restricted local volume of $\mathcal{L}$ at $S(W)_w$ as 

$$\mathrm{vol}_{C(W)_w}(\mathcal{L}):= \limsup_{n\to \infty} \frac{r!}{n^r} \dim_{\mathbb{C}}H_{S(W)}^{1}(\mathcal{A}_n)\otimes_{\mathcal{O}_W} k(w).
$$
Let $\mathcal{B}$ be a sheaf of graded $\mathcal{O}_{\mathcal{X}}$-algebras containing $\mathcal{A}$ such that locally $\mathcal{B}$ is a birational extension of $\mathcal{A}$ satisfying the properties listed in Prp.\ \ref{nice embedding}. Set $\mathcal{B}(w):=\mathcal{B} \otimes_{\mathcal{O}_{\mathcal{X}}} \mathcal{O}_{\mathcal{X}_{w}}$. For each $n$ denote the image of  $\mathcal{A}_n$ in $\mathcal{B}_n (w)$ by  $\mathcal{A}_{n}(w)$  We say that the restricted local volume specializes with passage to the fiber $\mathcal{X}_w$ if 
$$
\mathrm{vol}_{C(W)_w}(\mathcal{L}):= \limsup_{n\to \infty} \frac{r!}{n^r} \dim_{\mathbb{C}}H_{S(W)_w}^{1}(\mathcal{A}_n(w)).
$$
Replacing $Y$ by $W$ in Prp.\ \ref{generic limit} we get that there exists a Zariski open subset $U$ of $W$ such that the restricted local volume specializes with passage to $\mathcal{X}_w$ for each $w \in U$.

\begin{definition}
We say $W$ is a {\it good base space} for $\mathrm{vol}_{C(W)_w}(\mathcal{L})$ if there exists a Zariski open dense subset $U$ of $W$ such that $\mathrm{vol}_{C(W)_w}(\mathcal{L})$ is constant for all $w \in U$. In this case we say $\mathrm{vol}_{C(W)_w}(\mathcal{L})$ is stable. 
\end{definition}

Set $\mathcal{B}(Y):=\mathcal{B} \otimes_{\mathcal{O}_{\mathcal{X}}} \mathcal{O}_X$. Let $\mathcal{A}(Y)$ be the image of $\mathcal{A}$
in $\mathcal{B}(Y)$. Set $C:= \mathrm{Proj}(\mathcal{A}(Y))$ and $S:=S(W) \times_{W} Y$. Let $c_{X} \colon C \rightarrow X$ be the structure morphism. 

Let $W_1$ and $W_2$ be generic curves in $W$ passing through $y_0$ and a generic point $y$ in $Y$ close enough to $y_0$, respectively. Consider the families $\mathcal{X}(W_i):=\mathcal{X} \times_{W} W_i$.  Let $\mathcal{A}(W_i)$ be the image of $\mathcal{A}$
in $\mathcal{B} \otimes_{\mathcal{O}_{\mathcal{X}}}\mathcal{O}_{\mathcal{X}_{i}(W_i)}$. Set $C(W_i): = \mathrm{Proj}(\mathcal{A}(W_i))$ and $S(W_i):= S(W) \times_{W} W_i$. Denote by $D(W_1)_{\mathrm{vert}}$ the union of  components of the inverse image of $S(W_1)$  in $C(W_1)$ supported over $S(W)_{y_0}$, and by $D(W_2)_{\mathrm{vert}}$ the union of components of the inverse image of $S(W_2)$ in $C(W_2)$ supported over $S(W)_{y}$. 

\begin{theorem}\label{covering}
Suppose $W$ is a good base space for $\mathrm{vol}_{C(W)_w}(\mathcal{L})$. Assume that the set of points $x$ in $\mathcal{X}$ with  $\dim c^{-1}x \geq r$ is contained in $S(W)$. Then $\mathrm{vol}_{C(W_1)_{y_0}}(\mathcal{L})= \mathrm{vol}_{C(W_2)_{y}}(\mathcal{L})$ for $y \in Y$ near $y_0$ if and only if $\dim c_{X}^{-1}S_{y_0}<r$.

\end{theorem}
\begin{proof}
Select $W_1$ and $W_2$ generic enough so that $W_{1}-\{y_0\}$ and $W_{2}-\{y\}$ lie in the Zariski open subset $U$ over which the volume $\mathrm{vol}_{C(W)_w}(\mathcal{L})$ is stable and specializes with passage to $\mathcal{X}_w$. Let $w_1$ and $w_2$ be points from $W_1$ and $W_2$ close enough to $y_0$ and $y$, respectively. Apply the LVF to the families $\mathcal{X}(W_i) \rightarrow W_i$ to get

$$\mathrm{vol}_{C(W_1)_{y_0}}(\mathcal{L})-  \mathrm{vol}_{C(W_1)_{w_1}}(\mathcal{L}) = \int_{S_{y_0}} l^{r}[D(W_1)_{\mathrm{vert}}]$$
and
$$\mathrm{vol}_{C(W_2)_{y}}(\mathcal{L})-  \mathrm{vol}_{C(W_2)_{w_2}}(\mathcal{L}) = \int_{S_{y}} l^{r}[D(W_2)_{\mathrm{vert}}].$$

Because of volume stability and specialization (Prp.\ \ref{generic limit})
$\mathrm{vol}_{C(W_1)_{w_1}}(\mathcal{L})=\mathrm{vol}_{C(W_2)_{w_2}}(\mathcal{L})$.
Subtracting the second identity from the first we get 
\begin{equation}\label{curve subs}
\mathrm{vol}_{C(W_1)_{y_0}}(\mathcal{L})-\mathrm{vol}_{C(W_2)_{y}}(\mathcal{L})= \int_{S_{y_0}} l^{r}[D(W_1)_{\mathrm{vert}}]-\int_{S_{y}} l^{r}[D(W_2)_{\mathrm{vert}}].
\end{equation}
Our goal is to show that the right-hand side (\ref{curve subs}) equals $\int_{S_{y_0}} l^{r}[D_{\mathrm{vert}}].$ To do this we interpret each of the three intersection numbers as the local degrees of a covering of $W$.

Let $W(y_0)$ be a small enough neighborhood of $W$ around $y_0$. Let $y \in W(y_0) \cap Y$ be a generic point of $Y$ to be specified later, and let $W(y)$  be a small enough neighborhood of $y$ such that $W(y) \subset W(y_0)$.
Let  $C(W) \hookrightarrow \mathcal{X} \times \mathbb{P}^{u}$ be the embedding of $C$ induced by $\mathcal{L}$.

By the dimension formula $\dim C(W)= \dim W +r$. Let $H_r$ be a general plane in $\mathbb{P}^{u}$ of codimension $r$, with genericity conditions to be specified below. Set $\Gamma(C(W)):=C(W) \cap H_r$. Note that by replacing $\mathcal{X}$ with a smaller neighborhood of $S(W)_{y_0}$ if necessary, then $\Gamma(C(W))$ is nonempty if and only if $c^{-1}S(W)_{y_0} \geq r$. By Kleiman's transversality theorem (see Thm.\ 2 \rm{(ii)} and Rmk.\ 7 in \cite{K}),  $\Gamma(C(W))$ is reduced and of pure dimension equal to $\dim W$. Furthermore, the intersection of $c^{-1}S(W)$ with a general plane $H_{r}$ is of dimension at most $\dim W - 1$. Therefore, for generic $w \in W(y_0)$ the fiber $\Gamma(C(W))_{w}$ consists of the same number of points, each of them appearing with multiplicity one, and none of them lying in $c^{-1}S(W)$. Set $\mathrm{deg}_{W(y_0)}\Gamma(C(W)):= \mathrm{Card}(\Gamma(C(W))_{w})$. 


Define $\Gamma(C(W_1)):= C(W_1) \cap H_r$ and set $\mathrm{deg}_{W_1(y_0)}\Gamma(C(W_1)):= \mathrm{Card}(\Gamma(C(W_1))_{w_1})$ for generic $w_1 \in W_1$, where $W_{1}(y_0)$ is a small enough neighborhood of $y_0$ in $W_1$. Then for a generic $W_1$ we have $\Gamma(C(W))_{w_1}=\Gamma(C(W_1))_{w_1}$ for generic $w_1 \in W_1$. Because by assumption the only points in $X_{y_0}$ over which the fiber of $c$ is of dimension at least $r$ are contained in $S_{y_0}$. By replacing $W_1$ with a smaller neighborhood around $y_0$ if necessary we can assume
that $\Gamma(C(W_1)) \rightarrow W_1$ is finite. By conservation of number we obtain $$\mathrm{deg}_{W(y_0)}\Gamma(C(W))= \mathrm{deg}_{W_{1}(y_0)}\Gamma(C(W_1))=
\int_{S_{y_0}} l^{r}[D(W_1)_{\mathrm{vert}}].$$

In the same way define $\Gamma(C) = C \cap H_r$ and $\mathrm{deg}_{Y}\Gamma(C)$ where $H_r$ is generic enough so that the cover $\Gamma(C) \rightarrow Y$ is unramified for generic $y \in W(w_0) \cap Y$ and $\Gamma(C)_y \cap c_{X}^{-1}S_y = \emptyset$. 

Set $\Gamma(C(W_2)):= C(W_2) \cap H_r$ and $\mathrm{deg}_{W_2(y)}\Gamma(C(W_2)):= \mathrm{Card}(\Gamma(C(W_2))_{w})$ where $W_2(y)$ is small enough neighborhood of $y$ in $W_2$ and $w \in W_2$ is generic. We claim that 
$$\mathrm{deg}_{W_2(y)}\Gamma(C(W_2))= \int_{S_{y}} l^{r}[D(W_2)_{\mathrm{vert}}].$$

Note  $\int_{S_y}(H_r \cap c^{-1}S_{y})= \int_{S_{y}} l^{r}[D(W_2)_{\mathrm{vert}}]$ because the codimension of $H_r$ is right. Denote by $s_1, s_2, \ldots s_q$ to be the points in $S_y$ such that $\dim c^{-1}s_i \geq r$. Because $W_2$ is generic and because the cover the cover $\Gamma(C(W)) \rightarrow (W,w_0)$ is unramified for generic $w \in W(y)$, we get that the branches passing through $s_i$ form $\Gamma(C(W_2))$, which establishes the identity above.


The following relation (see \cite[Thm.\ 2.8]{GR}) shows that the presence of $\mathrm{deg}_{Y}\Gamma(C)$ is controlled by $\mathrm{deg}_{W(y_0)}\Gamma(C(W))$ and $\mathrm{deg}_{W_2(y)}\Gamma(C(W_2))$. We claim
\begin{equation}\label{key polar eq}
\mathrm{deg}_{W(y_0)}\Gamma(C(W)) - \mathrm{deg}_{W_2(y)}\Gamma(C(W_2))=\mathrm{deg}_{Y}\Gamma(C).
\end{equation}

Recall that $H_{r}$ was chosen so that it produces the covers $\Gamma(C(W)) \rightarrow (W,y_0)$ and $\Gamma(C) \rightarrow (Y,y_0)$. Observe that for generic $w \in W(y)$ close enough to $y$ the degree of the cover $\Gamma(C(W)) \rightarrow (W,w_0)$ is $\mathrm{deg}_{W(y_0)}\Gamma(C(W))$. Over a generic $y \in Y$ some of these branches merge at the $s_i$s and their number is $\mathrm{deg}_{W_2(y)}\Gamma(C(W_2))$ as shown above. The rest of the branches intersect $C_y$ at points away from $c^{-1}S_y$. Their number is $\mathrm{deg}_{Y}\Gamma(C)$ by our choice of $H_{r}$. 

Note that after possibly replacing $X$ with a smaller neighborhood of $S_{y_0}$ we have $\Gamma(C)$ is empty if and only if $\dim c_{X}^{-1}S_{y_0}<r$. 
So, if $\mathrm{vol}_{C(W_1)_{y_0}}(\mathcal{L})= \mathrm{vol}_{C(W_2)_{y}}(\mathcal{L})$ for all $y$ in a neighborhood of $y_0$ in $Y$, then by (\ref{curve subs}) $\Gamma(C)$ is empty and thus $\dim c_{X}^{-1}S_{y_0}<r$. Conversely, assume $\dim c_{X}^{-1}S_{y_0}<r$. Then $\mathrm{vol}_{C(W_1)_{y_0}}(\mathcal{L})= \mathrm{vol}_{C(W_2)_{y}}(\mathcal{L})$ for generic $y \in Y$. But $\mathrm{vol}_{C(W_1)_{y}}(\mathcal{L})$ is upper-semicontinuous because of (\ref{curve subs}) and $\mathrm{deg}_{W(y_0)}\Gamma(C(W)) \geq \mathrm{deg}_{W_2(y)}\Gamma(C(W_2))$. Thus, $\mathrm{vol}_{C(W_1)_{y_0}}(\mathcal{L})= \mathrm{vol}_{C(W_2)_{y}}(\mathcal{L})$ for all $y$ in a neighborhood of $y_0$ in $Y$. The proof of the theorem is now complete. 
\end{proof}

\begin{figure}[h]
\includegraphics[scale=0.2]{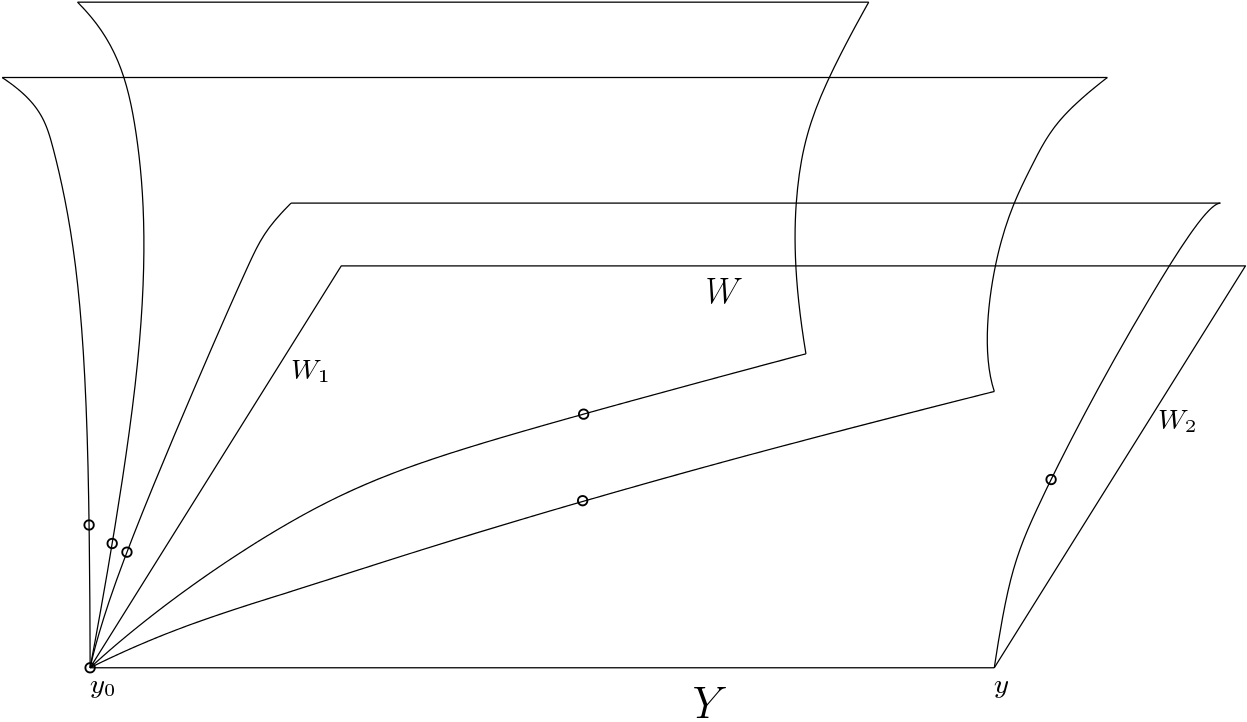}
\end{figure}

The figure summarizes the proof of (\ref{key polar eq}). For convenience the covers are pictured after projecting to $\mathcal{X}$. It's also assumed that there exists a section $\sigma \colon W \rightarrow \mathcal{X}$ where $W$ is identified with $\sigma(W)$ and $S=Y$. The following corollary follows immediately from (\ref{curve subs}) and (\ref{key polar eq}).

\begin{corollary}\label{bigMPT} In the setup of Thm.\ \ref{covering}, for generic $y \in Y$ in a sufficiently small neighborhood of $y_0$, we have
$$\mathrm{vol}_{C(W_1)_{y_0}}(\mathcal{L})-  \mathrm{vol}_{C(W_2)_{y}}(\mathcal{L}) = \mathrm{deg}_{Y}\Gamma(C).$$
\end{corollary}

In complex analytic singularity theory we apply Thm.\ \ref{covering} in the following setting: $(X,x_{0}) \rightarrow (Y,y_{0})$ is a family of isolated singularties, $W$ is a larger deformation base space containing $Y$, usually taken to be a component of the miniversal base space of $X_{y_0}$, $C(W)$ is the relative conormal space of $\mathcal{X} \rightarrow W$, and $S(W)$ is the singular locus of $\mathcal{X}$. Assume that $W$ is a good deformation space. Thus for $X_{y_0}$ and $X_{y}$ for $y \in Y$ close enough to $y_0$,  we can associate unique nonnegative real numbers depending on $W$ such that  $\int_{S_{y_0}}[D_{\mathrm{vert}}]$ vanishes if and only if these numbers are the same.


We finish this section by proving a general version Teissier's Principle of Specialization of Integral Dependence. Preserve the setup of Thm.\ \ref{covering}. When we write $\mathrm{vol}_{C_y}(\mathcal{L})$ we really mean the restricted local volume 
computed through the family $\mathcal{X} \times_{W} W'$ with special fiber $X_y$ where $W'$ is a generic smooth curve in $W$ passing through $y$. As usual denote by $\overline{\mathcal{A}(Y)}$ the integral closure of $\mathcal{A}(Y)$ in $\mathcal{B}(Y)$. 

\begin{theorem}[Principle of Specialization of Integral Dependence]\label{PSID} Assume $W$ is a good base space for $\mathrm{vol}_{C(W)_w}(\mathcal{L})$. Let $g \in \mathcal{B}(Y)$ is such that $\mathrm{Supp}_{\mathcal{O}_{\mathcal{X}}}(\overline{(g,\mathcal{A}(Y))}/\overline{\mathcal{A}(Y)}) \subset S$. Assume that there exists a Zariski open set $U$ in $Y$ such that for each $y \in U$ the image of $g$ in $\mathcal{B} \otimes_ {\mathcal{O}_{Y}}k(y)$ is integral over $\mathcal{A}(y)$. If $\mathrm{vol}_{C_y}(\mathcal{L})$ is constant, then $g$ is integral over $\mathcal{A}(Y)$. 
\end{theorem}
\begin{proof}
By Thm.\ \ref{covering}, constancy of $\mathrm{vol}_{C(W)_y}(\mathcal{L})$ implies that $\dim c^{-1}S_y <r$ for each $y$. Set $Z:=\mathrm{Supp}_{\mathcal{O}_{\mathcal{X}}}(\overline{(g,\mathcal{A}(Y))}/\overline{\mathcal{A}(Y)})$. By \cite[Lem.\ 1.2]{GK} there exists a smaller Zariski subset $U'$ of $Y$ such that $(g,\mathcal{A})$ is integral over $\mathcal{A}$. Thus $Z$ is proper closed subset of $S$. In particular, $\dim Z < \dim S=\dim Y$. 
Let $y_0 \in Y$ be such that $\dim c^{-1}S_{y_0}$ is maximal. Because $c^{-1}Z \subset c^{-1} S$, by upper semi-continuity of fiber dimension we have $$\dim c^{-1}Z \leq \dim c^{-1}S_{y_0}+\dim Z \leq  r-1+\dim Y -1 = \dim C -2.$$
Then by \cite[Cor.\ 10.7]{KT-Al} (cf.\ \cite[Thm.\ 4.1]{SUV} or \cite[Thm.\ 1.1 \rm{(iii)}]{Ran19a})
$Z$ is empty, and thus $g$ is integral over $\mathcal{A}(Y)$.
\end{proof}

\section{Whitney--Thom conditions, Jacobian modules, conormal spaces and integral dependence}\label{Whitney, Jacobian}
First we review the analytic and algebro-geometric formulations of the Whitney conditions and their connection to integral closure of modules. 
For a thorough overview of the history  of differential equisingularity theory we refer the reader to the masterful treatments of Kleiman \cite{KF}, Gaffney and Massey \cite{GaM} and 
Teissier \cite{Teissier3}. For a more recent account see \cite{Flores}. 

Let $X \subset \mathbb{C}^{n+k}$ be a equidimensional complex analytic space, and let $Y$ be a smooth subspace of $X$ of dimension $k$, so that $X-Y$ is smooth. Choose an embedding of $(X,0)$ in $\mathbb{C}^{n} \times \mathbb{C}^{k}$, so that $(Y,0)$ is represented by $0 \times V$, where $V$ is an open neighborhood of $0$ in $\mathbb{C}^k$. Let $\mathrm{pr}: \mathbb{C}^{n} \times \mathbb{C}^{k} \rightarrow \mathbb{C}^{k}$ be the projection. Set $h:=\mathrm{pr}_{|X}$. View $X$ as the total space of the family $h: (X,0) \rightarrow (Y,0)$. 
For each closed point $y \in Y$ set $X_y: = X \cap \mathrm{pr}^{-1}(y)$. Whitney \cite{Whitney} introduced the following conditions on the regularity of the pair $(X-Y,Y)$ at $0$. 

{\bf Whitney Condition A}. Let $(x_{i})$ be a sequence of points from $X-Y$ that converges to $0$, and suppose that the sequence $\{T_{x_i}X\}$ of tangent hyperplanes has a limit $T$ in the corresponding Grassmannian. Then $T_{0}Y \subset T$.

{{\bf Whitney Condition B}. Let $(x_{i})$ be a sequence of points from $X-Y$ and let $(y_{i})$ be a sequence of points from $Y$ both converging to $0$. Suppose that the sequence of secants $({\overline{x_{i}y_{i}}})$ has limit $l$ and the sequence of tangents $\{T_{x_{i}}X\}$ has limit $T$. Then $l \subseteq T$.

Whitney conjectured that these conditions would ensure the existence of a homeomorphism $h$ from $X_0 \times Y$ onto $X$. In other words the family $X_y$ is {\it topologically equisingular}. The conjecture is nowadays known as the Thom--Mather \cite{Thom} first isotopy theorem. Thom and Mather proved it by considering constant tangent vector field to $Y$ and lifting it carefully to $X$ such that the lift is integrable. The integral provides a continuous flow on $X$ which is $\mathbb{C}^{\infty}$ away from $Y$.

Hironaka \cite{Hironaka} introduced an intrinsic modified verson of the Whitney conditions. We say that $X$ satisfies the strict Whitney condition A if the distance from $Y$ to the tangent space $T_{x}X$ approaches $0$ as $x$ approaches $0$ along any analytic path $\phi: (\mathbb{C},0) \rightarrow (X,0)$ such that $\phi(u) \subset X-Y$ for any $u \neq 0$, i.e.
\begin{equation}\label{analytic}
\mathrm{dist}(Y,T_{x}X) \leq c \ \mathrm{dist}(x,Y)^{e},
\end{equation}
holds where the exponent $e$ and the constant $c$ depend on the path $\phi$. Hironaka proved that if $(X-Y,Y,0)$ satisfies both Whitney condition A and B, then it satisfies the strict Whitney condition A with exponent $e$ independent of the path $\phi$. The strict Whitney condition with exponent $e=1$ and $c$ independent of the path $\phi$ is called Verdier's W condition. Verdier \cite{Verdier} showed that his condition implies Whitney A and B. Teissier (Thm.\ 1.2 in Chap.\ V of \cite{Teissier}) showed that in the complex-analytic case $W$ is in fact equivalent to Whitney $A$ and $B$.

Let $f \colon (\mathbb{C}^{n+k},0) \rightarrow \mathbb{C}$ be an analytic function. Identify $f$ with its restriction to $(X,0)$. Assume $f$ is of constant rank off $Y$. In (\ref{analytic}) replace $T_{x}X$ by $T_{x}f^{-1}fx$.  The ``relative'' form of the Whitney condition A, called Thom's $A_f$ condition, holds at $0$ for the pair $(X-Y,Y)$ if (\ref{analytic}) is satisfied. The ``relative'' form of the Whitney condition B, called the $W_f$ condition, holds at $0$ for the pair $(X-Y,Y)$ if (\ref{analytic}) is satisfied
with $e=1$ and $c$ independent of the path $\phi$.

Let $U$ be an open set of  $\mathbb{C}^{n+k}$ containing a representative of $(X,0)$. The  differential $df$ of
$f$ defines an embedding of $X$ in  $\mathbb{C}^{n+k}\times {\mathbb{C}^{n+k}}^*$ by the graph map. Let $x_1, \ldots, x_{n+k}$ be coordinates on $U$ and $w_1, \ldots, w_{n+k}$ be the cotangent coordinates. Consider the blowup of $T_{X}^{*}U$ along the image of the graph map. It is the blowup of $T_{X}^{*}U$ by the ideal $(w_1- \frac{\partial f}{\partial x_1}, \ldots, w_{n+k} - \frac{\partial f}{\partial x_{n+k}})$ in $T_{X}^{*}U$. We denote this blowup by $\mathrm{Bl}_{\mathrm{im} \ df}T_{X}^{*}U$. Thus, the blowup is contained in $X \times {\mathbb{C}^{n+k}}^* \times \mathbb{P}^{n+k-1}$. Denote the exceptional divisor of this blowup by $E_{f}$. Denote the projection on $X$ of this exceptional divisor to $X$ by $\Sigma(f)$ and call it the {\it critical locus} of  $f$.

Define the the absolute conormal space $C(X,f)$ as the closure in $X \times \check{\mathbb{P}}^{n+k-1}$ of the set of pairs $(x,H)$ such that $x$ is a point in $X-\Sigma(f)$ and $H$ is a hyperplane tangent at $x$ to the level hypersurface $f^{-1}fx$. Teissier inspired by some work of Hironaka, considered the following normal-conormal diagram

\begin{displaymath}
\begin{CD}
\mathrm{Bl}_{c^{-1}(Y)}C(X,f) @>a'>> C(X,f)\\
@VVc'V  @VVcV\\
\mathrm{Bl}_{Y}X  @>a>> X
\end{CD}
\end{displaymath}

Set $\xi := c \circ a'$. Teissier \cite{Teissier} in the case of Whitney B, and Henry, Merle and Sabbah \cite{HMS} in the case of $W_f$  obtained the following equivalence for the Whitney conditions.
\begin{theorem}[Teissier]\label{Teissier} Preserving the setup from above, the following conditions are equivalent:
\begin{enumerate}
\item[\rm{(i)}] The pair $(X-Y,Y)$ satisfies $W_f$ at $0$.
\item[\rm{(ii)}] Let $\xi_{Y} : D_{Y} \rightarrow Y$ be the restriction of $\xi$ to the exceptional divisor $D_Y$ of
$\mathrm{Bl}_{c^{-1}(Y)}C(X,f)$. Then $\xi_{Y}$ is equidimensional and $\dim \xi_{Y}^{-1}(0) = n-2$. \end{enumerate}
\end{theorem}
Here is a way to see the relation between $\rm{(i)}$ and $\rm{(ii)}$ in the Whitney B case. Observe that set-theoretically $D_Y$ and $\mathbb{P}(C_{Y}X)\times_Y c^{-1}(Y)$ have the same irreducible components. Each point of
$\mathbb{P}(C_{Y}X)\times_Y c^{-1}(Y)$ has the form $(y,l,H)$ where $y$ is point of $Y$,
$l$ belongs to $(\mathbb{P}(C_{Y}X))_y$ and $H \in c^{-1}(y)$. Then the condition is
equivalent to $l \in H$. Thus $W_f$ holds at $0$ if $(D_Y)_{\mathrm{red}}$ is contained
in the incidence correspondence of $\mathbb{P}(m_{Y}/(m_Y)^{2})\times_{Y} \mathbb{P}((m_{Y}/(m_Y)^{2})^{*})$. Then a simple dimension count yields $\rm{(ii)}$. 

Next we provide a concrete
algebraic description of the conormal variety $C(X)$ using the Jacobian module of $X$. Suppose $X$ is reduced. Let $X$ be defined by the vanishing of some analytic
functions $f_1, \ldots, f_p$ on a Euclidean neighborhood of $0$ in $\mathbb{C}^{n+k}$. Consider the following exact sequence
\begin{equation}\label{normal sequence}
\begin{CD}
I/I^{2} @> \delta >> \Omega_{\mathbb{C}^{n+k}}^{1}|X \longrightarrow \Omega_{X}^{1} \longrightarrow 0 \end{CD}
\end{equation}
where $I$ is the ideal of $X$ in $\mathcal{O}_{\mathbb{C}^{n+k},0}$ and the map $\delta$ sends a function $f$ vanishing on $X$ to its differential $df$. Dualizing
we obtain the following  nested sequence of torsion-free sheaves:
\begin{displaymath}
\mathrm{Image}(\delta^{*}) \subset (\mathrm{Image} \ \delta)^{*} \subset (I/I^{2})^{*}.
\end{displaymath}
Observe that locally the sheaf $\mathrm{Image}(\delta^{*})$ can be viewed as the column space of the {\it (absolute) Jacobian module} of $X$ which we denote by $J(X)$. It's a module contained in the free module $\mathcal{O}_{X,0}^p$. Define the {\it Rees algebra} $\mathcal{R}(J(X))$ to be the subalgebra of $\mathrm{Sym}(\mathcal{O}_{X,0}^p)$ generated by the generators of $J(X)$ (for a general definition see \cite{Eisenbud}). Define the {\it conormal space} $C(X)$ of $X$ to be the closure in $X \times \check{\mathbb{P}}^{n+k-1}$ of the set of pairs $(x,H)$ where $x$ is a smooth point of $X$ and $H$ is a tangent hyperplane at $x$. Algebraically, 
\begin{displaymath}
 C(X) = \mathrm{Projan}(\mathcal{R}(J(X))).
\end{displaymath}
To see that equality holds observe that both sides are equal over $X-Y$ to the set of pairs $(x,H)$ where $H$ is a tangent hyperplane
to the simple point $x$. The left side is the closure of this set, and so is the right side simply because the Rees algebra is by construction
a subalgbebra of the symmetric algebra $\mathrm{Sym}(\mathcal{O}_{X,0}^{p})$. 
Similarly, one sees that $C(X,f)=\mathrm{Projan}(\mathcal{R}(J(X,f)))$ where $J(X,f)$ is the augmented Jacobian module with the partials of $f$. 

Consider the family setup $h: (X,0) \rightarrow (Y,0)$ from the beginning of the section. Assume each $f_i$ is analytic function on two sets of variables: the {\it fiber} variables $(x_1, \ldots, x_n)$ and the {\it parameter} variables $y_1, \ldots, y_k$. Form the augmented Jacobian matrix of $X$ and $f$ with respect to the fiber variables:
$$
\begin{pmatrix}
\partial f_{1} / \partial x_{1} & \cdots & \partial f_{1} / \partial x_{n} \\ \vdots & \ddots & \vdots  \\
\partial f_{p} / \partial x_{1} & \cdots & \partial f_{p} / \partial x_{n}\\
\partial f / \partial x_{1} & \cdots & \partial f / \partial x_{n}
\end{pmatrix}.
$$
The column space of this matrix generates a module over the local ring $\mathcal{O}_{X,0}$, which we denote by
\begin{displaymath}
J_{\mathrm{rel}}(X,f) := JM_{x}(f_1, \ldots, f_p;f) \subset \mathcal{O}_{X,0}^{p+1}
\end{displaymath}
and call the {\it relative Jacobian module} of $X$ and $f$. Given $y \in Y$ form the image of the module
 $JM_{x}(f_1, \ldots, f_p;f)$ in the free module $\mathcal{O}_{X_y}^{p+1}$ to get the Jacobian module of $(X_y,f_y)$:
\begin{displaymath}
J(X_y,f_y) := JM_{x}(f_1, \ldots, f_p;f)|X_y \subset {O}_{X_y}^{p+1}.
\end{displaymath}

Thus $C(X_y,f_y)=\mathrm{Projan}(\mathcal{R}(J(X_y,f_y)))$. Define the {\it relative conormal space} $C_{\mathrm{rel}}(X,f):=\mathrm{Projan}(\mathcal{R}(J_{\mathrm{rel}}(X,f)))$. 

Dropping the row with partials of $f$ in $J_{\mathrm{rel}}(X,f)$ we obtain the relative Jacobian module $J_{\mathrm{rel}}(X)$. Then the {\it relative conormal space} $C_{\mathrm{rel}}(X)$ of $X$ is defined as the closure of the set of pairs $(x,H)$ where $x$ is a smooth of $X$ and $H$ is tangent hyperplane containing a parallel to $Y$. Algebraically, $C_{\mathrm{rel}}(X)=\mathrm{Projan}(\mathcal{R}(J_{\mathrm{rel}}(X)))$.

The next ingredient we need in the proof of Thm.\ \ref{W_f} is the relation between $W_f$ and integral dependence of modules. Keep the setup from above. Let $\mathcal{F}$ be a free module, $\mathcal{M}$ be a submodule of $\mathcal{F}$, and
let $g$ be an element from $\mathcal{F}$. Assume $X$ is reduced. Form the Rees algebra $\mathcal{R}(\mathcal{M})$ of $\mathcal{M}$. It is a subalgebra of the symmetric algebra $\mathrm{Sym}(\mathcal{F})$ (see \cite[Prp.\ 1.8]{Eisenbud}). View $g$ as a degree one element in
$\mathrm{Sym}(\mathcal{F})$. We say that $g$ is {\it integrally dependent} on $\mathcal{M}$
if it satisfies an equation of integral dependence,
$$g^{u}+r_1g^{u-1}+ \cdots + r_u = 0$$
where $u \geq 1$ and each $r_i$ belongs to the $i$th homogeneous component of $\mathcal{R}(\mathcal{M})$.
The following well-known criteria tell us that we can check integral dependence on curves or by considering speeds of vanishing.

{\bf Valuative criterion.} An element $g \in \mathcal{F}$ is integrally dependent on $\mathcal{M}$ if for any map germ $\phi: (\mathbb{C},0) \rightarrow (X,0)$
the following holds
$$\phi^{*}g \in \phi^{*}\mathcal{M} \subset \phi^{*}\mathcal{F}.$$

{\bf Analytic criterion.} A necessary condition for $g \in \mathcal{F}$ to be integrally dependent on $\mathcal{M}$ is that for any finite set of generators $m_i$ of $\mathcal{M}$, there exists an Euclidean neighborhood $U$ of $0$ in $X$ and a constant $c$ such that $$|g(x)| \leq c \max |m_i(x)|$$ for each $x \in U$. Conversely, it suffices that this inequality holds for a finite generating set of $\mathcal{M}$.

Now for each $j=1, \ldots, k$ let $g_j$ be the column vector
\begin{equation}\label{section}
g_j :=
\begin{pmatrix}
\partial f_{1} / \partial y_{j}\\
\vdots \\
\partial f_{p} / \partial y_{j}\\
\partial f / \partial y_j
\end{pmatrix}.
\end{equation}
Denote the ideal of $Y$ in $\mathcal{O}_{X,0}$ by $m_{Y}$. The following result (\cite[Thm.\ 2.5]{Gaf1} and \cite[Prp.\ 2.3]{GK}) characterizes
the Whitney conditions by the ingtegral dependence of $g$ on the module $m_{Y}J_{\mathrm{rel}}(X,f)$. The main idea of its proof is to connect (\ref{analytic}) with integral dependence of modules through the analytic criterion mentioned above. This was done 
by Teissier in the case when $X$ is a of codimension one in $\mathbb{C}^{n+k}$ using integral closure of ideals,  and generalized by Gaffney and Kleiman \cite[Prp.\ 6.1]{GK} to arbitrary codimension using integral closure of modules.

\begin{theorem}[Gaffney--Kleiman]\label{Wht int dep}
 The pair $(X-Y,Y)$ satisfies $W_f$ at $0$ if and only if $g_j$ is integrally dependent on $m_Y J_{\mathrm{rel}}(X,f) $ for each $j=1, \ldots, p$.
\end{theorem}

To prove Thm.\ \ref{Wht int dep} one shows that (\ref{analytic}) is equivalent to statement that $|g_j(x)|$ are
bounded by $|m_{i}(x)|$ in a neighborhood of $0$ in $X$ where $m_i(x)$ are generators of $m_{Y}J_{\mathrm{rel}}(X,f)$. By the analytic criterion this in turn is equivalent to the integral dependence of $g_j$ on $m_{Y}J_{\mathrm{rel}}(X,f)$.

Hironaka \cite{Hironaka1} observed that the Whitney conditions hold generically, i.e.\ there exists
a Zariski dense open subset $U$ of $Y$ such that the Whitney conditions hold at each $y \in U$. The same result for $W_f$ was obtained by Henry, Merle and Sabbah (\cite[Thm.\ 5.1]{HMS}). Thus, showing that $W_f$ holds at $0$ is equivalent to showing that the closed set in $Y$ where the integral dependence of $g_j$ on $m_{Y}J_{\mathrm{rel}}(X,f)$ fails is empty. The last can be achieved by the constancy of fiberwise invariants as shown in Thm.\ \ref{PSID}.

 

\section{Numerical control of Whitney--Thom equisingularity}\label{numerical control}
Preserve the setup from the previous section: let $(X,0) \subset (\mathbb{C}^{n+k},0)$ be a reduced complex analytic germ, and let $(Y,0)$ be a linear subspace of $(X,0)$ of dimension $k$. As before, view $X$ as the total space of the family $h \colon (X,0) \rightarrow (Y,0)$ induced by a smooth retraction with fibers transverse to $Y$. In this section we apply the LVF and the Principle of Specialization of Integral Dependence (Thm.\ \ref{PSID}) to obtain numerical control for Whitney--Thom equisingularity (the $W_f$ and $A_f$ conditions) for $h \colon (X,0) \rightarrow (Y,0)$ and a function $f \colon X \rightarrow \mathbb{C}$. By this we mean finding invariants that depend on the fibers $X_y,f_y$ whose constancy across $Y$ is equivalent to Whitney--Thom equisingularity. 

We will assume that the fibers $X_y$ are isolated singularities. In this case we will show that finitely many invariants associated with each $X_y$ control equisingularity for all families having $X_y$ as a member. 
By Grauert's theorem \cite{Grauert} $(X_0,0)$  has a miniversal deformation 
$(X_0,0) \xhookrightarrow{i} (\mathcal{X},0) \xrightarrow{\phi} (Z,z)$
where $Z$ is complex analytic. By versality there exists a morphism $\psi \colon (Y,0) \rightarrow (Z,z)$ such that $h=\psi^{*} \phi$. Because $(Y,0)$ is irreducible its image lies in an irreducible component $W$ of $Z$. To simplify the exposition from now on we will assume $Y \subset W$ and by abuse of notation $\mathcal{X}$ will denote the total deformation space of $(X_0,0)$ over $W$.


Let $\Sigma(f)$ be the critical set of $f$ - the union of singular sets of the various level hypersurfaces. Set $Q:=X \cap f^{-1}0$ and $Q_y:=X_y \cap f^{-1}0$.
Assume $(W,0) \subset (\mathbb{C}^{l},0)$. Then in our context  an {\it unfolding} of $f$ over $W$ is an analytic  map $\tilde{f} \colon (\mathbb{C}^{n+k} \times \mathbb{C}^l,0) \rightarrow (\mathbb{C},0)$ such that $\tilde{f}({\bf x,0})= f$.

In the setup of Thm.\ \ref{covering} denote by $C(W)$ the  conormal space $C_{\mathrm{rel}}(\mathcal{X},\tilde{f})$ relative to $\mathcal{X} \rightarrow W$. Then $C$ is $C_{\mathrm{rel}}(X,f)$. Denote by $J_{\mathrm{rel}}(\mathcal{X},\tilde{f})$ the Jacobian module of the pair $(\mathcal{X},\tilde{f})$ relative to $\mathcal{X} \rightarrow W$. 
Algebraically, $C(W) = \mathrm{Projan}(\mathcal{R}(J_{\mathrm{rel}}(\mathcal{X},\tilde{f})))$.


As $\mathcal{X}$ is induced by perturbing the equations for $(X,0)$ in $(\mathbb{C}^{n+k},0)$, then $J_{\mathrm{rel}}(\mathcal{X},\tilde{f})$ is contained in a free module $\mathcal{F}:=\mathcal{O}_{\mathcal{X}}^{p+1}$ (see Sct.\ \ref{Whitney, Jacobian} and \cite[Cor.\ 1.6, pg.\ 230]{Greuel}). Denote by $\mathcal{F}(y)$ the restriction of $\mathcal{F}$ to $X_y$. Then the image of $J_{\mathrm{rel}}(\mathcal{X},\tilde{f})$ in $\mathcal{F}(y)$ is the Jacobian module $J(X_y,f_y)$ of the pair $(X_y,f_y)$. Denote by $\varepsilon(y)$  
the restricted local volume of $\mathcal{O}_{C(W)}(1)$ at $y$ computed through a deformation of $X_y$ over a generic curve in $W$ as in Thm.\ \ref{covering}. By the LVF applied for a generic one-parameter deformation of $X_y$ a curve in $W$ and by Prp.\ \ref{generic limit} and Prp.\ \ref{intr. of limits} (iii) $\varepsilon(w)$ is the epsilon multiplicity of the module $J(\mathcal{X}_w,\tilde{f}_w)$ in case $\dim \mathcal{X}_w \geq 2$. 

Let $m_y$ be the ideal of $y$ in $\mathcal{O}_{X_y}$. Consider the BR multiplicity $e(m_yJ(X_y,f_y),J(X_y,f_y))$  with $\mathcal{A}(y)$ and $\mathcal{A}'(y)$ the Rees algebras of $m_yJ(X_y,f_y)$ and $J(X_y,f_y)$, respectively (see (\ref{BR}) and setup of Cor.\ \ref{MPT}). Set 
\begin{equation}
\varepsilon m(y):=\varepsilon(y)+e(m_yJ(X_y,f_y),J(X_y,f_y)).
\end{equation}

When $(X_0,0)$ is a complete intersection, then $W$ is the base space of miniversal deformations of $(X_0,0)$ and $\varepsilon(y)$ is the BR multiplicity $e(J(X_y,f_y),\mathcal{F}(y))$. Thus $\varepsilon m (y)=e(m_{y}J(X_y,f_y),\mathcal{F}(y))$. This is the invariant used by Gaffney and Kleiman in \cite[Thm.\ 6.4]{GK} to give numerical characterization for $W_f$ in the complete intersection case. 
We do this here for arbitrary isolated singularities of dimension at least $2$ using $\varepsilon m (y)$ under the assumption of volume stability.

Suppose $\dim X_0 =1$. Embed $J_{\mathrm{rel}}(X,f)$ in a free module $\mathcal{F}'$ such that generically the two modules are the same (see Prp.\ \ref{intr. of limits} (iii)).  Then the BR multiplicity $e(m_yJ(X_y,f_y),\mathcal{F}'(y))$ is well-defined. We show that it gives numerical characterization for $W_f$ which is an unpublished result of Gaffney.





\begin{theorem}\label{W_f} Assume $X$ is reduced and equidimensional and $X_y$ is equidimensional for each $y$. 
Suppose $X_y$ and $Q_y$ are isolated singularities.
Suppose $h \colon X \rightarrow Y$ is a subfamily of $\mathcal{X} \rightarrow W$, where $W$ is a component of the base space of  the miniversal deformation of $X_0$ and $\mathcal{X}$ is the total deformation space of $X_0$ over $W$. Let $\tilde{f}$ be a generic unfolding of $f$. Suppose $X-Y$ is smooth, $\dim X_0 \geq 2$ and $\varepsilon (w)$ is stable. The following holds:
\begin{enumerate}
    \item[\rm{(i)}] Suppose $\Sigma(f) = Y$. Assume $\varepsilon m (y)$ is independent of $y$ near $0$. Then the union of the singular points of $f_y$ is $Y$ and the pair $(X-Y,Y)$ satisfies $W_f$ at $0$. 
    \item[\rm{(ii)}] Suppose $\Sigma(f)$ is equal to $Y$ or is empty and the pair $(X-Y,Y)$ satisfies $W_f$ at $0$. Then $\varepsilon m (y)$ is independent of $y$ near $0$. 
\end{enumerate}
Suppose $\dim X_0 = 1$. Then $\rm{(i)}$ and $\rm{(ii)}$ hold with $\varepsilon m (y)$ replaced by $e(m_yJ(X_y,f_y),\mathcal{F}'(y))$.
\end{theorem}
\begin{proof}
Suppose $\dim X_0 \geq 2$. Adopt the notation of Sct.\ \ref{stability}. Set $C(W):=C_{\mathrm{rel}}(\mathcal{X},\tilde{f})$ where 
$\tilde{f}$ is a generic unfolding so that $\Sigma(\tilde{f}) \subset \mathrm{Sing}(\mathcal{X})$. Set $C:=C_{\mathrm{rel}}(X,f)$. Define $\mathrm{deg}_Y\Gamma(C)$ and $\mathrm{deg}_Y\Gamma(\mathrm{Bl}_{c^{-1}Y}C)$ as before.
Note that in the notation of Cor.\ \ref{MPT}, these degrees are $\mathrm{mult}_{Y}\Gamma_d(J_{\mathrm{rel}}(X,f))$ and $\mathrm{mult}_{Y}\Gamma_d(m_YJ_{\mathrm{rel}}(X,f))$, respectively. Because $\varepsilon (w)$ is stable, by Cor.\ \ref{bigMPT}
\begin{equation}\label{epsilonMPT}
\varepsilon (0)-\varepsilon (y)=\mathrm{deg}_Y\Gamma(C),
\end{equation}
and by Cor.\ \ref{MPT}
$$e(m_{0}J(X_{0},f_{0}),J(X_{0},f_{0}))-e(m_yJ(X_y,f_y),J(X_y,f_y))= \mathrm{deg}_Y\Gamma(\mathrm{Bl}_{c^{-1}Y}C) - \mathrm{deg}_Y\Gamma(C)$$
for generic $y \in Y$ in a sufficiently small neighborhood of $y_0$. Therefore, 
$$\varepsilon m (0)-\varepsilon m (y)=\mathrm{deg}_Y\Gamma(\mathrm{Bl}_{c^{-1}Y}C).$$

Consider $\rm{(i)}$. By the result of Henry, Merle and Sabbah there exists a Zariski open subset $U$ of $Y$ such that the image of each $g_j$ (see (\ref{section})) in $\mathcal{O}_{X_y}^{p+1}$ is integral over $m_{y}J(X_y,f_y)$. 
Because $\varepsilon m (y)$ is constant on $y$, then $\Gamma(\mathrm{Bl}_{c^{-1}Y}C)$ is empty, so as in Thm.\ \ref{PSID} applied with $\mathrm{Bl}_{c^{-1}Y}C(\mathcal{X},\tilde{f})$ and $g:=g_j$ for $1 \leq j \leq k$, we conclude that each $g_j$ is integral over $m_{Y}J_{\mathrm{rel}}(X,f)$

Set $\mathcal{M}:= J_{\mathrm{rel}}(X,f)$ and $\mathcal{N}: = \langle \mathcal{M}, g_1, \ldots, g_k \rangle$. Recall $\mathcal{N}$ is contained in a free module $\mathcal{O}_{X,0}^{p+1}$. Because $\mathcal{N}$ is integral over $m_Y\mathcal{M}$ it follows that the integral closure of $\mathcal{N}$ in $\mathcal{O}_{X}^{p+1}$ is in the integral closure  of $J_{\mathrm{rel}}(X,f)$ in $\mathcal{O}_{X}^{p+1}$. Hence the nonfree locus of the former module, which is the union of singular points of $f_y$ and $X_y$, is equal to the nonfree locus of the latter module, which by assumption is equal to $Y$. This completes the proof of $\rm{(i)}$. 

Consider $\rm{(ii)}$. By Thm.\ \ref{Wht int dep} $\mathcal{N}$ is integral over $\mathcal{M}$. Thus $m_{Y}\mathcal{N}$ and $m_{Y}\mathcal{M}$ are the same up to integral closure. 
Consider the structure morphism $b \colon C(X,f) \rightarrow X$. Note that $\mathrm{Projan}(\mathcal{R}(m_{Y}\mathcal{N}))$ equals $\mathrm{Bl}_{b^{-1}Y}C(X,f)$ and $\mathrm{Projan}(\mathcal{R}(m_{Y}\mathcal{M}))$ equals $\mathrm{Bl}_{c^{-1}Y}C$. By Thm.\ \ref{Teissier} the exceptional divisor of  $\mathrm{Bl}_{c^{-1}Y}C(X,f)$ is equidimensional. Because $$\mathrm{Bl}_{b^{-1}Y}C(X,f) \rightarrow  \mathrm{Bl}_{c^{-1}Y}C$$ is finite, 
then the exceptional divisor $D$ of  $\mathrm{Bl}_{c^{-1}Y}C$ is equidimensional, too. Therefore, $\Gamma(\mathrm{Bl}_{c^{-1}Y}C)$ is empty and $c^{-1}Y$ has no vertical components of top dimension. Thus $\Gamma(C)$ is empty, too. 
So, by Thm.\ \ref{bigMPT} applied for any smooth curve joining $0$ with $y$ close enough to $0$ we obtain that  $e(m_yJ(X_y,f_y),J(X_y,f_y))$ is independent of $y$. Because  $\Gamma(C)$ is empty, then by (\ref{epsilonMPT}) $\varepsilon (y)$ is constant for $y$ in a Zariski open neighborhood of $y_0$. But $\varepsilon (y)$ is upper semi-continuous. So $\varepsilon (y)$ is independent of $y$ near $0$ which shows that $\varepsilon m (y)$ is independent of $y$ near $0$, too. 

Suppose $\dim X_0=1$. By Cor.\ \ref{MPT} (cf.\ Thm.\ \ref{LVF} \rm{(i)}) $$e(m_0J(X_0,f_0),\mathcal{F}'(0))-e(m_yJ(X_y,f_y),\mathcal{F}'(y))=\mathrm{deg}_Y\Gamma(\mathrm{Bl}_{c^{-1}Y}C)$$
and the rest of the proof is the same as before. Note that in this case the $e(m_yJ(X_y,f_y),\mathcal{F}'(y))$ does not depend on $W$. The proof of Thm.\ \ref{W_f} is now complete. \end{proof}

\begin{remark}\label{stability of f}
\rm{In the case when the epsilon multiplicity $\varepsilon (J(\mathcal{X}_w))$ of the Jacobian module of $X_w$ vanishes for some $w$, that is the case when $\mathcal{X}_w$ is a deficient conormal singularity (see the next section for a definition), then we claim that $\varepsilon (w)=0$ for $w$ in Zariski open dense subset of $W$. In particular, $\varepsilon (w)$ is stable. To see this choose an unfolding  $\tilde{f}:=f+\sum_{i=1}^{l}\alpha_iw_i$, where $w_i$ are coordinates for $(\mathbb{C}^l,0)$ and $\alpha_i$ are generic constants so that the hyperplane determined by $\mathrm{grad}(\tilde{f})$ at $s$ is not part of $C(\mathcal{X})_s$ for $s \in S(W)_w$. Then by \cite[Thm.\ 2.2]{GRB} $C(\mathcal{X},\tilde{f})_{s}$ is the join of $d\tilde{f}(s)$ and $C(\mathcal{X})_s$. But $C(\mathcal{X})_s$ is not of maximal dimension because $\mathcal{X}_w$ is deficient conormal. Because $\dim C(\mathcal{X},\tilde{f})=\dim C(\mathcal{X})+1$, it follows that the fiber $C(\mathcal{X},\tilde{f})_{s}$ is not of maximal dimension either. Thus, $\varepsilon (w)=0$ (see Prp.\ \ref{vanishing of DCS}). The existence of Zariski open subset of $W$ where $\varepsilon (w)=0$ follows from upper semi-continuity of the restricted local volume (cf.\ Prp.\ \ref{vanishing of DCS}).}
\end{remark}
To get numerical control for the Whitney condition B, simply drop $f$ in Thm.\ \ref{W_f}: set $C(W)=C_{\mathrm{rel}}(\mathcal{X})$ and $C:=C_{\mathrm{rel}}(X)$. A fiberwise numerical characterization of Whitney A is not possible (see \cite[Ex.\ 4.3]{GK}).

Denote by $X^i$ the section of $X$ by a general linear space of codimension $i$ in $\mathbb{C}^{n+k}$ containing $Y$. Set $f^{i}:=f|X^{i}$. A fundamental result of L\^e and Teissier \cite{LT2} states that $W_f$ holds if and only if $X^{i}_{y},f^{i}_y$ are topologically equisingular, i.e.\ if there exists a homeomorphism $l^{i} \colon (\mathbb{C}^{n+k-i},0) \rightarrow (\mathbb{C}^{n+k-i},0)$ which induces a homeomorphism  $(\mathbb{C}^{n+k-i},X^i,f^il^i) \cong (\mathbb{C}^{n-i}\times Y,X_{0}^{i} \times Y, f_{0}^{i} \times 1_Y)$. Thus we have the following corollary to Thm.\ \ref{W_f}. 

\begin{corollary}
Under the assumption of Thm.\ \ref{W_f}, the pairs $X^{i}_{y},f^{i}_y$ are topologically equisingular if and only if $\varepsilon m (y)$ is constant. 
\end{corollary}

Next we turn our attention to providing a numerical characterization for Thom's $A_f$ condtion.

\begin{definition}
We say that {\it $A_f$ condition} holds for the pair $X_{\mathrm{sm}},Y$ at $0$ if  $f(Y)=0$ and $Y$ lies in every hyperplane obtained as a limit of tangent hyperplanes to a level hypersurface at a point $x \in X_{\mathrm{sm}}$ as $x$ approaches $0$.
\end{definition}
The $A_f$ condition is known to hold generically along $Y$ by a result of Hironaka \cite[Thm.\ 2, pg.\ 247]{Hironaka1}. We review briefly the connection between the theory of integral closure of modules and Thom's $A_f$ condition. Recall that given a submodule $\mathcal{M}$ of a free $\mathcal{O}_{X,0}$-module $\mathcal{F}$, we say that $u \in \mathcal{F}$ is {\it strictly dependent} on $\mathcal{M}$ and we write $ u \in \mathcal{ M}^{\dagger}$, if for all analytic path germs $\phi:(\mathbb C,0)\to(X,0)$, $\phi^*u$ is contained in $\phi^*(\mathcal{M})m_1$, where $m_1$ is
the maximal ideal of $\mathcal{O}_{\mathbb{C},0}.$

Denote by $b \colon C(X,f) \rightarrow X$ the structure morphism and by $C(Y)$ the conormal space of $Y$ in $\mathbb{C}^{n+k}$. The following result expresses the $A_f$ condition in terms of strict dependence. 
\begin{proposition}\label{A_f equiv.}
\label{A_f} Assume $f(Y)=0$. The following are equivalent
\begin{itemize}
\item[\rm{(i)}]The $A_f$ condition holds for the pair $(X_{\mathrm{sm}},Y)$ at $0$.
\item[\rm{(ii)}] $b^{-1}(Y)\subset C(Y)$.
\item[\rm{(iii)}] $g_j \in J(X,f)^{\dagger}$ for all $j=1, \ldots, k.$
\end{itemize}
\end{proposition}
\begin{proof}  The equivalence of i) and ii) is obvious; the equivalence of i) and iii) is Lemma 5.1 of \cite{GK}.
\end{proof}
A similar result holds for the Whitney A condition. The condition we need for our main result is a much weaker version of Whitney A.
\begin{definition} We say that $(X,0) \rightarrow (Y,0)$ satisfies the {\it infinitesimal Whitney A fiber condition at} $0$ if the image of each $g_j$ in $\mathcal{F}(0)$ is in $J(X_0)^{\dagger}$ for $j=1, \ldots, k$, where $J(X_0)$ is the Jacobian module of $X_0$ and $\mathcal{F}(0):=\mathcal{O}_{X,0}^{p}$ is the free module that contains it. 
\end{definition}

\begin{theorem}\label{A_f} Assume $X$ is reduced and equidimensional and $X_y$ is equidimensional for each $y$. 
Suppose $X_y$ and $Q_y$ are isolated singularities.
Suppose $h \colon X \rightarrow Y$ is a subfamily of $\mathcal{X} \rightarrow W$, is a component of the base space of the miniversal deformation of $X_0$ and $\mathcal{X}$ is the total deformation space of $X_0$ over $W$. Let $\tilde{f}$ be a generic unfolding of $f$ over $\mathcal{X}$. Suppose $(X-Y)$ is smooth, $\dim X_0 \geq 2$ and $\varepsilon (w)$ is stable, and that the infinitesimal Whitney A fiber condition at $0$ holds in case $f \notin m_{Y}^2$. The following holds:
\begin{enumerate}
    \item[\rm{(i)}] Suppose $\Sigma(f) = Y$. Suppose $\varepsilon (y)$ is independent of $y$ near $0$. Then the union of the singular points of $f_y$ is $Y$ and the pair $(X-Y,Y)$ satisfies $A_f$ at $0$. 
    \item[\rm{(ii)}] Suppose $\Sigma(f)$ is equal to $Y$ or is empty and the pair $(X-Y,Y)$ satisfies $A_f$ at $0$. Then $\varepsilon (y)$ is independent of $y$ near $0$. 
\end{enumerate}
Suppose $\dim X_0 = 1$. Then $\rm{(i)}$ and $\rm{(ii)}$ hold with $\varepsilon (y)$ replaced by $e(J(X_y,f_y),\mathcal{F}'(y))$.
\end{theorem}
\begin{proof}
Suppose $\dim X_0 \geq 2$. Adopt the notation of Sct.\ \ref{stability}. Set $C(W):=C_{\mathrm{rel}}(\mathcal{X},\tilde{f})$ where 
$\tilde{f}$ is a generic unfolding so that $\Sigma(\tilde{f}) \subset \mathrm{Sing}(\mathcal{X})$. Set $C:=C_{\mathrm{rel}}(X,f)$. Define $\mathrm{deg}_Y\Gamma(C)$ and as before.
Because $\varepsilon (w)$ is stable, by Cor.\ \ref{bigMPT}
$$\varepsilon (0)-\varepsilon (y)=\mathrm{deg}_Y\Gamma(C).$$
Consider $\rm{(i)}$. By Hironaka's genericity result \cite{Hironaka1} there exists a Zariski open subset $U$ of $Y$ such that $A_f$ condition holds at all points in $U$. Then by selecting analytic paths in $U$ and the discussion above we obtain that $J(X,f)$ and $J_{\mathrm{rel}}(X,f)$ have the same integral closure over $U$. Let $\tilde{f}$ be such that $\Sigma(\tilde{f}) \subset \mathrm{Sing}(\mathcal{X})$. Because $\varepsilon (y)$ is independent of $y$ near $0$
by Thm.\ \ref{covering} applied with $C(W):=C(\mathcal{X},\tilde{f})$, $S(W):=\mathrm{Sing}(\mathcal{X})$, and $g:=g_j$ for $1 \leq j \leq k$, it follows that each $g_j$ is integral over $J_{\mathrm{rel}}(X,f)$. Thus,  $J(X,f) \subset \overline{J_{\mathrm{rel}}(X,f)}$ and so there exists a finite map $C(X,f) \rightarrow C$. Because $J_{\mathrm{rel}}(X,f)$ is generated by $n$ elements  $\dim C_0 < n$. Thus $\dim C(X,f)_{0} < n$. Then by \cite[Thm.\ 4.4]{GRB} it follows that the pair $(X-Y,Y)$ satisfies $A_f$. As in the proof of Thm.\ \ref{W_f} we obtain that union of singular points of $f_y$ is $Y$.

Consider $\rm{(ii)}$. Let $K$ be the submodule of $\mathcal{F}$ generated by $g_1, \ldots, g_k$. Then $K+J_{\mathrm{rel}}(X,f) =J(X,f)$. Because $A_f$ holds, then by Prp.\ \ref{A_f equiv.} we get
$$J(X,f)\circ \phi=K \circ \phi + J_{\mathrm{rel}}(X,f) \circ \phi \subset m_{1}(J(X,f)\circ \phi) + J_{\mathrm{rel}}(X,f) \circ \phi$$
for any analytic path $\phi \colon (\mathbb{C},0) \rightarrow (X,0)$. By Nakayama's lemma we get $J(X,f)\circ \phi=J_{\mathrm{rel}}(X,f) \circ \phi$. By the valuative criterion for integral dependence we obtain that $J(X,f) \subset \overline{J_{\mathrm{rel}}(X,f)}$, where $\overline{J_{\mathrm{rel}}(X,f)}$ is the integral closure of $J_{\mathrm{rel}}(X,f)$ in $\mathcal{F}$. In other words, there exists a finite map $C(X,f) \rightarrow C$. 

By Prp.\ \ref{A_f equiv.} $b^{-1}(Y) \subset C(Y)$. But  $\dim C(Y)_{0} = n-1$. Because
$C(X,f) \rightarrow C_{\mathrm{rel}}(X,f)$ is finite we get $\dim C_{\mathrm{rel}}(X,f)_0 < n$. Apply Cor. \ref{bigMPT}. Because $\Gamma (C)$ is defined as the intersection of $C$ with a general plane of codimension $n$, then $\Gamma (C)$  is empty. Thus by Cor.\ \ref{bigMPT}  $\varepsilon (y)$ is independent of $y$ near $0$. The case $\dim X_0=1$ follows in the same way by applying Cor.\ \ref{MPT} for $e(J(X_y,f_y),\mathcal{F}'(y))$.
\end{proof}

\section{Deficient conormal singularities and generalized smoothability}\label{deficient section}
In this section we define the class of deficient conormal (dc) singularities. We show that the dc propery is intrinsic and stable under deformations. We show that the fibers of conormal spaces behave well under pullbacks of transverse holomorphic maps between affine spaces. Then using Thom's transversality we show that determinantal and Pfaffian singularities (with two exceptions) deform to dc singularities. As a corollary we obtain that the the generic deformations of Cohen--Macaulay codimension $2$ and Gorenstein codimension $3$ singularities are dc. Inspired by a recent result of Koll\'ar and Kov\'{a}cs and following work of Schlessinger, we show that the affine cones over normally embedded abelian varieties of dimension at least $2$ do not admit infinitesimal deformations to dc singularities. We finish the section with computing the restricted local volume associated with the conormal space of a nonsmoothable Cohen--Macaulay codimension $2$ singularity. 

Let $Z \subset \mathbb{C}^N$ be a reduced complex analytic variety with irreducible components of positive dimension. 
As before define the conormal space $C(Z)$ of $Z$ as the closure in $Z \times \check{\mathbb{P}}^{N-1}$ of the set of pairs $(z,H)$ where $z$ is a smooth point of $Z$ and $H$ is a tangent hyperplane at $z$, i.e. a hyperplane containing $T_{Z,z}$. Denote by $c$ the structure map $c: C(Z) \rightarrow Z$. For each point $z_0 \in Z$ a simple dimension count yields $1 \leq \codim (c^{-1}z_0, C(Z)) \leq \dim Z.$

\begin{definition} We say $Z$ is {\it deficient conormal} (dc) at $z_0 \in Z$ if $$\codim (c^{-1}z_{0}, C(Z)) \geq 2.$$ 
Alternatively, $Z$ is dc at $z_0$ if $(Z,z_0)$ has no polar curve (see the proof of Prp.\ \ref{plr} below). We say $Z$ is dc if $Z$ is dc at each of its points. 
\end{definition}

\begin{example}\label{correspondance} \rm{If $Z$ is smooth at $z_0$ with $\dim (Z,z_0) \geq 2$, then $Z$ is dc at $z_0$. If $(Z,z_0)$ is a local complete intersection that is dc at $z_0$ and smooth elsewhere, then by Lem.\ 5.7 in \cite{GaM} $Z$ is smooth at $z_0$. 

If $Z$ is the affine cone over a smooth projective variety $V \subset \mathbb{P}^{N-1}$ with positive defect (or degenerate dual) (see \cite{Ein86} and \cite{Ein85} for examples and classifications), then $Z$ is dc at the vertex of the cone. To see this it's enough to show that the fiber of the conormal of $Z$ over the vertex is equal set-theoretically to the dual of the base (not requiring $V$ to be smooth).  An easy computation shows that  $H$ is a tangent hyperplane at a point $v$ in $V_{\mathrm{sm}}$ if and only if $H$ is tangent to each point of $Z$ lying on the line $0v$ and different from the origin. If $z_1, z_2, \ldots$ are points from $Z$ converging to the origin in $\mathbb{C}^{N}$ (the vertex of the cone), and $H_1, H_2, \ldots$ are the corresponding tangent hyperplanes, then the $H_i$s are tangent hyperplanes at points of $V$, and thus their limit belongs to the dual of $V$. Conversely, if $H$ is in the dual of $V$, say $H$ is a limit of tangent hyperplanes $H_1, H_2, \ldots$ at smooth points $v_1,v_2, \ldots$ then $H$ is a limit of tangent hyperplanes at any sequence of points from the lines $0v_{1},0v_2,\ldots$, and it belongs to the fiber $C(Z)$ over the origin. 
 
It can be shown that many of the examples of rigid singularities considered by Schlessinger \cite{Sch} like fans (eg.\ two planes meeting at a point in $\mathbb{C}^4$), quotient singularities of dimension at least $2$, etc.\ are dc. In \cite{Ran19c} it is shown that
if $Z \subset \mathbb{C}^N$
is normal with $\codim (c^{-1}z_0,C(Z)) =\dim Z$, then $Z$ is smooth at $z_0$. In particular, dc normal surfaces are smooth.} 
\end{example}

Set $\mathcal{L}:=\mathcal{O}_{C(Z)}(1)$.  The next result shows that the dc property is characterized by the vanishing of the local volume of $\mathcal{L}$. Suppose $Z$ is equidimensional and $\dim C(Z):=r$. Recall that the local volume $\mathrm{vol}_{C}(\mathcal{L})$ at $z_0$ is given by the epsilon multiplicity of the Jacobian module $J_{z_0}(Z)$ of $Z$ at $z_0$:
$$\varepsilon(J_{z_0}(Z)):= \lim_{n \to \infty} \frac{r!}{n^r}\dim_{\mathbb{C}} H_{z_0}^{0}(\mathcal{F}^n/J_{z_0}(Z)^n),$$ 
where $J_{z_0}(Z)^n:=\mathrm{Sym}^n(J_{z_0}(Z))/(\mathcal{O}_{Z,z_0}\text{-torsion})$ and $\mathcal{F}$ is a free  $\mathcal{O}_{Z,z_0}$-module containing $J_{z_0}(Z)$ with $\mathcal{F}^n:=\mathrm{Sym}^n(\mathcal{F})$.
\begin{proposition}\label{vanishing of DCS}
Suppose $Z$ is equidimensional with $\dim Z \geq 2$. Then $\varepsilon(J_{z_0}(Z))=0$ if and only if $Z$ is dc at $z_0$. 
\end{proposition}
\begin{proof}
Follows immediately from Thm.\ \ref{vanishing}.
\end{proof}
Our next result shows that the dc property is intrinsic. Suppose $Z_1 \subset \mathbb{C}^{n_1}$ and $Z_2 \subset \mathbb{C}^{n_2}$ are reduced equidimensional complex analytic varieties. Denote the respective conormal spaces by $C(Z_1)$ and $C(Z_2)$. Denote by $c_{i}\colon C(Z_i) \rightarrow Z_i$ for $i=1,2$ the corresponding structure morphisms. 

\begin{proposition}\label{plr} Suppose $z_1 \in Z_1$ and $z_2 \in Z_2$ are two points such that there exists an analytic isomorphism $\phi \colon (Z_1,z_1) \rightarrow (Z_2,z_2)$. Then $\codim (c_{1}^{-1}z_1,C(Z_1)) = l$  for some positive integer $l$ if and only if $\codim (c_{2}^{-1}z_2, C(Z_2)) = l$. In particular, $Z_1$ is dc at $z_1$ if and only if $Z_2$ is dc at $z_2$. 
\end{proposition} 
\begin{proof} 
Let $Z \subset \mathbb{C}^N$ be a reduced equidimensional complex analytic varieity. Let $z \in Z$. Consider the germ $(Z,z)$. Let's recall
the construction of the local polar varieties of $(Z,z)$ as developed by Teissier in \cite[Chp.\ IV]{Teissier}.  Consider the following diagram
$$
\begin{tikzcd}
C(Z) \arrow[hook]{r}{i}\arrow{rd}{\lambda} \arrow{d}{c}
&Z \times \check{\mathbb{P}}^{N-1} \arrow{d}{\mathrm{pr}_2}\\
(Z,z) & \check{\mathbb{P}}^{N-1}.
\end{tikzcd}
$$
Assume $(Z,z)$ is a germ of pure dimension $d$ and of codimension $e$ in $\mathbb{C}^N$. Let $H_{e+l-1}$ be a  general plane in $\check{\mathbb{P}}^{N-1}$ of codimension $e+l-1$. Define the {\it local polar variety} $\Gamma_{l}(Z,H_{e+l-1})$ with respect to $H_{e+l-1}$ as $c(\lambda^{-1}(H_{e+l-1}))$. 
Then $\Gamma_{l}(Z,H_{e+l-1})$ is either empty or of pure codimension $l$. Teissier \cite[Chp.\ IV, Thm.\ 3.1]{Teissier} showed that for sufficiently general $H_{e+l-1}$ the multiplicity of $\Gamma_{l}(Z,H_{e+l-1})$ at $z$ depends only on the analytic type of $(Z,z)$. Observe that for sufficiently general $H_{N-l}$ the local polar variety $\Gamma_{d-l+1}(Z,H_{N-l})$ is empty if and only if $\codim (c^{-1}z, C(Z)) \geq l$. 

Apply Teissier's result to $Z_1$ and $Z_2$. Because $\codim (c_{1}^{-1}z_1, C(Z_1)) = l$, then the multiplicity at $z_1$ of the polar varieties of $(Z_1,z_1)$ of codimension $d-l+1$ is zero. But then by Teissier's result, so is the multiplicity at $z_2$ of the the polar varieties of $(Z_2,z_2)$ of dimension $d-l+1$. Hence $\codim (c_{2}^{-1}z_2, C(Z_2)) = l$.   
\end{proof}
Note that dc is an open condition: by upper semicontinuity of fiber dimension if $Z$ is dc at $z_0$, then $Z$ is dc at $z$ close enough to $z_0$.  Next we show that the dc property is stable under flat deformations. 

\begin{proposition}
Let $(X,0) \rightarrow (Y,y_0)$ be one-parameter flat deformation of a reduced equidimensional dc singularity $(X_{y_0},0) \subset (\mathbb{C}^n,0)$.
Then $X_y$ is dc for all $y$ close enough to $y_0$. 
\end{proposition}
\begin{proof}
Assume $(X,0) \subset (\mathbb{C}^{n+1},0)$. Let $J_{\mathrm{rel}}(X)$ be the relative Jacobian module of $(X,0) \rightarrow (Y,0)$, and let $\mathcal{F}$ be a free module that contains it. Set $\mathcal{F}(y):=\mathcal{F} \otimes_{\mathcal{O}_Y} k(y)$ for each point $y \in Y$. The image of $J_{\mathrm{rel}}(X)$ in $\mathcal{F}(y)$ is the Jacobian module $J(X_y)$ of $X_y$. Denote by $c \colon C_{\mathrm{rel}}(X) \rightarrow (X,0)$ the relative conormal space of $(X,0)$. Then $\dim C_{\mathrm{rel}}(X)=n$. If $\dim c^{-1}0 \leq n-3$, then $X$ is dc at each point sufficiently close to $0$, thus $X_y$ is dc for $y$ close enough to $0$. Suppose $\dim c^{-1}0 \geq n-2$. Set $T:= \{x \in X | \dim c^{-1}x \geq n-2\}$. Then $T$ is closed and nonempty with $\dim T \leq 1$. If $\dim c^{-1} x_y \leq n-3$ for $x_y \in X_y$ then $X_y$ is dc at $x_y$. Thus all non-dc points of the fibers $X_y$ are contained in $T$. If $T$ lies in $X_{y_0}$, then $X_y$ is dc for $y \neq y_0$. Suppose $T$ has a component $S$ such that its image in $Y$ is dense. By replacing $Y$ with a small enough neighborhood of $y_0$ we can assume that $S$ is finite over $Y$. We have $$H_{S}^{0}(\mathcal{F}^n/J_{\mathrm{rel}}(X)^n)\otimes_{\mathcal{O}_Y} k(y_0) \hookrightarrow  H_{S_{y_0}}^{0}(\mathcal{F}(y_0)^n/J(X_{y_0})^n).$$ Thus $\varepsilon(J_{\mathrm{rel}}(X))(y_0) \leq \varepsilon(J(X_{y_0}))$.
Because $X_{0}$ is dc at each point of $0$, then by Prp.\ \ref{vanishing of DCS} $\varepsilon(J(X_{y_0}))=0$ and so  $\varepsilon(J_{\mathrm{rel}}(X))(y_0)=0$. Applying the LVF and noting that the intersection number in its right-hand side is nonnegative we obtain  $\varepsilon(J_{\mathrm{rel}}(X))(y)=0$. 
By Prp.\ \ref{generic limit} $\varepsilon(J_{\mathrm{rel}}(X))(y)=\varepsilon(J(X_y))$ for $y$ close enough to $0$. Thus $\varepsilon(J(X_y))=0$ locally at each singular point of $X_y$. So by Prp.\ \ref{vanishing of DCS} $X_y$ is dc. 
\end{proof}

The next three results are the key ingredients in the proof of the main theorem of this section. The first one shows that the codimension of the fibers of conormal spaces can only increase under transverse maps. The second result due to Trivedi allows us to deform holomorphic maps in such a way so that the deformations become transverse to a predetermined collection of submanifolds in the target space. The third result due to Laksov and Buchweitz allows us to obtain deformations of singularities obtained from pullbacks of holomorphic maps. 

\begin{definition}
Let $M$ and $N$ be complex manifolds. Let $g \colon M \rightarrow N$ be a smooth map. We say that $g$ is transvserse to a submanifold $S$ of $N$ at a point $m \in M$ and we denote it by $g \pitchfork_{m} S$, if either $g(m) \not \in S$ or $g(m) \in S$ and $Dg_{m}(T_{m}M)+T_{g(m)}S=T_{g(m)}N$. If $V = \bigsqcup_{i=1}^{q} V_i$ is a stratification of a complex analytic variety $V \subset N$ and $K \subset M$, then by $g \pitchfork_{K} V$ we mean that $g$ is transverse to each stratum $V_i$ at all points in $K$. 
\end{definition}

Suppose $(V,0) \subset (\mathbb{C}^m,0)$ is a complex analytic variety. We say that $V = \bigsqcup_{i=1}^{q} V_i$ is Whitney A (respectively Whitney B) stratification if pairs of nearby strata satisfy Whitney condition A (respectively Whitney B). The existence of these stratifications was established by Whitney \cite{Whitney} as discussed in Sct.\ \ref{Whitney, Jacobian}. 

Let $g \colon (\mathbb{C}^n,0) \rightarrow (\mathbb{C}^m,0)$ be a holomorphic map.  Consider the diagram where the right square is cartesian:

\begin{displaymath}
\begin{CD}
T^{*}\mathbb{C}^n   @<<(dg)^{*}< g^{*}T^{*}\mathbb{C}^m @>g_{\pi}>> T^{*}\mathbb{C}^m\\
@VV\pi_{n}V @VV\pi V  @VV \pi_{m}V\\
\mathbb{C}^n @<<\mathrm{Id}< \mathbb{C}^n  @>g>> \mathbb{C}^m.
\end{CD}
\end{displaymath}
We say that $A$ is a {\it conic} subset of $T^{*}\mathbb{C}^m$ if its fibers are invariant in the fibers of $T^{*}\mathbb{C}^m$ under a multiplication with a nonzero complex number. For a closed conic subset $A$ of $T^{*}\mathbb{C}^m$ we say that $g$ is {\it noncharacteristic} with respect to $g$ if $g_{\pi}^{-1}(A) \cap \mathrm{Ker}((dg)^{*})$ is the
zero-section of $g^{*}T^{*}\mathbb{C}^m$. Denote by $T_{V}^{*}\mathbb{C}^{m}$ the closure in $\mathbb{C}^m \times \mathbb{C}^{m*}$ of the pairs of $(v, \eta_v)$ where $v$ is a smooth point of $V$ and $\eta_v$ is a conormal vector at $v$. Define $T_{X}^{*}\mathbb{C}^{n}$ analogously. Note that $T_{V}^{*}\mathbb{C}^{m}$ and $T_{X}^{*}\mathbb{C}^{n}$ are conical. By projectivizing with respect to vertical homotheties in $T^{*}\mathbb{C}^m$ and $T^{*}\mathbb{C}^n$ we obtain $\mathbb{P}(T_{V}^{*}\mathbb{C}^{m})=C(V)$ and $\mathbb{P}(T_{X}^{*}\mathbb{C}^{n})=C(X)$. 

\begin{theorem}\label{functoriality of conormal} Let $(X,0) \subset (\mathbb{C}^n,0)$ and $(V,0) \subset (\mathbb{C}^m,0)$ be equidimensional reduced complex analytic varieties with $g^{-1}V=X$ and $\codim X = \codim V$. Assume $g^{-1}V'$ is an irreducible component of $X$ for each irreducible component $V'$ of $V$. Assume that $g \pitchfork_{0}V$ where $V$ is given a Whitney B stratification. Let $x \in g^{-1}(0)$. If $V$ is dc at $0$, then so is $X$ at $x$.  
\end{theorem}
\begin{proof} Let $(V,0)= \bigsqcup_{i=1}^{q} V_i$ be a Whitney $B$ stratification of $(V,0)$. Suppose $V_1$ is the smooth locus of $V$ and $0 \in V_2$. First we show that there exists a neighborhood $U$ of $0$ in $(\mathbb{C}^m,0)$ such that $g$ is transverse to each point in $U \cap V_i$ for $i=1, \ldots q$. Note that $f$ is transverse at $U \cap V_2$ for sufficiently small $U$ by openess of transversality. We will show that $g$ is transverse to $U \cap V_1$ for $U$ sufficiently small (the proof for the rest of the strata is the same). Suppose $x_1, x_2, \ldots$ is a sequence of points in $X$ converging to $x$ with $x \in g^{-1}(0)$ such that $g$ fails to be transverse at $g(x_1), g(x_2), \ldots$ Then $Dg(T_{x_j}\mathbb{C}^n)+T_{g(x_j)}V_1 \neq T_{g(x_j)}\mathbb{C}^m$ for $j=1,2, \ldots$. 
Assume that as $x_j \rightarrow x$ the linear spaces $T_{g(x_j)}V_1 \rightarrow T$. Because $(V_1,V_2)$ satisfies Whitney condition A at $0$ it follows that 
$T_{0}V_2 \subset T$. This would imply that $Dg(T_{x}\mathbb{C}^n)+T_{0}V_2 \neq T_{0}\mathbb{C}^m$ which contradicts our assumption $g \pitchfork_0 V_1$. Replace $V_i$ by $U \cap V_i$ for each $i$ where $U$ is a sufficiently small neighborhood of the origin in $\mathbb{C}^m$. Because $g$ is transverse to $V_i$, then $\bigsqcup_{i=1}^{q}g^{-1}(V_i)$ is a Whitney B stratification of $X$ (see \cite[pg.\ 257]{Sc}). Set $X_1:=g^{-1}V_1$ and $X_2:=g^{-1}V_2$. 

Suppose $n > m$. Then $\dim X_2 = n-m+\dim V_2 \geq 1$. Because $(X_2,X_1)$ satisfies Whitney B at $x$, then by Thm.\ \ref{Teissier} the exceptional divisor $D_{X_2}$ of $\mathrm{Bl}_{c^{-1}(X_2)}C(X)$ is equidimensional. Thus $\dim D_{X_2}(x)=n-\dim X_2 - 2$. But $D_{X_2}(x)$ surjects onto $C(X)_x$. So $X$ is dc at $x$.

Suppose $n \leq m$. Let $T_{X_2}$ be a transversal of $X_2$ through $x$. Because $g$ is transverse at $V_2$, it follows from the local diffeomorphism theorem, after $V_2$ and $X_2$ are replaced by small enough neighborhood around $0$ and $x$, that $(T_{V_2},0):=g((T_{X_2},x))$ is transversal of $V_2$ through $0$ and it is diffeomorphic to $T_{X_2}$. Let $X_x$ be the fiber over $x$ of the transverse retraction from $X$ to $X_2$ induced by $T_{X_2}$. Similarly, denote by $V_0$ the fiber over $0$ of a transverse retraction from $V$ to $V_2$ induced by $T_{V_2}$.

Because $(X_2,x)$ is smooth, then its defining equations are part of regular system of parameters in $\mathcal{O}_{X,x}$. Thus the family $(X,x) \rightarrow (X_2,x)$ obtained from the transversal retraction induced by $T_{X_2}$ has equidimensional fibers.  Because $(X_1,X_2)$ satisfies Whitney B at $x$, by  \cite[Thm.\ 3.1]{GRB}
set-theoretically, $C(X)_x = C(X_x)_x$ where the conormal spaces are taken in $\mathbb{C}^n$ and $C(X)_x$ is the fiber of $C(X)$ over $x$. 

By assumption the inverse image of each irreducible component of $(V,0)$ is an irreducible component of $(X,x)$. Thus the family $(V,0) \rightarrow (V_2,0)$ obtained from the transversal retraction induced by $T_{V_2}$ has equidimensional fibers. Hence by  \cite[Thm.\ 3.1]{GRB}
we have a set-theoretic equality $C(V)_0 = C(V_0)_0$. Our goal is to show that set-theoretically $C(X_x)_x = (dg)^{*}(g_{\pi}^{-1}(C(V_0)_0))$.

By hypothesis we can arrange $T_{V_2}$ such that the set in $V_0$ where the preimage of $g$ is empty does not contain a component of $V_0$. Thus $g((X_0)_{\mathrm{sm}})$ is dense in $V_0$. Because $g$ is a diffeomorphism between $T(X_2)$ and $T(V_2)$ it follows that $g$ is a bijection between $(X_0)_{\mathrm{sm}}$ and $g((X_0)_{\mathrm{sm}})$, and the last set is open.

Set $\Lambda:= g_{\pi}^{-1}(T_{V_0}^{*}\mathbb{C}^m)$ where $T_{V_0}^{*}\mathbb{C}^m$ is the closure of $(v, \eta_v)$ in $\mathbb{C}^m \times \mathbb{C}^{m*}$ with $v$ a smooth point of $V_0$ and $\eta_v$ a conormal vector at $v$. Set $\Lambda':= dg^{*}(\Lambda)$. Because the restriction of $g$ is a diffeomorphism between $T_{X_2}$ and $T_{V_2}$, then $g_{\pi}^{-1}(A) \cap \mathrm{Ker}((dg)^{*})$ is the zero section for every closed conic $A \subset T_{T_{V_2}}^{*}\mathbb{C}^m$. Thus $g$ is noncharacteristic with respect to 
$\Lambda$. In particular, by Lem.\ 4.3.1 in \cite[pg. 255]{Sc} $(dg)^{*} \colon \Lambda \rightarrow \Lambda'$ is proper.

Let $\omega_n$ and $\omega_m$ be the Liouville $1$-forms on $T^{*}\mathbb{C}^n$ and $T^{*}\mathbb{C}^m$ respectively. We have $$\omega_m | T_{V_0}^{*}\mathbb{C}^m \equiv 0 \implies g_{\pi}^{*}\omega_m | \Lambda \equiv 0 \implies (dg)^{*}\omega_{n} | \Lambda \equiv 0 \implies \omega_n | \Lambda' \equiv 0$$
where the last implication follows from the properness of the map $(dg)^{*}$ and Prp.\ 8.3.11 in \cite[pg. 332]{KS}. Because $\Lambda'$ is closed and $\omega_n$ vanishes at at all tangent vectors to $\Lambda'$ at each of its smooth points, then by \cite[Prp.\ 2.2]{Flores} $\Lambda'$ equals the conormal of its image under $\pi_n$, i.e.\ $\Lambda' = T_{X_x}^{*}\mathbb{C}^{n}$. Note that $\mathbb{P}(T_{V_0}^{*}\mathbb{C}^m)=C(V_0)$ and $\mathbb{P}(T_{X_x}^{*}\mathbb{C}^n)=C(X_x)$. Thus, set-theoretically, $C(X)_x = (dg)^{*}(g_{\pi}^{-1}(C(V)_0))$.

Let $f_1(y_1, \ldots, y_m), \ldots, f_s(y_1, \ldots, y_m)$ be equations for $V_0$ in $(\mathbb{C}^m,0)$. Observe that $\mathbb{P}(\Lambda)$ is the $\mathrm{Projan}$ of the Rees algebra of the coherent $\mathcal{O}_{X_{x},x}$-module generated by the columns of the matrix $(\frac{\partial f_i}{\partial y_j} \circ g)$. In particular,  $\mathbb{P}(\Lambda)$ is analytic. Because $\mathbb{P}(\Lambda) \rightarrow \mathbb{P}(\Lambda')$ is a proper map between analytic spaces whose generic fiber is of dimension $m-n$, by upper semi-continuity of fiber dimension and the dimension formula we get  $\dim (\Lambda_x) \geq  \dim \Lambda'_x + m-n$. Because $\Lambda_x = C(V)_0$ and $\Lambda'_x = C(X)_x$, then  $$\codim (C(X)_x,C(X)) \geq \codim (C(V)_0,C(V)).$$ In particular, if $V$ is dc at $0$, then so is $X$ at $x$.
\end{proof}





Denote by $\mathcal{H}(M,N)$ the complete metric space of holomorphic maps between two complex manifolds $M$ and $N$ with the weak topology induced by the weak topology of $C^{\infty}(M,N)$. The following result due to Trivedi is a generalization of Thom's classical transversality result \cite{Thom} to the complex analytic setting. 

\begin{theorem}[Trivedi \cite{Trivedi}, Thm.\ 2.1 and Thm.\ 3.1]\label{Trivedi} Let $M$ be a Stein manifold and $N$ be an Oka manifold. Let $V$ be a complex analytic variety in $N$ and let $V = \bigsqcup_{i=1}^{u} V_i$ be a Whitney A stratification. Then for any compact subset $K$ in $M$, the set of maps $ \{g \in \mathcal{H}(M,N): g \pitchfork_{K} V_i\}$ is open and dense in $\mathcal{H}(M,N)$.
\end{theorem}

Next we record a key result due to Laksov in the algebraic case and due to Buchweitz in the complex analytic case which allows to obtain flat deformations of varieties obtained from pullbacks of holomorphic maps between complex affine spaces by deforming these maps. 

\begin{theorem}\rm{(\cite[Sct.\ 3]{Laksov} and \cite[4.3.4]{Bu81}\label{Buchweitz}).}
Let $F \colon (\mathbb{C}^n,0) \rightarrow (\mathbb{C}^m,0)$ be a holomorphic map. Let $(X,0) \subset (\mathbb{C}^n,0)$ and $(V,0) \subset (\mathbb{C}^m,0)$ be complex analytic varieties with $\codim X = \codim V$ and $g^{-1}V=X$. Assume $g^{-1}V'$ is an irreducible component of $X$ for each irreducible component $V'$ of $V$ and assume that $(V,0)$ is Cohen--Macaulay. Then for each unfolding $\tilde{F} \colon (\mathbb{C}^n \times \mathbb{C}^k,0) \rightarrow (\mathbb{C}^m,0)$ the family $\tilde{F}^{-1}V \rightarrow (\mathbb{C}^k,0)$ is a flat deformation of $X$. 
\end{theorem}

Denote by $\Sigma^{u}$ the subvariety of $\mathrm{Hom}(\mathbb{C}^{l}, \mathbb{C}^{l+s})$ consisting of linear maps of rank less than $u$. Consider the map $F \colon (\mathbb{C}^{n},0) \rightarrow \mathrm{Hom}(\mathbb{C}^{l}, \mathbb{C}^{l+s})$ given by a $l+s$ by $l$ matrix $M_X$ with entries complex analytic functions in a neighborhood of the origin in $\mathbb{C}^{n}$. We say $X:=F^{-1}\Sigma^{u}$ is {\it determinantal of type} $(l+s,l,u)$ ($X$ consists of the points for which $\mathrm{rk}(M_X) < u$) if $\codim X = \codim \Sigma^u = (l+s-u+1)(l-u+1)$ in $\mathbb{C}^n$. 

An isolated determinantal singularity $X$ of type $(l+s,l,u)$ is smoothable if $\dim X \leq 2l+s-2u+2$ \cite[Thm.\ 6.2]{Wahl81}. By the Hilbert--Burch theorem \cite[Sct.\ 20.4]{EisenbudCM} Cohen--Macaulay codimension $2$ varieties are determinantal of type $(l+1,l,l)$. In particular, they are smoothable up to dimension $3$ (\cite{Schaps}, cf.\ pg.\ 19--20 in \cite{Artin}) and in general nonsmoothable in dimension $4$ and higher unless they are complete intersections (M.\ Zach, priv.\ comm., 2019). 

Denote by $\mathrm{Skew}_q$ the space of $q \times q$ skew-symmetric matrices. Depending on the parity of $q$ we write $q=2l$ or $q=2l+1$. Denote by $\Pi_{2u} \subset \mathrm{Skew}_q$ the  generic Pfaffian variety defined by the Pfaffians of the $2u \times 2u$ skew-symmetric submatrices of the generic $q \times q$ skew-symmetric matrix with entries $q(q-1)/2$ indeterminates. Consider the map $F\colon (\mathbb{C}^n,0)\rightarrow \mathrm{Skew}_q$ given by a skew-symmetric $q \times q$ matrix $S_X$ with entries complex analytic functions in a neighborhood of the origin in $\mathbb{C}^n$. We say $X:=F^{-1}\Pi_{2u}$ is Pfaffian  of type $(q,2u)$ ($X$ consists of the points for which $\mathrm{rk}(S_X) < 2u$) if $\codim X = \codim \Pi_{2u} = \binom{q-2u+2}{2}$ (see \cite[Sct.\ 5]{KL80}).

Isolated Pfaffian singularities of a  $(2l+1) \times (2l+1)$ skew-symmetric matrix are smoothable if $\dim X \leq 4(l-u)+6$ \cite[Thm.\ 6.3]{Wahl81}. Gorenstein codimension $3$ singularities are Pfaffian with $q=2l+1$ and $u=l$ by the structure theorem of Buchsbaum and Eisenbud \cite[Thm.\ 2.1]{BE77}. 

The following theorem, which combines theorems \ref{functoriality of conormal}, \ref{Trivedi} and \ref{Buchweitz}, is the main result of this section. 
\begin{theorem} Let $F \colon (\mathbb{C}^n,0) \rightarrow (\mathbb{C}^m,0)$ be a holomorphic map. Let $(X,0) \subset (\mathbb{C}^n,0)$ and $(V,0) \subset (\mathbb{C}^m,0)$ be reduced complex analytic varieties with $\codim X = \codim V$ and $g^{-1}V=X$. Assume $g^{-1}V'$ is an irreducible component of $X$ for each irreducible component $V'$ of $V$. Assume $(V,0)$ is dc and Cohen--Macaulay. Then there exists an embedded flat deformation $(\mathcal{X},0) \rightarrow (W,0)$ of $X$ such that $\mathcal{X}_{w}$ is dc for generic $w$. In particular, determinantal and Pfaffian singularities deform to dc singularities with two exceptions: 
\begin{enumerate}
    \item [\rm{(i)}] in the determinantal case when  $n \geq m:=l^2$ and $V$ is the affine cone over the Segre embedding of $\mathbb{P}^{l-1} \times \mathbb{P}^{l-1}$ in $\mathbb{P}^{l^2-1}$; 
    \item [\rm{(ii)}] in the Pfaffian case when $n \geq m:=\binom{2l}{2}$ and $V$ is the affine cone over the Pl\"{u}cker embedding of $\mathrm{Gr}(2,2l)$ in $\mathbb{P}(\bigwedge^2\mathbb{C}^{2l})$.
\end{enumerate}
\end{theorem}
\begin{proof}
Let $V = \bigsqcup_{i=1}^{u} V_i$ be a Whitney B stratification. Deforming the entries of $F$ by generic linear forms on $(\mathbb{C}^m,0)$ by Thm.\ \ref{Buchweitz} we obtain an unfolding $\tilde{F}$ of $F$ such that $\tilde{F}^{-1}V \rightarrow (\mathbb{C}^m,0)$ is a flat deformation of $X$. Set $\mathcal{X}:= \tilde{F}^{-1}V$ and $W:=(\mathbb{C}^m,0)$. By Thm.\ \ref{Trivedi} $\tilde{F}_w \pitchfork_{0}V$ for generic $w$ as complex affine spaces are Oka and Stein. Because $\tilde{F}_w \pitchfork_{v}V$
and $V$ is dc at $v$ for $v$ close enough to $0$, we can assume that $\tilde{F}^{-1}0$ is nonempty. By Thm.\ \ref{functoriality of conormal} $\mathcal{X}_w$ is dc. 

Each determinantal or Pfaffian variety $X$ is obtained as $F^{-1}V$ where $V$ is the generic determinantal or Pfaffian singularity of appropriate sizes. Note that $V$ is irreducible and Cohen--Macaulay by a result of Eagon and Hochster \cite{EH71} in the determinantal case, and by a result of Kleppe and Laksov \cite{KL80} in the Pfaffian case. In the determinantal case $\codim(C(V)_0, C(V))$ is $(u-1)(s+u-1)$ by  \cite[Prp.\ 2.6]{GR} which is greater or equal to $2$ unless $u=2$ and $s=0$. In this case $V$ is the affine cone over the Segre embedding of $\mathbb{P}^{l-1} \times \mathbb{P}^{l-1}$ in $\mathbb{P}^{l^2-1}$. By  \cite[Prp.\ 2.6]{GR} $C(V)_0$ is the hypersurface in  $\mathbb{P}^{l^2-1}$ cut out by the determinant of the generic $l \times l$ matrix. 

If $n<l^2$, then $\tilde{F}_w$ will miss the singular locus of $V$ which is the origin. Thus by the implicit function theorem $\mathcal{X}_w: = \tilde{F}_{w}^{-1}V$ is smooth. If $n=l^2$ and $X=V$, then $X$ is rigid (e.g.\ \cite[Thm.\ 2.2.8]{Land}), hence $X$ cannot be deformed to a dc singularity.

Suppose $X$ is Pfaffian with $V=\Pi_{2u}$. Then by \cite[Thm.\ 5.9]{LS17} $C(V)_0$ is $\Pi_{q-2u+4}$ if $q$ is even and $\Pi_{q-2u+3}$ if $q$ is odd. The only 
way $\codim (C(V)_0,C(V)) = 1$ is if $q=2l$ and $u=2$. Then $V$ is the affine cone over the image of the Pl\"ucker embedding of $\mathrm{Gr}(2,2l)$ in $\mathbb{P}(\bigwedge^2 \mathbb{C}^{2l})$. If $n<\binom{2l}{2}$, then by the implicit function theorem $\mathcal{X}_w$ is smooth. If $n=\binom{2l}{2}$ and $X=V$, then $X$ is rigid by \cite[Thm.\ 5.3]{Sv75} and thus cannot be deformed to a dc singularity.  
\end{proof}
An immediate consequence of the structure theorems of Hilbert--Buch and Buchsbaum--Eisenbud is that the versal deformation space of a Cohen--Macaulay codimension $2$ or a Gorenstein codimension $3$ singularity is smooth (see pg.\ 67--68 and pg.\ 76 in \cite{Har}). Thus we have following corollary. 
\begin{corollary}\label{CMGor}
Suppose $(X,0)$ is a Cohen--Macaulay codimension $2$ or a Gorenstein codimension $3$ singularity. Then the generic deformation of $(X,0)$ is dc. 
\end{corollary}
\begin{example}
\rm{We give another example of isolated singularities whose nearby flat deformations are not dc. Let $Z$ be an abelian variety and $\mathcal{L}$ an ample line bundle on $Z$. Then $\mathcal{L}^3$ gives a projectively normal embedding  $Z \hookrightarrow \mathbb{P}^{N-1}=\mathbb{P}H^{0}(\mathcal{L})$ where $N= \dim H^{0}(Z,\mathcal{L})$ (see \cite{Koizumi}). Denote by $\Theta$ the tangent sheaf of $Z$. Then for nonzero $i$ we have $H^{1}(\Theta(i))=0$, because the tangent bundle of any group variety is trivial, and $H^{1}(\mathcal{O}_{Z}(i))=0$ by the vanishing theorem in \cite[pg.150]{Mumford2}. Denote by $C_Z$ the affine cone over $Z$. By \cite[Thm.\ 2, pg.159]{Sch} the versal deformation space of of $C_Z$ is isomorphic to the versal deformation space of $Z$ in $\mathbb{P}^{N-1}$. So, every deformation of $C_Z$ is conical. In particular, $C_Z$ is not smoothable (this was proved without the hypothesis of projective normality by Koll\'ar and Kov\'{a}cs in \cite{KK18}). 

Let  $C_{\mathcal{Z}} \rightarrow Y$ be a deformation of $C_Z$ with $(C_{\mathcal{Z}})_{y_0}=C_{Z}$ induced by a deformation $\mathcal{Z} \rightarrow Y$ of $Z$ with $\mathcal{Z} \subset Y \times \mathbb{P}^{N-1}$ and $\mathcal{Z}_{y_0}=Z$ for some closed point $y_0 \in Y$. To show that the fiber of the conormal over the vertex of the cone $(C_{\mathcal{Z}})_y=C_{\mathcal{Z}_y}$ is of maximal dimension for each $y$, by the correspondence established in Ex.\ \ref{correspondance}, it is enough to show that the dual of $\mathcal{Z}_y$ is a hypersurface in $\check{\mathbb{P}}^{N-1}$, i.e.\ $\mathrm{def}(\mathcal{Z}_y)=0$. 
If $\mathrm{def}(Z)=r \geq 1$, then by  \cite[Thm.\ 1.8 \rm{(i)}]{Tevelev}  $Z$ is ruled by projective subspaces of dimension $r$. But that's impossible because any morphism from $\mathbb{P}^1$ to a group variety is constant. Thus $\mathrm{def}(Z)=0$. We claim that $\mathrm{def}(\mathcal{Z}_y)=0$ for each $y$ close enough to $y_0$. 

Let $e$ be the identity element in $Z$. Because $Z$ is smooth, after $Y$ is replaced by sufficiently small neighborhood of $y_0$, there exists $Y' \subset Z$ passing through $e$ such that $r \colon Y' \rightarrow Y$ is \'etale. Consider the family $Z \times_{Y} Y' \rightarrow Y'$ with the section $y' \rightarrow (y',y')$ for $y' \in Y'$. Then for each $y'$ the fiber over $y'$ is isomorphic $Z_{r(y')}$. By \cite[Thm.\ 6.14]{Fogarty}  $Z \times_{Y} Y'$ is abelian; hence, there are no projective spaces contained in it and thus $\mathcal{Z}_y$ is not ruled. Therefore, $C_{\mathcal{Z}_y}$ is not dc.}
\end{example}

As mentioned before all Cohen--Macaulay codimension $2$ singularities of dimension at most $3$ are smoothable.  The result is sharp because the cone in $\mathbb{C}^6$ over the Segre embedding of $\mathbb{P}^{1} \times \mathbb{P}^2$ in $\mathbb{P}^5$, which is a Cohen--Macaulay codimension $2$ singularity, is rigid and hence not smoothable (this is the first example of a nonsmoothable singularity due to Thom). In fact, in dimension $4$ and higher Cohen--Macaulay codimension $2$ singularities are nonsmoothable unless they are complete intersections. Next we consider a class of  $4$-dimensional isolated nonsmoothable Cohen--Macaulay codimension $2$ singularities suggested to me by T.\ Gaffney. The generic deformations of this isolated singularity are singular but dc by Cor.\ \ref{CMGor}. For this class of singularities the restricted local volume associated with the the relative conormal spaces takes particularly nice form - it's a sum of a Buchsbaum--Rim multiplcity and a polar multiplicity.

\begin{example}
\rm{To each polynomial $h(w,x,y)$ which defines an isolated singularity in $(\mathbb{C}^3,0)$ associate the Cohen--Macaulay codimension $2$ singularity $X_h$ in $(\mathbb{C}^6,0)$ defined by the vanishing of the $2$ by $2$ minors of the following matrix

\begin{equation}\label{CM example}
F_h:=
\begin{pmatrix}
u & x  \\
v & y\\
h(w,x,y) & z
\end{pmatrix}.
\end{equation}
where we view the presentation matrix above as a map $F_{h}\colon (\mathbb{C}^6,0) \rightarrow \mathrm{Hom}(\mathbb{C}^2,\mathbb{C}^3)$. Observe that $X_h$ is not smoothable (see \cite[pg.\ 19--20]{Artin}). Note that $X_h=F_{h}^{-1}(\Sigma^2)$.  A one-parameter deformation $\mathcal{X}_h$ of $X_h$ with fibers $\mathcal{X}_{h}(t)$ is obtained by perturbing the entries of the presentation matrix (\ref{CM example}). Then by Cor.\ \ref{CMGor} the generic fiber $\mathcal{X}_{h}(t)$ is dc. 

Apply Cor.\ \ref{MPT} to the pair of modules: the relative Jacobian module $J_{\mathrm{rel}}(\mathcal{X}_h)$, and the normal module $N(\mathcal{X}_h)$ of $\mathcal{X}_h$ which in this case is $F_{h}^{*}(J(\Sigma^2))$, the pullback of the Jacobian module of $\Sigma^2$ (see  \cite[Prp.\  2.11]{GRu} and the discussion preceeding \cite[Prp.\ 2.4]{GR}). Observe that $J_{\mathrm{rel}}(\mathcal{X}_h)$ specializes to the Jacobian module $J(\mathcal{X}_h (t))$ of each fiber $\mathcal{X}_{h}(t)$. 
Because $\mathcal{X}_{h}(t)$ is dc for generic $t$, by \cite[Thm.\ 6.1]{Rangachev2} the module $N(\mathcal{X}_h (t))$ is integrally dependent on $J(\mathcal{X}_{h}(t))$. So, for generic $t$ the Buchsbaum--Rim multiplicity vanishes: $$e(J(\mathcal{X}_{h}(t)),N(\mathcal{X}_h (t)))=0.$$ For dimension reasons the codimension $4$ polar variety $\Gamma_{4}(\Sigma^2)$ is 
empty (see the discussion that follows  \cite[Prp.\ 2.14]{GRu}). Hence, by \cite[Thm.\ 2.5]{GR} $\Gamma_{4}(N(\mathcal{X}_h))$ is empty. Thus Cor.\ \ref{MPT} yields $e(J(X_{h}), N(X_h))=m_{4}(X_h)$
where $m_4(X_h)$ is the degree of the polar curve of $\mathcal{X}_h$. Applying the LVF, or more precisely (\ref{epsilonMPT}), with the observation that the generic term on the left-hand side vanishes by Prp.\ \ref{vanishing of DCS}, we get 
$$\varepsilon (0) =e(J(X_{h}),N(X_h)).$$
Thus we reduced the problem of computing the restricted local volume  to computing a relative Buchsbaum--Rim multiplicity, which as defined in Sct.\ \ref{Excess-Degree} can be computed as a sum of intersection numbers of the blowup of $\mathrm{Proj}(\mathcal{R}(N(X_h)))$ with center the ideal generated by $J(X_{h})$. First we need to find $\mathrm{Proj}(\mathcal{R}(N(X_h)))$. Recall that 

$$B_{X_h}:= \overline{\{(a,l_1,l_2)|a \in X_{h}, l_1 \in \mathbb{P}(\mathrm{Ker}(M_{X_{h}}(a))), l_2 \in \mathbb{P}(\mathrm{Ker}(M_{X_{h}}^t(a))\}}$$
which sits inside $X_{h}\times \mathbb{P}^{1} \times \mathbb{P}^2$ is set-theoretically equal to $\mathrm{Proj}(\mathcal{R}(N(X_h)))$ by  \cite[Thm.\ 3.7]{GR}. For $a \in X_h$ the morphism between $X_{h} \times \mathbb{P}^{1} \times \mathbb{P}^2$ and $X_{h} \times \mathbb{P}(\mathrm{Hom}(\mathbb{C}^2, \mathbb{C}^3))$ is given by $$(a,[S_1,S_2], [T_1,T_2,T_3]) \rightarrow  \left( a, \begin{bmatrix}
S_1T_1 & S_2T_1  \\
S_1t_2 & S_2T_2\\
S_1T_3 & S_2T_3
\end{bmatrix} \right) .$$

An easy computation shows that $B_{X_h}$ is cut out locally at the chart $[1,s],[1,t_1,t_2]$ from $\mathbb{C}[x,y,z,u,v,w,s,t_1,t_2]$ by $u+sx=0, v+sy=0,h+sz=0, t_{1}x+t_{2}y+z=0$. Hence $B_{X_h}$ is a complete intersection. Then $\mathrm{Proj}(\mathcal{R}(N(X_h)))$ set-theoretically is a complete intersection. But it is generically reduced because $X_h$ is reduced. So, $\mathrm{Proj}(\mathcal{R}(N(X_h)))$ is reduced. Thus 
\begin{equation}\label{projisom}
\mathrm{Proj}(\mathcal{R}(N(X_h))) \simeq \mathbb{C}[x,y,w][s,t_1,t_2]/ (h-s(t_1x+t_2y)). 
\end{equation}
where $\mathbb{C}[x,y,w]$ is localized at the origin. Our next task is to compute the ideal induced by $J(X_h)$ in $\mathcal{R}(N(X_h))$. Set $$G \colon \mathrm{Hom}(\mathbb{C}^2, \mathbb{C}^3) \rightarrow \mathbb{C}^2$$ such that $G^{-1}(0) = \Sigma^2$. From the chain rule $D(G \circ F_h)=(D(G) \circ F_h)\circ D(F_h)$ and $N(X_h)=F_{h}^{*}(J(\Sigma^2))$ it follows that the ideal $\mathcal{I}_{J}$ generated by $J(X_h)$ in $\mathcal{R}(N(X_h))$ is generated by $D(F_h)$. An easy computation shows that the generators for $\mathcal{I}_{J}$  are $t_1,t_2,st_1+h_x,st_2+h_y,s,h_w$. Thus by (\ref{projisom}) $\mathbb{V}(\mathcal{I}_{J}) = \mathrm{Spec}(\mathbb{C}[x,y,w]/\langle h,J(h) \rangle)$ where $J(h)$ is the {\it Jacobian ideal} of $h$. Therefore, the computation of the Buchsbaum--Rim multiplicity reduces to computing the Hilbert--Samuel multiplicity of $\mathcal{I}_{J}$ in $\mathbb{C}[x,y,w][s,t_1,t_2]/ (h-s(t_1x+t_2y))$. The latter ring is Cohen--Macaulay of dimension $5$. Therefore, if $\mathcal{I}_{J}'$ is a reduction of $\mathcal{I}_{J}$, i.e. an ideal generated by $5$ generic $\mathbb{C}$-linear combinations of the generators, then the Hilbert--Samuel multiplicity $e(\mathcal{I}_{J})$ of $\mathcal{I}_{J}$ is equal to $\dim_{\mathbb{C}}\mathbb{C}[x,y,w][s,t_1,t_2]/ (\mathcal{I}_{J}',h-s(t_1x+t_2y))$. For generators of $\mathcal{I}_{J}'$ we can choose $t_1,t_2,s$ and $2$ generic linear combinations of $h_x,h_y$ and $h_w$. Thus, 
$$e(\mathcal{I}_J) = \dim_{\mathbb{C}}\mathbb{C}[x,y,w]/(h,\alpha_{1}h_x+\alpha_{2}h_y+\alpha_{3}h_w, \beta_{1}h_x+\beta_{2}h_y+\beta_{3}h_w)$$ for generic $\alpha_i$ and $\beta_i$. Hence $e(\mathcal{I}_{J})$ is equal to the Hilbert--Samuel multiplicity $e(J(h))$ of $J(h)$ in $\mathbb{C}[x,y,w]$. Finally, by \cite[Cor.\ 1.5, pg.\ 320]{Teissier2}
$$\varepsilon (0)= \mu(\mathbb{V}(h)) + \mu(\mathbb{V}(h) \cap H)$$
where $\mu(\mathbb{V}(h))$ is the Milnor number of $\mathbb{V}(h)$ and $\mu(\mathbb{V}(h) \cap H)$ is the Milnor number of $\mathbb{V}(h) \cap H$ for a generic
hyperplane $H$ in $(\mathbb{C}^{3},0)$. 
}
\end{example}

\end{document}